\documentclass[reqno,centertags,11pt]{amsart}

\usepackage{amssymb,amsmath,amsfonts,amssymb}
\usepackage{hyperref}
\usepackage{enumerate}
\usepackage{color}
-\marginparwidth \oddsidemargin -9mm \evensidemargin \oddsidemargin
\usepackage{latexsym}
\advance\hoffset by 5mm

\def\@abssec#1{\vspace{.05in}\footnotesize \parindent .2in
{\bf #1. }\ignorespaces}

\newtheorem{theorem}{Theorem}[section]

\newtheorem{lemma}[theorem]{Lemma}
\newtheorem{proposition}[theorem]{Proposition}

\newtheorem{definition}[theorem]{Definition}
\newtheorem{remark}[theorem]{Remark}

\newcounter{hypo}

\DeclareMathOperator{\divg}{div}

\newcommand{\R}{\ensuremath{\mathbb{R}}}
\newcommand{\RR}{\ensuremath{\mathcal{R}}}
\newcommand{\BB}{\ensuremath{\widetilde{\mathcal{B}}}}
\newcommand{\bb}{\ensuremath{\mathcal{B}}}

\newcommand{\Z}{\ensuremath{\mathbb{Z}}}
\newcommand{\N}{\ensuremath{\mathbb{N}}}

\newcommand{\Id}{\ensuremath{\mathrm{Id}}}

\newcommand{\beqs}{\begin{equation*}}
\newcommand{\eeqs}{\end{equation*}}

\renewcommand{\div}{{\rm div}\,}

\newcommand{\tend}{\bibliographystyle{plain}\bibliography{ccrituniq}\end{document}}
\newcommand{\dd}{\mathrm{d}}

\allowdisplaybreaks \numberwithin{equation}{section}

\begin{document}

\title[Temperature fronts for 2D Viscous Boussinesq system]{Global regularity of non-diffusive temperature fronts for the 2D viscous Boussinesq system}
\author{Dongho Chae}
\address{Department of Mathematics, Chung-Ang University, Dongjak-gu, Heukseok-ro 84, Seoul 06974, Republic of Korea}
\email{dchae@cau.ac.kr}
\author{Qianyun Miao}
\address{School of Mathematics and Statistics, Beijing Institute of Technology, Beijing 100081, P. R. China}
\email{qianyunm@bit.edu.cn}
\author{Liutang Xue}
\address{Laboratory of Mathematics and Complex Systems (MOE), School of Mathematical Sciences, Beijing Normal University, Beijing 100875, P.R. China}
\email{xuelt@bnu.edu.cn}
\date{\today}
\maketitle

\begin{abstract}
  In this paper we address the temperature patch problem of the 2D viscous Boussinesq system without heat diffusion term.
The temperature satisfies the transport equation and the initial data of temperature is given in the form of non-constant patch,
usually called the temperature front initial data.
Introducing  a  good unknown and applying the method of striated estimates,
we prove that our partially viscous Boussinesq system admits a unique global regular solution and the initial $C^{k,\gamma}$ and $W^{2,\infty}$ regularity of the temperature front boundary with $k\in \Z^+ = \{1,2,\cdots\}$ and $\gamma\in (0,1)$
will be preserved for all the time. In particular, this naturally extends the previous work by Danchin $\&$ Zhang (2017)
and Gancedo $\&$ Garc\'ia-Ju\'arez (2017). In the proof of the persistence result of higher boundary regularity,
we introduce the striated type Besov space $\bb^{s,\ell}_{p,r,W}(\R^d)$
and establish a series of refined striated estimates in such a function space, which may have its own interest.
\end{abstract}

\section{Introduction}
We consider the 2D Boussinesq system without heat diffusion
\begin{equation}\label{BoussEq2D}
\begin{split}
\begin{cases}
  \partial_t \theta + u\cdot \nabla \theta = 0, \\
  \partial_t u + u\cdot\nabla u - \nu \Delta u + \nabla p = \theta e_2, \\
  \mathrm{div}\,u=0, \\
  (\theta, u)|_{t=0}(x) = (\theta_0,u_0)(x),
\end{cases}
\end{split}
\end{equation}
where $(x,t)\in \R^2\times \R^+$, $e_2=(0,1)$, $\nu\geq 0$ is the kinematic viscosity, 
$u=(u_1,u_2)$ is the velocity vector field, while
the scalars $\theta$, $p$ denote the temperature and the pressure of the fluid, respectively.
Boussinesq system is widely used to model the natural convection phenomena in the ocean and atmospheric dynamics \cite{Maj03,Ped87},
and it also plays an important role in studying the Rayleigh-B\'enard problem \cite{ConD99}.
It arises from the density-dependent fluid equations by applying the so-called Boussinesq approximation which neglects the density dependence in all the terms but the buoyancy force due to the gravity.
One can refer to \cite{FN09} for a rigorous justification of the approximation from the complete Navier-Stokes-Fourier system. 

From the mathematical viewpoint, Boussinesq systems contain the incompressible Navier-Stokes and Euler equations as special cases,
and the 2D inviscid Boussinesq system is essentially identical to the 3D axisymmetric swirling Euler equations away from the axis \cite{MB02}.
Furthermore, the important vortex-stretching mechanism is present in both 2D and 3D Boussinesq systems. As pointed out in \cite{Mof01,Yud03},
the global well-posedness issue of (inviscid) Boussinesq systems is a major open problem in the theory of mathematical fluid dynamics.

Due to the physical relevance and mathematical importance, Boussinesq systems recently attracted a lot of attention and were intensely studied.
For the 2D viscous Boussinesq system \eqref{BoussEq2D} (i.e. $\nu>0$), Chae \cite{Cha06} and Hou, Li \cite{HouL05} independently proved the global well-posedness for regular initial data.
Later, 
Abidi and Hmidi \cite{AH07} considered less regular initial data $\theta_0\in B^0_{2,1}$ and $u_0 \in L^2\cap B^{-1}_{\infty,1}$,
and showed the global existence and uniqueness. Hmidi and Keraani \cite{HK07} established the global existence of weak solution to system \eqref{BoussEq2D}
with initial data $\theta_0\in L^2$, $u_0\in H^s$, $s\in [0,2)$, and furthermore Danchin and Paicu \cite{DanP09} resolved the uniqueness issue by using the para-differential calculus.
Hu, Kukavica and Ziane \cite{HKZ15} also obtained the persistence of regularity  result in various Sobolev spaces.

While for the 2D inviscid Boussinesq system (i.e. $\nu=0$ in \eqref{BoussEq2D}), so far the global regularity issue still remains an outstanding unsolved problem.
Numerical simulations once suggested global regularity for this system in a periodic domain \cite{ES94},
but recent numerical studies \cite{LH14b} proposed an important potential scenario of finite-time blowup in the bounded domain with smooth boundary.
Motivated by this singularity scenario, several 1D Boussinesq models \cite{CKY15,CHKLSY} and the modified 2D Boussinesq system \cite{KT18} admitting incompressibility were developed,
and the finite-time blowup of smooth solutions for these models has been rigorously justified. Recently, concerning the original 2D inviscid Boussinesq system on a spatial domain with an acute corner,
\cite{ElgJ20} constructed the finite-time blowup in Lipschitz norm for some locally well-posed solution with finite energy.
One can also see \cite{CCL19,ElgW16} for the interesting global stability results for the 2D inviscid Boussinesq system.

In the modeling the large scale atmospheric and oceanic flows, the viscosity and diffusion coefficients of the Boussinesq system are usually different in the horizontal and vertical directions.
For these scenarios, there are some global well-posedness results for the Boussinesq system with various anisotropic and partial dissipation
(one can refer to \cite{CaoW13,LLT13,LiT16} and the references therein).

Recently, there are also much attention on the so-called \emph{Boussinesq temperature patch problem} for the viscous Boussinesq system \eqref{BoussEq2D},
which is a free boundary problem of the system \eqref{BoussEq2D} with singular initial data $\theta_0 = 1_{D_0}$, i.e. the characteristic function of a simply connected bounded domain $D_0$.
In view of equation $\eqref{BoussEq2D}_1$ and the particle trajectory $X_t(x)$ given by
\begin{equation}\label{eq:flow0}
    \frac{\partial X_t(x)}{\partial t} = u ( X_t(x),t),\quad X_t(x)|_{t=0}=x,
\end{equation}
one can see that the patch structure of the temperature will be preserved so that $\theta(x,t) = 1_{D(t)}$ with $D(t)= X_t(D_0)$.
Thus a natural problem arises: whether the initial regularity of the patch boundary persists globally in time, e.g.,
\begin{align}\label{Ques-patch}
  \textrm{suppose $\partial D_0\in C^{k,\gamma}$, $k\in \Z^+$, $\gamma\in (0,1)$, whether or not $\partial D(t)\in C^{k,\gamma}$ for all time?}
\end{align}
In the above, we denote $\partial D(t)\in C^{k,\gamma}$ provided that there is a parametrization of the patch boundary $\partial D(t) = \big\{z(\alpha ,t)\in \R^2, \alpha\in \mathbb{S}^1=[0,1] \big\}$ with $z(\cdot, t)\in C^{k,\gamma}$.

Such  kind of patch problems were initiated in 1980s by studying the famous vorticity patch problem of the 2D Euler equations.
Although numerical simulations once suggested the possibility of finite-time singularity for this problem,
the global persistence result of initial $C^{k,\gamma}$-boundary regularity was proved by Chemin \cite{Chem88,Chem91}
using the paradifferential calculus and the striated regularity method.
A different proof of the same result was obtained by Bertozzi and Constantin \cite{BerC} applying a geometric lemma and the harmonic analysis techniques.
For the related density patch problem of the nonhomogeneous Navier-Stokes system, one can also see \cite{DanZ17b,GGJ18,LZ16,LZ19} for the global regularity persistence results.

Concerning the temperature patch problem, Danchin and Zhang \cite{DanZ17} firstly proved the global well-posedness of the viscous Boussinesq system \eqref{BoussEq2D} with rough initial data
$\theta_0\in B^{2/q-1}_{q,1}$, $q\in (1,2)$,
which admits the $C^{1,\gamma}$-temperature patch,
and then by using the striated estimates method, they showed that the $C^{1,\gamma}$-regularity of patch boundary is globally preserved in the 2D case as well as at the 3D case under an additional smallness condition.
Later, Gancedo and Gar\'cia-Ju\'arez \cite{GGJ17} in the 2D case gave a different proof of the $C^{1,\gamma}$-regularity persistence result,
and furthermore proved the global persistence of $W^{2,\infty}$- and $C^{2,\gamma}$-regularity of temperature patch boundary.
Meanwhile, the curvature of the temperature patch remained bounded for all the time, by taking advantage of new cancellations in the time-dependent Calder\'on-Zygmund operators.
The authors in \cite{GGJ20} extended the same global regularity results to the 3D Boussinesq temperature patch problem under a scaling-invariant smallness assumption of initial data,
and they also treated the temperature front initial data which is the temperature patch of non-constant values.

For more general patch-type solutions and related contour interface dynamics, they can be used to model many important physical phenomena arising from water waves, porous media or frontogenesis and so on,
and have been intensively studied in the last decades, and one can see \cite{CCFG13,CCFGG13,CCFGG19,CCFG16,CFMR,CouS14,CouS19,FIL16,GanS14,KRYZ16} and references therein for the recent progress.
In particular, the finite-time singularities were rigorously proved for Muskat system \cite{CCFG13,CCFG16}, free-surface Euler equations \cite{CCFGG13,CouS14}, free-surface Navier-Stokes equations \cite{CCFGG19,CouS19} and a modified SQG equation \cite{KRYZ16,GP21}.
\vskip0.1cm

In this paper we focus on the problem \eqref{Ques-patch} of the 2D viscous Boussinesq sytem \eqref{BoussEq2D} with initial temperature patch of non-constant values.
This setting describes the evolution of the temperature front governed by the fluid flow,
which is an important physical scenario in geophysics \cite{Gil82,Maj03}.
Our main purpose is to show the $C^{k,\gamma}$-regularity propagation result of the temperature front boundary with any $k\in \Z^+=\{1,2,\cdots\}$,
which also naturally generalizes the results of \cite{DanZ17,GGJ17}.
\vskip0.1cm

Assume $\theta_0(x)=\bar{\theta}_0(x) 1_{D_0}(x)$ to be an initial temperature front, where $D_0\subset \R^2$ is a bounded simply connected domain with boundary $\partial D_0 \in C^{k,\gamma}(\R^2)$, $k\in \Z^+$, $\gamma\in (0,1)$.
We consider the level-set characterization of the domain $D_0$: there exists a function $\varphi_0\in C^{k,\gamma}(\R^2)$ such that
\begin{equation}\label{patch-ls-exp}
  \partial D_0 = \{x\in \R^2: \varphi_0(x) = 0\}, \quad D_0 = \{x\in \R^2: \varphi_0(x)>0\},
  \quad \textrm{and}\quad \textrm{$\nabla \varphi_0\neq 0$ on $\partial D_0$}.
\end{equation}
Then the boundary $\partial D_0$ can be parameterized as
\begin{equation}\label{patch-para-exp}
  z_0: \mathbb{S}^1 \mapsto \partial D_0\quad \textrm{with}\;\;  \partial_\alpha z_0(\alpha) = \nabla^\perp \varphi_0(z_0(\alpha)) =: W_0(z_0(\alpha)),
\end{equation}
with $\nabla^\perp = (-\partial_2,\partial_1)^T$.
In the sequel we also set the viscosity $\nu=1$ for brevity.

Our main results read as follows.

\begin{theorem}\label{thm:Bouss-patch1}
  Let $D_0 \subset \R^2$ be a bounded simply connected domain with boundary $\partial D_0 \in C^{1,\gamma}(\R^2)$, and $\theta_0 (x) = \bar{\theta}_0(x) 1_{D_0}(x)$
be the temperature front initial data with $\bar\theta_0\in L^\infty(\overline{D_0})$.
Let $u_0\in H^1(\R^2)$ be a divergence-free vector field. Then, there exists a unique global solution $(\theta,u)$
to the 2D Boussinesq system \eqref{BoussEq2D} such that for any $T>0$,
\begin{equation}
  u \in C(0,T; H^1(\R^2)) \cap L^2(0,T; H^2(\R^2)) \cap L^1(0,T; C^{1,\gamma}(\R^2)),\quad \forall \gamma \in (0,1),
\end{equation}
and
\begin{equation}
  \theta(x,t) = \bar{\theta}_0(X^{-1}_t(x)) 1_{D(t)}(x), \quad \textrm{with}\;\; \partial D(t) = X_t(\partial D_0) \in L^\infty(0,T;C^{1,\gamma}(\R^2)),
\end{equation}
where $X_t$ is the particle-trajectory generated by the velocity $u$ (see \eqref{eq:flow0} above)
and $X^{-1}_t$ is its inverse.

Moreover, the boundary of the temperature front has the following regularity persistence properties.
\begin{enumerate}[(1)]
\item If additionally, $\partial D_0 \in W^{2,\infty}(\R^2)$, $\bar{\theta}_0\in C^\mu(\overline{D_0})$, $\mu\in (0,1)$, and $u_0\in H^1\cap W^{1,p}(\R^2)$ with some $p>2$, we get
\begin{equation}\label{eq:bdr-W2inf}
  \partial D(t) \in L^\infty(0,T; W^{2,\infty}(\R^2)),
\end{equation}
and
$u\in  L^\rho(0,T; W^{2,\infty}(\R^2))$ with $1\leq \rho < \frac{2p}{p+2}$.
\item If additionally, $\partial D_0 \in C^{k,\gamma}(\R^2)$, $k\in \N \cap [2,\infty)$, $\gamma\in (0,1)$, $\bar{\theta}_0\in C^{k-2,\gamma}(\overline{D_0})$, and $u_0\in H^1\cap W^{1,p}(\R^2)$,
$(\partial_{W_0}u_0,\cdots, \partial_{W_0}^{k-1} u_0) \in W^{1,p}(\R^2)$ with some $p>2$, we obtain
\begin{equation}\label{eq:bdr-Ckgam}
  \partial D(t) \in L^\infty(0,T; C^{k,\gamma}(\R^2)).
\end{equation}
\end{enumerate}
\end{theorem}

In the above $\partial_{W_0} u_0 := W_0\cdot \nabla u_0 = \divg\,(W_0\, u_0)$ denotes the directional derivative of $u_0$ along the vector field $W_0$.

\begin{remark}
  With slight modification, one can also deal with the more general temperature front initial data $\theta_0(x) = \bar{\theta}_1(x) 1_{D_0}(x) + \bar{\theta}_2(x) 1_{D_0^c}(x)$,
where $\bar\theta_1$ and $\bar\theta_2$ are the functions defined on $\overline{D_0}$ and $\overline{D_0^c}$ respectively.
\end{remark}

In the proof of the $C^{1,\gamma}$-, $W^{2,\infty}$- and $C^{2,\gamma}$-regularity persistence result of the temperature front boundary,
noting that the domain $D(t)= X_t(D_0)$ can be determined by the level-set function $\varphi(x,t)=\varphi_0(X_t^{-1}(x))$ which solves
\begin{equation}\label{varphi-eq}
  \partial_t \varphi + u\cdot\nabla \varphi =0,\quad \varphi(0,x) = \varphi_0(x),
\end{equation}
one needs only to prove the uniform boundedness of $\varphi(t)$ in the norms of $C^{1,\gamma}$, $W^{2,\infty}$ and $C^{2,\gamma}$, respectively.
Compared with \cite{DanZ17,GGJ17}, a new ingredient is the introduction of a good unknown\footnote{Such a technique may be called as \emph{Alinhac's good unknown}. One can see \cite{HKR10,HKR11,CW,WX12,KX20} and reference therein
for the application of this method to the 2D Boussinesq system with various partial fractional dissipation.}
$$
\Gamma := \omega - \RR_{-1}\theta,
$$
with $\omega := \partial_1 u_2 -\partial_2 u_1$ the vorticity of the fluid and $\RR_{-1} := \partial_1 (-\Delta)^{-1}$.
The equation of $\Gamma$ reads as
$\partial_t \Gamma + u\cdot\nabla \Gamma -\Delta \Gamma = [\mathcal{R}_{-1},u\cdot\nabla]\theta$,
and the term $[\mathcal{R}_{-1},u\cdot\nabla]\theta$ can be better controlled
(e.g. see Lemma \ref{lem:comm-es}) than the corresponding term $\partial_1 \theta$ in the the vorticity equation (see \eqref{vort-eq} below).
The quantity $\Gamma$ usually has good regularity estimates stemming from the smoothing effect of heat equation, thus thanks to the relation
\begin{align}\label{rel:nabla-u}
  \nabla u = \nabla \nabla^\perp (-\Delta)^{-1} \omega = \nabla\nabla^\perp (-\Delta)^{-1} \Gamma + \nabla \nabla^\perp (-\Delta)^{-1}\RR_{-1}\theta,
\end{align}
the restriction of the regularity of $\nabla u$ mainly comes from the $\theta$-term.
Since $\theta$ belongs to $L^2\cap L^\infty$ uniformly in time and $\nabla \nabla^\perp (-\Delta)^{-1}\RR_{-1}$ is a pseudo-differential operator of $-1$-order,
we can directly prove that $\nabla \nabla^\perp (-\Delta)^{-1}\RR_{-1} \theta$ and $\nabla u$ belong to $L^1_T(C^\gamma)$ for every $\gamma\in (0,1)$,
which ensures the global uniform $C^{1,\gamma}$-boundedness of $\varphi(t)$.

In order to prove that $u$ belongs to $L^1_T(W^{2,\infty})$, which implies the uniform $W^{2,\infty}$-boundedness of $\varphi(t)$, we mainly need to show that $\nabla^2 \nabla^\perp (-\Delta)^{-1}\RR_{-1}\theta$
belongs to $L^\infty_T(L^\infty)$. The situation is quite analogous to that in the vorticity patch problem of 2D Euler equations, where one needs to control the $L^\infty$-norm of $\nabla \nabla^\perp (-\Delta)^{-1}\omega$ with $\omega$ of patch structure,
and by using the additional cancellation property of the singular integral operator with even kernel (see the geometric lemma in \cite{BerC}),
we can derive the desired uniform boundedness estimate.

To obtain the global uniform $C^{2,\gamma}$-estimate of $\varphi(t)$, we consider the quantity $W=\nabla^\perp \varphi$ (similarly as \cite{GGJ17}),
and by estimating the $C^\gamma$-norm of $\nabla W(t)$, it mainly needs to control the striated term $\partial_W \nabla u$ in $L^1_T(C^\gamma)$,
with $\partial_W := W\cdot \nabla $ the directional derivative.
In view of \eqref{rel:nabla-u} and the patch structure of $\theta$, we deal with the estimates of $\partial_W \Gamma$ and $\partial_W \theta$ respectively,
and through using the striated estimate \eqref{eq:str-es1}, we finally can bound $L^1_T(C^\gamma)$-norm of $\partial_W \nabla u$ in terms of $\| \nabla W(t)\|_{C^\gamma}$ and a suitable norm of $\partial_W \Gamma$,
so that the Gronwall inequality ensures the wanted global uniform estimate. We remark that the proof of the above results presented here is relatively simpler than that in the work \cite{GGJ17}.
\vskip0.15cm

In the proof of the propagation of even higher $C^{k,\gamma}$-boundary regularity, motivated by \cite{Chem91},
it indeed suffices to show the striated estimate $\partial_W^{k-1} W $ in the norm $L^\infty_T(C^\gamma)$ (see \eqref{eq:target} below).
The method of high-order striated estimates (or conormal estimates) initiated by Chemin \cite{Chem88,Chem91} plays an important role in the whole process.
We would like to emphasize, however, that there exists a crucial difference compared with the application to the vorticity patch problem of Euler equations well studied in \cite{Chem88,Chem91} (see also \cite{LZ16}).
The regularity of the vector field $W$ and its striated counterpart $\partial_W^\ell W$ in the vorticity patch problem are of $C^\gamma$-H\"older type with $0<\gamma <1$,
while in our situation they all belong to $C^{1,\gamma}$ uniformly in time.
As a consequence, it will yield a lot of substantial difference in the analysis. 
The foremost one can be seen from the estimation of the operator $R_q$ given by \eqref{def:Rq}: 
there is a factor $2^{q\varepsilon(\ell -|\mu|)}$ in \cite[Lemma A.2]{Chem91} or \cite[Eq. (7.3)]{LZ16},
while in our case such a factor vanishes in the corresponding inequality \eqref{eq:Tw-Rq} below\footnote{We mention that J.-Y. Chemin had already clarified this key difference in \cite[Pg. 446]{Chem88} at the special case $p=\infty$.}.
The factor $2^{q\varepsilon(\ell -|\mu|)}$ usually leads to the various striated estimates in \cite{Chem88,Chem91,LZ16} with essential $\varepsilon$-regularity loss,
but here the striated estimates have no regularity loss. 
In order to be able to develop such fine-scale striated estimates, we introduce the striated type Besov space $\bb^{s,\ell}_{p,r,W}(\R^d)$ (see Definition \ref{def:gBS} below).
By adopting this function space and using the tedious paradifferential calculus, we establish a series of refined striated estimates in Lemmas \ref{lem:prod-es2} and \ref{lem:prod-es}.
These striated estimates are natural generalization of some classical product and commutator estimates in the usual Besov space with negative regularity index,
which might be interesting in its own.

In order to show that $\partial_W^{k-1}W$ belongs to $L^\infty_T(C^\gamma)$, or more precisely, to build the stronger estimate \eqref{eq:Targ-k}, we use the induction method.
Suppose that we already have good control on the quantities $W,\nabla u$ and $\Gamma$ in the appropriate $\bb^{s,\ell}_{\infty,r,W}$-norms as in \eqref{assum-el}
with $\ell\in \{1,\cdots,k-2\}$, 
we intend to show the corresponding estimates with $\ell+1$.
The procedure is as in the proof of $C^{2,\gamma}$-persistence result, and the above refined striated estimates will be intensively used.
In order to get the $L^\infty_T(\bb^{\gamma+1,\ell}_{\infty,W})$-estimate of $W$, from the equation of $\partial_W^\ell \nabla^2 W$ and the striated estimate \eqref{eq:prod-es2-1},
it needs to consider the quantity $\nabla^2 u$ in $L^1_T(\bb^{\gamma-1,\ell+1}_{\infty, W})$.
In light of \eqref{rel:nabla-u}, we treat the $\Gamma$-term and $\theta$-term separately: by applying the smoothing estimate of transport-diffusion equation and the induction assumption,
we obtain good striated regularity estimate of $\Gamma$ in terms of $\Gamma$ itself and $W$ in suitable $\bb^{s,\ell}_{\infty,r,W}$-norms,
which can be used to control the term $ \nabla^2\nabla^\perp (-\Delta)^{-1}\Gamma$;
while by using the patch structure of $\theta$ and striated estimate \eqref{eq:prod-es6}, we can bound the $L^1_T(\bb^{\gamma-1,\ell+1}_{\infty,W})$-norm of $\nabla^2 \nabla^\perp (-\Delta)^{-1}\RR_{-1}\theta$.
Gathering all these estimates and using Gronwall's inequality yield the uniform estimates with $\ell+1$, so that the induction scheme can be continued to fulfill the final target.
\vskip0.15cm

The paper is organized as follows. In Section \ref{sec:prel}, we introduce the striated type Besov space $\bb^{s,\ell}_{p,r,W}(\R^d)$
and establish several related estimates,
and also compile some auxiliary lemmas. We prove the $C^{1,\gamma}$-, $W^{2,\infty}$- and $C^{2,\gamma}$-regularity persistence results in the section \ref{sec:C1-2gam},
and then in the section \ref{sec:Ck-gam} we deal with the $C^{k,\gamma}$-regularity persistence result with $k\in \N\cap [3,\infty)$.
Sections \ref{sec:prod-es} and \ref{sec:str-es-Lem} are both concerned with the key striated estimates, and we present the detailed proof of Lemmas \ref{lem:prod-es2}, \ref{lem:prod-es}
and Lemmas \ref{lem:Rq} - \ref{lem:Tw-es-k1}.
Finally, we prove the auxiliary Lemma \ref{lem:prod-es0} in the appendix.


\section{Preliminaries and auxiliary lemmas}\label{sec:prel}

\subsection{Besov type spaces and related estimates}

One can choose two nonnegative radial functions $\chi, \varphi\in C_c^\infty(\mathbb{R}^d)$ be
supported respectively in the ball $\{\xi\in \mathbb{R}^d:|\xi|\leq 4/3 \}$ and the annulus $\{\xi\in
\mathbb{R}^d: 3/4\leq |\xi|\leq  8/3 \}$ such that (see \cite{BCD11})
\begin{align}\label{eq:POU}
  \chi(\xi)+\sum\limits_{j\ge0}\varphi(2^{-j}\xi)=1,\quad \textrm{for every}\;\;\xi\in\R^d.
\end{align}
For every tempered distribution $f$, the dyadic block operators $\Delta_j$ and $ S_j$ are defined by
\begin{align}\label{exp:Del-Sj}
  \Delta_{-1} f = \chi(D)f = h* f,\quad \Delta_j f = \varphi(2^{-j}D)f=2^{jd} h(2^j\cdot)*f, \quad  \forall j\in \N, \nonumber \\
  S_j f= \chi(2^{-j}D)f = \sum_{-1\leq l\leq j-1} \Delta_l f = 2^{jd}  h'(2^j\cdot)*f,\quad \forall j\in \N,
\end{align}
with $h=\mathcal{F}^{-1}{\varphi}$, $h'=\mathcal{F}^{-1}{\chi}$ and $\mathcal{F}^{-1}$ the Fourier inverse transform.

For every $f,g\in\mathcal{S}'(\R^d)$, we have the following Bony's decomposition
\begin{equation}
  f\,g = T_f g + T_g f + R(f,g),
\end{equation}
with 
\begin{equation}
  T_f g:= \sum_{q\in \N} S_{q-1}f \Delta_q g,\quad R(f,g)=\sum_{q\geq -1}\Delta_q f \widetilde{\Delta}_q g,\quad \widetilde{\Delta}_q := \Delta_{q-1} + \Delta_q + \Delta_{q+1}.
\end{equation}

Now we introduce the Besov space $B^s_{p,r}(\R^d)$ and its generalized type suited to the striated estimates.
\begin{definition}\label{def:gBS}
Let $s\in \R$, $(p,r)\in [1,\infty]^2$.
Denote by $B^s_{p,r}= B^s_{p,r}(\R^d)$ the space of tempered distributions $f\in \mathcal{S}'(\R^d)$ such that
\begin{align*}
  \|f\|_{B^s_{p,r}(\R^d)} := \big\| \big\{2^{qs}  \|\Delta_q f\|_{L^p(\R^d)}\big\}_{q\geq -1}  \big\|_{\ell^r}  < \infty.
\end{align*}
For every $\ell\in \N$, $N\in \mathbb{Z}^+$ and a set of regular vector fields $\mathcal{W}= \{W_i\}_{1\leq i \leq N}$
with $W_i:\R^d\rightarrow \R^d$,
denote by $\bb^{s,\ell}_{p,r,\mathcal{W}}=\bb^{s,\ell}_{p,r,\mathcal{W}}(\R^d)$ the space of tempered distributions
$f\in B^s_{p,r}(\R^d)$ such that
\begin{align}\label{norm:BBsln2}
  \|f\|_{\bb^{s,\ell}_{p,r,\mathcal{W}}} := \sum_{\lambda=0}^\ell \|\partial_{\mathcal{W}}^\lambda f\|_{B^s_{p,r}}
  = \sum_{\lambda=0}^\ell \sum_{\lambda_i\in \N;\lambda_1 + \cdots + \lambda_N=\lambda}
  \|\partial_{W_1}^{\lambda_1}\cdots \partial_{W_N}^{\lambda_N} f\|_{B^s_{p,r}}  <\infty ;
\end{align}
we also denote by $\BB^{s,\ell}_{p,r,\mathcal{W}}= \BB^{s,\ell}_{p,r,\mathcal{W}}(\R^d)$ the set of tempered distributions $f\in B^s_{p,r}(\R^d)$ such that
\begin{equation}\label{norm:BBsln}
\begin{split}
  \|f\|_{\BB^{s,\ell}_{p,r, \mathcal{W}}}  := \sum_{\lambda=0}^\ell \|(T_{\mathcal{W}\cdot\nabla})^\lambda f\|_{B^s_{p,r}}
  = \sum_{\lambda=0}^\ell \sum_{\lambda_1 +\cdots + \lambda_N = \lambda}
  \|(T_{W_1\cdots \nabla})^{\lambda_1}\cdots (T_{W_N\cdots\nabla})^{\lambda_N} f\|_{B^s_{p,r}} < \infty .
\end{split}
\end{equation}
In particular, when $p=\infty$, we always use the following abbreviations
\begin{equation}\label{eq:abbr}
\begin{split}
  \bb^{s,\ell}_{r,\mathcal{W}} : = \bb^{s,\ell}_{\infty,r,\mathcal{W}} , \quad &
  \BB^{s,\ell}_{r,\mathcal{W}} : = \BB^{s,\ell}_{\infty,r,\mathcal{W}}, \\
  \bb^{s,\ell}_{\mathcal{W}} := \bb^{s,\ell}_{1,\mathcal{W}} = \bb^{s,\ell}_{\infty,1,\mathcal{W}},
  \quad & \BB^{s,\ell}_{\mathcal{W}} := \BB^{s,\ell}_{1,\mathcal{W}} = \BB^{s,\ell}_{\infty,1,\mathcal{W}}.
\end{split}
\end{equation}
Besides, if $\mathcal{W}$ contains only one regular vector field $W$, i.e. $\mathcal{W}= \{W\}$,
we also denote
\begin{align}
  \bb^{s,\ell}_{p,r,W}(\R^d) & : = \Big\{f\in B^s_{p,r}(\R^d)\, \big|\, \|f\|_{\bb^{s,\ell}_{p,r,W}}
  := \sum_{\lambda=0}^\ell \|\partial_W^\lambda f\|_{B^s_{p,r}} < \infty \Big\}, \label{norm:BBsln2-2}\\
  \BB^{s,\ell}_{p,r,W}(\R^d) & : = \Big\{f\in B^s_{p,r}(\R^d)\, \big|\, \|f\|_{\BB^{s,\ell}_{p,r,W}}
  := \sum_{\lambda=0}^\ell \|(T_{W\cdot\nabla})^\lambda f\|_{B^s_{p,r}} < \infty \Big\}, \label{norm:BBsln-2}
\end{align}
and similar abbreviations \eqref{eq:abbr} hold with $W$ in place of $\mathcal{W}$.
\end{definition}

In the above, the notation $\partial_\mathcal{W} = \mathcal{W}\cdot\nabla $ denotes as the vector-valued operators
$\{ W_i\cdot\nabla\}_{1\leq i\leq N}$, and $\partial_\mathcal{W}^\lambda = (\mathcal{W}\cdot\nabla)^\lambda =\partial_\mathcal{W} \otimes \cdots \otimes\partial_\mathcal{W}$ for every $\lambda\in \N$.

Some basic properties of the space $\bb^{s,\ell}_{p,r,\mathcal{W}}$ are presented as follows.
\begin{lemma}\label{lem:Bes-prop}
  Let $s,\widetilde{s}\in \R$, $\ell,\widetilde{\ell}\in \N$, $r, \widetilde{r}\in [1,\infty]$, $p\in [1,\infty]$
and $\mathcal{W}= \{W_i\}_{1\leq i\leq N}$ be composed of regular vector fields $W_i:\R^d\rightarrow \R^d$.
The function space $\bb^{s,\ell}_{p,r,\mathcal{W}}$ satisfies that
\begin{equation}\label{eq:Bes-prop1}
\begin{split}
  \bb^{s,\ell}_{p,r,\mathcal{W}} \subset \bb^{\widetilde{s},\ell}_{p,r,\mathcal{W}},\;\, \textrm{for}\;\, s\geq \widetilde{s},\quad \bb^{s,\ell}_{p,r,\mathcal{W}} \subset \bb^{s,\widetilde{\ell}}_{p,r,\mathcal{W}},\;\, \textrm{for}\;\, \ell \geq \widetilde{\ell}, \quad
  \bb^{s,\ell}_{p,r,\mathcal{W}} \supset \bb^{s,\ell}_{p,\widetilde{r},\mathcal{W}},\;\, \textrm{for}\;\, r \geq \widetilde{r},
\end{split}
\end{equation}
\begin{equation}\label{eq:Bes-prop3}
  \|f\|_{\bb^{s,\ell+1}_{p,r,\mathcal{W}}} = \|\partial_\mathcal{W}^{\ell+1} f\|_{B^s_{p,r}} + \|f\|_{\bb^{s,\ell}_{p,r,\mathcal{W}}},
  \quad \textrm{and} \quad  \|f\|_{\bb^{s,\ell+1}_{p,r,\mathcal{W}}} = \|\partial_\mathcal{W} f\|_{\bb^{s,\ell}_{p,r,\mathcal{W}}} + \|f\|_{B^s_{p,r}}.
\end{equation}
\end{lemma}

We first present a useful product estimate in $B^s_{p,r}(\R^d)$ (one can see the appendix for the proof).
\begin{lemma}\label{lem:prod-es0}
  Let $u$ be a smooth divergence-free vector field of $\R^d$, and let $\phi:\R^d\rightarrow \R$ be a smooth function.
Then we have that for every $p\in [1,\infty]$ and $\epsilon \in (0,1)$,
\begin{equation}\label{eq:prod-es}
  \|u\cdot \nabla \phi\|_{B^{-\epsilon}_{p,r}} \leq C \min\Big\{\|u\|_{B^{-\epsilon}_{p,r}} \|\nabla\phi\|_{L^\infty},
  \|u\|_{L^\infty} \|\nabla\phi\|_{B^{-\epsilon}_{p,r}} \Big\}.
\end{equation}
\end{lemma}

The following striated estimates in the framework of $\bb^{s,\ell}_{p,r,\mathcal{W}}$ plays an important role in the main proof,
and we place the detailed proof in the subsection \ref{subsec:str-es1}.
\begin{lemma}\label{lem:prod-es2}
  Let $k\in \N$, $\sigma\in (0,1)$, $N\in \mathbb{Z}^+$ and $\mathcal{W}= \{W_i\}_{1\leq i\leq N}$ be a set of regular divergence-free vector fields
$W_i:\R^d\rightarrow \R^d$ satisfying that
\begin{align}\label{norm:W}
  \|\mathcal{W}\|_{\bb^{1+\sigma,k-1}_{\infty,\mathcal{W}}}
  := \sum_{\lambda=0}^{k-1} \|\partial_\mathcal{W}^\lambda \mathcal{W}\|_{B^{1+\sigma}_{\infty,\infty}}
  = \sum_{\lambda=0}^{k-1} \sum_{\lambda_1 +\cdots + \lambda_N=\lambda} \|\partial_{W_1}^{\lambda_1}\cdots \partial_{W_N}^{\lambda_N} \mathcal{W}\|_{B^{1+\sigma }_{\infty,\infty}}
  < \infty.
\end{align}
Let $m(D)$ be a zero-order pseudo-differential operator with $m(\xi)\in C^\infty(\R^d\setminus\{0\})$.
Assume that $u$ is a smooth divergence-free vector field of $\R^d$, and $\phi:\R^d\rightarrow \R$ is a smooth function.
Then for every $\epsilon \in (0,1)$ and $(p,r)\in [1,\infty]^2$, there exists a constant $C>0$ depending only on $d,k,\epsilon$ and $\|\mathcal{W}\|_{\bb^{1+\sigma,k-1}_{\infty,\mathcal{W}}}$ (when $k=0$ this norm plays no role) such that the following statements hold.
\begin{enumerate}[(1)]
\item
We have
\begin{equation}\label{eq:prod-es2-1}
  \|u\cdot \nabla \phi \|_{\bb^{-\epsilon,k}_{p,r,\mathcal{W}} } \leq C \min\bigg\{\sum_{\mu=0}^k  \|u\|_{\bb^{0,\mu}_\mathcal{W}} \|\nabla \phi\|_{\bb^{-\epsilon,k-\mu}_{p,r,\mathcal{W}}},
  \sum_{\mu=0}^k  \|u\|_{\bb^{-\epsilon,\mu}_{p,r,\mathcal{W}}} \|\nabla \phi\|_{\bb^{0,k-\mu}_\mathcal{W}} \bigg\}.
\end{equation}
\item We have
\begin{align}\label{eq:prod-es6}
  \|m(D)\phi\|_{\bb^{-\epsilon,k+1}_{p,r,\mathcal{W}}} \leq C \|\phi\|_{\bb^{-\epsilon,k+1}_{p,r,\mathcal{W}}}
  + C \Big( 1+ \| \mathcal{W}\|_{\bb^{1,k}_\mathcal{W}} \Big) \Big(\|\phi\|_{\bb^{-\epsilon,k}_{p,r,\mathcal{W}}} + \|\Delta_{-1}m(D)\phi\|_{L^p} \Big).
\end{align}
\item We have 
\begin{align}\label{eq:prod-es7}
  \|[m(D),u\cdot\nabla]\phi\|_{\bb^{-\epsilon,k}_{p,r,\mathcal{W}}} \leq C \Big(\|\nabla u\|_{\bb^{0,k}_\mathcal{W}}
  + \|u\|_{L^\infty} \Big) \|\phi\|_{\bb^{-\epsilon,k}_{p,r,\mathcal{W}}}.
\end{align}
\end{enumerate}
If $\mathcal{W}$ contains only one divergence-free vector field $W$, the inequalities \eqref{eq:prod-es2-1}--\eqref{eq:prod-es7} hold
with $W$ in place of $\mathcal{W}$.
\end{lemma}


In particular, for the special case $k=0,1$, the dependence of lower-order term $\|\mathcal{W}\|_{\bb^{1+\sigma,k-1}_{\infty,\mathcal{W}}}$ in the constant $C$ of Lemma \ref{lem:prod-es2} can be calculated explicitly,
and the corresponding striated estimates are stated as follows (whose proof is placed in the subsection \ref{subsec:str-es2}).
\begin{lemma}\label{lem:prod-es}
  Let $u$ be a smooth divergence-free vector field of $\R^d$ and $\mathcal{W}=\{W_i\}_{1\leq i \leq N}$ ($N\in \mathbb{Z}^+$)
be a set of smooth divergence-free vector fields. Let $\phi:\R^d\rightarrow \R$ be a smooth function.
Let $m(D)$ be a zero-order pseudo-differential operator with $m(\xi)\in C^\infty(\R^d\setminus\{0\})$.
Then for every $\epsilon\in (0,1)$ and $(p,r)\in [1,\infty]^2$, there exists a constant $C>0$ depending only on $d,\epsilon$ such that the following statements hold true.
\begin{enumerate}[(1)]
\item
We have
\begin{equation}\label{eq:prod-es3}
\begin{split}
  \|\partial_\mathcal{W} (u\cdot \nabla \phi)\|_{ B^{-\epsilon}_{p,r}}
  + \|T_{\mathcal{W}\cdot\nabla}(u\cdot\nabla \phi)\|_{B^{-\epsilon}_{p,r}} \leq C \min\{A_1,A_2,A_3\},
\end{split}
\end{equation}
with
\begin{align}
  A_1 & := \|u \|_{ B^{-\epsilon}_{p,r}} \|\partial_\mathcal{W}\nabla\phi\|_{ B^0_{\infty,1}}
  + \Big(\| \partial_\mathcal{W} u\|_{ B^{-\epsilon}_{p,r}}
  + \| \mathcal{W}\|_{ B^1_{\infty,1}} \|u\|_{ B^{-\epsilon}_{p,r}} \Big) \|\nabla \phi\|_{ B^0_{\infty,1}}, \label{A1} \\
  A_2 & := \|u\|_{ B^0_{\infty,1}} \|\partial_\mathcal{W}\nabla\phi\|_{ B^{-\epsilon}_{p,r}} + \Big(\|\partial_\mathcal{W} u\|_{ B^0_{\infty,1}}
  + \| \mathcal{W}\|_{ B^1_{\infty,1}} \|u\|_{ B^0_{\infty,1}} \Big) \|\nabla \phi \|_{ B^{-\epsilon}_{p,r}}, \label{A2} \\
  A_3 & := \|u\|_{ B^0_{\infty,1}} \Big(\|\partial_\mathcal{W}\nabla\phi\|_{ B^{-\epsilon}_{p,r}} + \|\mathcal{W}\|_{B^1_{\infty,1}} \|\nabla\phi\|_{ B^{-\epsilon}_{p,r}} \Big)  \nonumber \\
  &\quad + \Big(\|\partial_\mathcal{W} u\|_{B^{-\epsilon}_{p,r}} +  \|\mathcal{W}\|_{B^1_{\infty,1}} \|u\|_{B^{-\epsilon}_{p,r}}\Big)\|\nabla \phi\|_{B^0_{\infty,1}}. \label{A3}
\end{align}
\item
We have
\begin{equation}\label{eq:str-es1}
  \|\partial_\mathcal{W} (m(D) \phi)\|_{B^{-\epsilon}_{p,r}} \leq C \|\partial_\mathcal{W} \phi\|_{B^{-\epsilon}_{p,r}}
  + C \|\mathcal{W}\|_{B^1_{\infty,1}}
  \big(\|\Delta_{-1}m(D)\phi\|_{L^p} +\|\phi\|_{B^{-\epsilon}_{p,r}}\big) .
\end{equation}
\item
We have
\begin{equation}\label{eq:str-es3-0}
  \|[m(D), u\cdot\nabla]\phi\|_{B^{-\epsilon}_{p,r}} \leq C \| u\|_{W^{1,\infty}} \|\phi\|_{B^{-\epsilon}_{p,r}}.
\end{equation}
\end{enumerate}
If $\mathcal{W}$ contains only one divergence-free vector field $W$, the inequalities \eqref{eq:prod-es3}--\eqref{eq:str-es1} hold
with $\mathcal{W}$ replaced by $W$.
\end{lemma}

The lemma below is concerned with the striated estimate of the patch-type initial data.
\begin{lemma}\label{lem:str-reg}
  Let $k\in \N \cap [2,\infty)$ and $0<\gamma<1$. Assume that $D_0 \subset \R^2$ is a bounded simply connected domain with boundary
$\partial D_0$ characterized by the level-set function $\varphi_0\in C^{k,\gamma}(\R^2)$ (see \eqref{patch-ls-exp}),
and $\theta_0(x) = \bar{\theta}_0(x) 1_{D_0}(x)$ with $\bar{\theta}_0 \in C^{k-2,\gamma}(\overline{D_0})$. 
Let $W_0 = \nabla^\perp \varphi_0$. Then we have
\begin{align}\label{eq:str-reg}
  \partial_{W_0}^{k-1} \theta_0(x) \in C^{-1,\gamma}(\R^2).
\end{align}
\end{lemma}

\begin{proof}[Proof of Lemma \ref{lem:str-reg}]
  We argue as \cite[Proposition 3.1]{Sue15}.
First note that Rychkov's extension theorem (\cite{Ryc99})
guarantees that there exists a function $\widetilde\theta_0\in C^{k-2,\gamma}(\R^2)$ with the restriction
$\widetilde\theta_0|_{D_0} = \theta_0 = \bar{\theta}_0$. 

Then it suffices to prove that $\partial_{W_0}^{k-1} (\widetilde\theta_0\cdot 1_{D_0} )$ belongs to $ C^{-1,\gamma}(\R^2)$.
Since the vector field $W_0$ is tangential to the patch boundary $\partial D_0$, the operator $\partial_{W_0}^{k-1}$
communicates with the characteristic function $1_{D_0}$. Note also that (e.g. see \cite[Chap. 4.6.3]{RS96})
\begin{align}\label{eq:claim0}
  \textrm{$1_{D_0}(x)$ is the pointwise multiplier in the space $C^{-1,\gamma}(\R^2)$}.
\end{align}
Hence it needs only to show that $\partial_{W_0}^{k-1} \widetilde\theta_0 \in C^{-1,\gamma}(\R^2)$.
Due to that $\widetilde\theta_0\in C^{k-2,\gamma}(\R^2)$ and $W_0\in C^{k-1,\gamma}(\R^2)$,
this indeed can be justified from repeatedly using the product estimate \eqref{eq:prod-es}:
\begin{align}\label{es:parW0-k-the}
  \|\partial_{W_0}^{k-1} \widetilde\theta_0\|_{C^{-1,\gamma}} & \lesssim_{\|W_0\|_{L^\infty}}
  \|\nabla \partial_{W_0}^{k-2}\widetilde\theta_0\|_{C^{-1,\gamma}} \nonumber \\
  & \lesssim_{\|W_0\|_{W^{1,\infty}}} \|\nabla \partial_{W_0}^{k-3}\widetilde\theta_0\|_{C^{-1,\gamma}}
  +  \|\nabla^2 \partial_{W_0}^{k-3} \widetilde\theta_0\|_{C^{-1,\gamma}} \nonumber \\
  & \lesssim_{\|W_0\|_{W^{k-2,\infty}}} \|\nabla \widetilde\theta_0\|_{C^{-1,\gamma}}
  + \|\nabla^2 \widetilde\theta_0\|_{C^{-1,\gamma}} + \cdots
  + \|\nabla^{k-1} \widetilde\theta_0\|_{C^{-1,\gamma}} \nonumber \\
  & \lesssim_{\|W_0\|_{W^{k-2,\infty}}} \|\widetilde\theta_0\|_{C^{k-2,\gamma}} .
\end{align}
\end{proof}

\subsection{Some auxiliary lemmas}

We have the following useful commutator estimate.
\begin{lemma}\label{lem:comm-es}
  Assume $p\in [2,\infty]$, $\mathcal{R}_{-1}:=\frac{m(D)}{\Lambda}$, $\Lambda =(-\Delta)^{1/2}$ and $m(D)$ is a zero-order pseudo-differential operator with $m(\xi)\in C^\infty(\R^d\setminus\{0\})$.
Let $u=(u_1,\cdots,u_d)$ be a smooth divergence-free vector field and $\phi$ is a smooth scalar function. Then we have
\begin{equation}\label{eq:comm-es1}
  \|[\mathcal{R}_{-1}, u\cdot\nabla] \phi\|_{B^1_{p,\infty}(\R^d)} \leq C \big( \|\nabla u\|_{L^p(\R^d)} \|\phi\|_{B^0_{\infty,\infty}(\R^d)}
  + \|u\|_{L^2(\R^d)} \|\phi\|_{L^2(\R^d)} \big),
\end{equation}
with $C>0$ a constant depending on $p$ and $d$.
\end{lemma}

\begin{proof}[Proof of Lemma \ref{lem:comm-es}]
  Thanks to Bony's decomposition, we have
\begin{align}
  [\mathcal{R}_{-1}, u\cdot \nabla ] \phi & = \sum_{q\in \N}[\mathcal{R}_{-1}, S_{q-1}u\cdot\nabla] \Delta_q \phi + \sum_{q\in \N} [\mathcal{R}_{-1}, \Delta_q u \cdot\nabla] S_{q-1}\phi
  + \sum_{q\geq -1} [\mathcal{R}_{-1}, \Delta_q u\cdot\nabla] \widetilde{\Delta}_q \phi \nonumber \\
  & := \mathrm{I} + \mathrm{II} + \mathrm{III}.
\end{align}

For $\mathrm{I}$, from the spectral property there exists a bump function $\widetilde{\psi}\in C^\infty_c(\R^d)$
supported on an annulus of $\R^d$ such that
$\mathrm{I}= \sum_{q\in \N} [\mathcal{R}_{-1} \widetilde{\psi}(2^{-q}D), S_{q-1}u\cdot\nabla]\Delta_q \phi$, and due to the fact that
$\mathcal{R}_{-1}\widetilde{\psi}(2^{-q}D) = 2^{q(d-1)} \overline{h}(2^q\cdot)* $ with $\overline{h}\in \mathcal{S}(\R^d)$, we find that for every $j \geq -1$,
\begin{align*}
  2^j \|\Delta_j\mathrm{I}\|_{L^p} & \lesssim 2^j \sum_{q\in \N,|q-j|\leq 4}
  \| [\RR_{-1}\widetilde{\psi}(2^{-q}D), S_{q-1}u\cdot\nabla]\Delta_q\phi\|_{L^p} \\
  & \lesssim 2^j \sum_{|q-j|\leq 4} 2^{q(d-1)}\Big\| \int_{\R^d}\overline{h}(2^q y) (S_{q-1}u(x-y) - S_{q-1}u(x))
  \cdot\nabla\Delta_q \phi(x-y) \dd y\Big\|_{L^p_x} \\
  & \lesssim 2^j \sum_{|q-j|\leq 4} 2^{-2q} \|y \overline{h}(y)\|_{L^1} \|\nabla S_{q-1} u\|_{L^p} \|\nabla \Delta_q \phi\|_{L^\infty}
  \leq C \|\nabla u\|_{L^p} \|\phi\|_{B^0_{\infty,\infty}}.
\end{align*}
The second term $\mathrm{II}$ can be estimated in the similar way
\begin{align*}
  2^j \|\Delta_j \mathrm{II}\|_{L^p} & \lesssim 2^j \sum_{q\in \N, |q-j|\leq 4} \|[\RR_{-1} \widetilde{\psi}(2^{-q}D), \Delta_q u\cdot \nabla] S_{q-1}\phi\|_{L^p} \\
  & \lesssim 2^j \sum_{|q-j|\leq 4} 2^{-2q} \|y \overline{h}(y)\|_{L^1} \|\nabla \Delta_q u\|_{L^p} \Big(\sum_{-1\leq l \leq q-2} 2^l \|\Delta_l \phi \|_{L^\infty}\Big)
  \lesssim \|\nabla u\|_{L^p} \|\phi\|_{B^0_{\infty,\infty}} .
\end{align*}
For $j=-1$, a direct computation shows that
\begin{align*}
  \|\Delta_{-1}\mathrm{III}\|_{L^p} & \leq \|\Delta_{-1}\RR_{-1}\div (u\,\phi)\|_{L^p} + \|\Delta_{-1} (u \cdot\nabla \RR_{-1}\phi)\|_{L^p} \\
  & \lesssim \|\Delta_{-1} (u\,\phi)\|_{L^1} + \|\Delta_{-1}(u\cdot\nabla \RR_{-1}\phi)\|_{L^1} \lesssim \|u\|_{L^2} \|\phi\|_{L^2}.
\end{align*}
In view of $\mathrm{div}\, u=0$, we further split $\mathrm{III}$ as follows
\begin{align*}
  \Delta_j \mathrm{III} = \sum_{q\geq j-3}  \Delta_j \RR_{-1}\mathrm{div} (\Delta_q u \, \widetilde{\Delta}_q \phi) -  \sum_{q\geq j-3}  \Delta_j (\Delta_q u \cdot\nabla \RR_{-1}\widetilde{\Delta}_q \phi) = : \Delta_j \mathrm{III}_1 + \Delta_j \mathrm{III}_2 .
\end{align*}
Since $\Delta_j \RR_{-1}\mathrm{div}\,$ is uniformly bounded in $L^p$ for every $j\in\N$, we use Bernstein's inequality to derive that for every $j\in \N$,
\begin{align*}
  & \quad 2^j \|\Delta_j \mathrm{III}_1\|_{L^p} \lesssim 2^j \sum_{q\geq j-3,q\geq 2} \|\Delta_q u \,\widetilde{\Delta}_q \phi\|_{L^p} + 2^j \sum_{q\geq j-3, q\leq 2} \|\Delta_q u\,\widetilde{\Delta}_q \phi\|_{L^1} \\
  & \lesssim \sum_{q\geq j-3} 2^{j-q} \|\Delta_q \nabla u\|_{L^p} \|\widetilde\Delta_q \phi\|_{L^\infty} + \sum_{-1\leq q\leq 2} \|\Delta_q u\|_{L^2} \|\widetilde\Delta_q \phi\|_{L^2}
  \lesssim \|\nabla u\|_{L^p} \|\phi\|_{B^0_{\infty,\infty}} + \|u\|_{L^2} \|\phi\|_{L^2}.
\end{align*}
Similarly, one also gets
$2^j \|\Delta_j \mathrm{III}_2\|_{L^p} \lesssim \|\nabla u\|_{L^p} \|\phi\|_{B^0_{\infty,\infty}} + \|u\|_{L^2} \|\phi\|_{L^2}$.
Hence, gathering the above estimates leads to the commutator estimate \eqref{eq:comm-es1}.

\end{proof}

We recall the following regularity estimates of the transport and transport-diffusion equations
(one can see \cite[Chap. 3]{BCD11} for the detailed proof).
\begin{lemma}\label{lem:TD-sm2}
  Let $(\rho ,p,r)\in [1,\infty]^3$ and $-1< s < 1$.
Assume that $u$ is a smooth divergence-free vector field, and $\phi$ is a smooth function solving the transport/transport-diffusion equation
\begin{equation}\label{eq:TD-eq2}
  \partial_t \phi + u\cdot\nabla \phi - \nu \Delta \phi = f,\quad \phi|_{t=0}(x)=\phi_0(x),\quad x\in \R^d.
\end{equation}
We have the following statements.
\begin{enumerate}[(1)]
\item
If $\nu=0$, then there exists a constant $C=C(d,s)$ so that for every $t>0$,
\begin{equation}\label{eq:T-sm2}
  \|\phi\|_{L^\infty_t (B^s_{p,r})} \leq C \bigg( \|\phi_0\|_{B^s_{p,r}} + \|f\|_{\widetilde{L}^1_t (B^s_{p,r})}
  + \int_0^t \|\nabla u(\tau)\|_{L^\infty} \|\phi(\tau)\|_{ B^s_{p,r}} \dd \tau \bigg),
\end{equation}
and
\begin{equation}\label{eq:T-sm3}
  \|\phi\|_{L^\infty_t (B^s_{p,r})} \leq C e^{C \int_0^t \|\nabla u\|_{L^\infty} \dd \tau} \Big( \|\phi_0\|_{B^s_{p,r}} + \|f\|_{\widetilde{L}^1_t (B^s_{p,r})} \Big).
\end{equation}
\item If $\nu>0$, then there exists a constant $C= C(d,s)$ so that for every $t>0$,
\begin{equation}\label{eq:TD-sm2}
  \nu^{\frac{1}{\rho}} \|\phi\|_{\widetilde{L}^\rho_t (B^{s+ \frac{2}{\rho}}_{p,r})} \leq C (1+\nu t)^{\frac{1}{\rho}}
  \bigg( \|\phi_0\|_{B^s_{p,r}} + \|f\|_{\widetilde{L}^1_t (B^s_{p,r})}  + \int_0^t \|\nabla u(\tau)\|_{L^\infty} \|\phi(\tau)\|_{B^s_{p,r}} \dd \tau \bigg),
\end{equation}
and
\begin{equation}\label{eq:TD-sm3}
  \nu^{\frac{1}{\rho}} \|\phi\|_{\widetilde{L}^\rho_t (B^{s+ \frac{2}{\rho}}_{p,r})} \leq C (1+\nu t)^{\frac{1}{\rho}} e^{C\int_0^t \|\nabla u\|_{L^\infty}\dd \tau}
  \Big( \|\phi_0\|_{B^s_{p,r}} + \|f\|_{\widetilde{L}^1_t (B^s_{p,r})}  \Big).
\end{equation}
\end{enumerate}
In the above $\|g\|_{\widetilde{L}^\rho_T(B^s_{p,r})} : = \|\{2^{qs} \|\Delta_q g\|_{L^\rho_T(L^p)} \}_{q\geq -1}\|_{\ell^r}$
denotes the norm of Chemin-Lerner's spacetime Besov space $\widetilde{L}^\rho(0,T; B^s_{p,r}(\R^d))$ (see \cite{BCD11}).
\end{lemma}

We also use the following smoothing estimate for the transport-diffusion equation.
\begin{lemma}\label{lem:TD-sm}
  Assume $p\in [2,\infty)$, $\rho\in [1,\infty]$. Let $u$ be a smooth divergence-free vector field, and $\phi$ be a smooth function solving equation \eqref{eq:TD-eq2} with $\nu>0$.
Then for any $t>0$, we have
\begin{equation}\label{TD-sm-es}
  \sup_{q\in \N} 2^{\frac{2 q}{\rho}} \|\Delta_q \phi\|_{L^\rho_t (L^p)} \leq C \bigg(\sup_{q\in \N}\|\Delta_q \phi_0\|_{L^p} + \int_0^t \|\nabla u\|_{L^p} \|\phi\|_{B^0_{\infty,\infty}} \dd \tau + \|f\|_{L^1_t (L^p)} \bigg) .
\end{equation}
Besides, the estimate \eqref{TD-sm-es} also holds by replacing $\|\nabla u\|_{L^p} \|\phi\|_{B^0_{\infty,\infty}} $
with $\|u\|_{W^{1,p}} \big(\|(\mathrm{Id}- \Delta_{-1})\phi\|_{B^0_{\infty,\infty}} + \|\nabla \Delta_{-1}\phi\|_{L^\infty}\big)$.
\end{lemma}

\begin{proof}[Proof of Lemma \ref{lem:TD-sm}]

For every $q\in \N$, applying the operator $\Delta_q$ to equation \eqref{eq:TD-eq2} yields
\begin{equation}\label{phi-q-eq}
  \partial_t (\Delta_q \phi) + u\cdot \nabla (\Delta_q \phi) - \nu\Delta (\Delta_q \phi) = - [\Delta_q, u\cdot\nabla] \phi + \Delta_q f .
\end{equation}
Multiplying both sides of \eqref{phi-q-eq} with $|\Delta_q \phi|^{p-2}\Delta_q \phi$ and using the following estimate (see \cite{Dan97})
\begin{equation*}
  \int_{\R^d} -\Delta (\Delta_q \phi)\, |\Delta_q \phi|^{p-2} \Delta_q \phi \dd x \geq c 2^{2q } \| \Delta_q \phi\|_{L^p}^p,
\end{equation*}
with $c>0$ independent of $q$, we obtain
\begin{equation*}
  \frac{1}{p} \frac{\dd}{\dd t} \| \Delta_q \phi(t)\|_{L^p}^p + c\nu 2^{2q} \| \Delta_q \phi(t)\|_{L^p}^p \leq \|\Delta_q \phi(t)\|_{L^p}^{p-1}\Big(\|[\Delta_q,u\cdot\nabla]\phi(t)\|_{L^p} + \|\Delta_q f(t)\|_{L^p} \Big).
\end{equation*}
Integrating on time interval $[0,t]$ leads to
\begin{equation*}
  \|\Delta_q \phi(t)\|_{L^p} \leq e^{-c\nu 2^{2q}t} \|\Delta_q \phi_0\|_{L^p} + \int_0^t e^{- c\nu 2^{2q} (t-\tau)} \Big(\|[\Delta_q,u\cdot\nabla]\phi(\tau)\|_{L^p} + \|\Delta_q f(\tau)\|_{L^p} \Big) \dd \tau.
\end{equation*}
Young's inequality ensures that
\begin{equation*}
  \| \Delta_q \phi \|_{L^\rho_t (L^p)} \leq C 2^{-\frac{2q}{\rho}} \Big( \|\Delta_q \phi_0\|_{L^p} + \|[\Delta_q, u\cdot\nabla]\phi\|_{L^1_t(L^p)} + \|\Delta_q f\|_{L^1_t (L^p)}  \Big).
\end{equation*}
Recall the following commutator estimate $\sup_{q\geq -1} \|[\Delta_q, u\cdot\nabla]\phi\|_{L^p} \leq C \|\nabla u\|_{L^p} \|\phi\|_{B^0_{\infty,\infty}}$ (see \cite[Lemma 6.10]{HKR10}),
then the desired estimate \eqref{TD-sm-es} follows from combining the above two inequalities.

Furthermore, as for the replacement in Lemma \ref{lem:TD-sm}, it suffices to notice that
\begin{align*}
  & \quad \sup_{q\geq -1} \|[\Delta_q, u\cdot\nabla]\phi\|_{L^p} \\
  & \leq \sup_{q\geq -1} \|[\Delta_q, u\cdot\nabla](\mathrm{Id}-\Delta_{-1})\phi\|_{L^p} + \sup_{q\geq -1} \|\Delta_q (u\cdot\nabla\Delta_{-1}\phi)\|_{L^p} +
  \sup_{q\geq -1} \| u\cdot\nabla \Delta_q \Delta_{-1}\phi\|_{L^p}  \\
  & \leq C \|\nabla u\|_{L^p} \|(\mathrm{Id}-\Delta_{-1})\phi\|_{B^0_{\infty,\infty}} + C \|u\|_{L^p} \|\nabla \Delta_{-1}\phi\|_{L^\infty}.
\end{align*}
\end{proof}

We list some basic properties of the particle-trajectory map $X_t$ as follows (one can refer to \cite[Proposition 3.10]{BCD11} for the proof).
\begin{lemma}\label{lem:flow}
Assume $u(x,t)$ is a divergence-free velocity field belonging to $L^1(0,T; W^{1,\infty}(\R^d))$. Let $X_t(x)$ be the the particle-trajectory generated by velocity $u$ which solves \eqref{eq:flow0},
that is
\begin{equation}\label{eq:flow2}
   X_t(x) = x + \int_0^t u(X_\tau(x),\tau) \dd \tau.
\end{equation}
Then the system \eqref{eq:flow0} or \eqref{eq:flow2} has a unique solution $X_t(\cdot):\R^d\mapsto \R^d$ on $[0,T]$ which is a measure-preserving bi-Lipschitzian homeomorphism satisfying that
$\nabla X_t$ and its inverse $\nabla X^{-1}_t$
belong to $L^\infty([0,T]\times \R^d)$ with
\begin{equation}\label{DXest}
  \|\nabla X^{\pm1}_t\|_{L^\infty(\R^d)} \leq e^{\int_0^t \|\nabla u\|_{L^\infty}\dd \tau}.
\end{equation}
Besides, the following statements hold true.
\begin{enumerate}[(1)]
\item If additionally $u\in L^1(0,T; C^{1,\gamma}(\R^d))$, then $X^{\pm1}_t \in L^\infty(0,T; C^{1,\gamma}(\R^d))$ with
\begin{equation}\label{X-C1gam-es}
  \|\nabla X^{\pm 1}_t\|_{C^\gamma} \leq e^{(2+\gamma)\int_0^t \|\nabla u\|_{L^\infty} \dd \tau} \bigg( 1 + \int_0^t \|\nabla u(\tau)\|_{C^\gamma} \dd \tau \bigg).
\end{equation}
\item If additionally $u\in L^1(0,T; W^{2,\infty}(\R^d))$, then $X^{\pm1} \in L^\infty(0,T; W^{2,\infty}(\R^d))$ with
\begin{equation}\label{X-W2inf-es}
  \|\nabla^2 X^{\pm 1}_t\|_{L^\infty} \leq e^{3\int_0^t \|\nabla u\|_{L^\infty} \dd \tau} \int_0^t \|\nabla^2 u(\tau)\|_{L^\infty} \dd \tau.
\end{equation}
\end{enumerate}
\end{lemma}

The operator $\nabla^2 \nabla^\perp \partial_1 (-\Delta)^{-2}$ has the following explicit repression formula
(one can refer to \cite{Stein} for the proof, especially see Section III.3 for the calculation of the coefficients $a_{ijk}$ and $\sigma_{ijk}$).

\begin{lemma}\label{lem:op}
  Let $\nabla = (\partial_1,\partial_2)^{\mathrm{T}}$, $\nabla^\perp = (\partial_1^\perp,\partial_2^\perp)^{\mathrm{T}} = (-\partial_2,\partial_1)^{\mathrm{T}}$, $\Lambda= (-\Delta)^{\frac{1}{2}}$.
Then the family of operators $\nabla^2 \nabla^\perp \partial_1 \Lambda^{-4}$ is composed of zero-order pesudo-differential operators satisfying that for each $i,j,k=1,2$,
\begin{equation}\label{eq:op}
  \partial_i \partial_j \partial_k^\perp \partial_1 \Lambda^{-4} f(x) = \mathrm{p.v.} \int_{\R^2} \frac{\sigma_{ijk}(x-y)}{|x-y|^2} f(y) \dd y + a_{ijk} f(x),\quad \forall f\in \mathcal{S}(\R^2),
\end{equation}
with $a_{ijk}\in \R$ and $\frac{\sigma_{ijk}(y)}{|y|^2}$ is the standard Calder\'on-Zygmund kernel; more precisely, the coefficient $a_{ijk}$ and the zero-mean function $\sigma_{ijk}(y)$ are given by
\begin{align*}
  \sigma_{111}(y) = - \sigma_{212}(y) = -\sigma_{122}(y) = -\pi \frac{2y_1^3y_2 -6y_1y_2^3}{|y|^4},& \quad a_{111} = a_{212} = a_{122} = 0, \\
  \sigma_{112}(y) =  \pi \frac{y_1^4 + 6 y_1^2y_2^2 + 3 y_2^4}{|y|^4},& \quad a_{112}=\frac{3\pi^2}{2}, \\
  \sigma_{121}(y) = \sigma_{211}(y) = \sigma_{222}(y) = \pi \frac{-y_1^4 + 6 y_1^2 y_2^2 - y_2^4}{|y|^4},& \quad a_{121} = a_{211} =  a_{222} = \frac{\pi^2}{2}, \\
  \sigma_{221}(y) = \pi \frac{6 y_1^3 y_2 - 2 y_1 y_2^3}{|y|^4},&\quad a_{221}=0.
\end{align*}
\end{lemma}

\section{Persistence of $C^{1,\gamma}$, $W^{2,\infty}$ and $C^{2,\gamma}$-boundary regularities}\label{sec:C1-2gam}
In this section we are dedicated to the proof of the  $C^{1,\gamma}$-, $W^{2,\infty}$- and
$C^{2,\gamma}$-regularity persistence result for the temperature front boundary.

As mentioned in the introduction section, a good unknown $\Gamma$ is introduced and plays a crucial role in the proof.
Note that the equation of vorticity $\omega := \textit{curl}\, u = \partial_1 u_2 -\partial_2 u_1$ reads as
\begin{equation}\label{vort-eq}
  \partial_t \omega + u\cdot \nabla \omega - \Delta \omega = \partial_1 \theta,\qquad \omega|_{t=0} = \omega_0.
\end{equation}
Denote by $\mathcal{R}_{-1} := \partial_1 (-\Delta)^{-1}=\partial_1 \Lambda^{-2} $
and $\Gamma:= \omega - \mathcal{R}_{-1} \theta$.
We see that
$\partial_t \omega + u\cdot \nabla \omega - \Delta \Gamma =0$, and
\begin{equation}\label{R-1the-eq}
  \partial_t \mathcal{R}_{-1}\theta + u\cdot \nabla \mathcal{R}_{-1}\theta = - [\mathcal{R}_{-1},u\cdot\nabla]\theta,
\end{equation}
which immediately leads to
\begin{equation}\label{Gamm-eq}
  \partial_t \Gamma + u\cdot\nabla \Gamma -\Delta \Gamma = [\mathcal{R}_{-1},u\cdot\nabla]\theta,  \qquad \Gamma|_{t=0}=\Gamma_0,
\end{equation}
with the notion of commutator operator $[A,B] := AB-BA$.

\subsection{Persistence of $C^{1,\gamma}$-boundary regularity}

The main result of this subsection is the following global wellposedness result of the 2D Boussinesq system \eqref{BoussEq2D}.
\begin{proposition}\label{prop:gwp}
  Let $\theta_0\in L^2\cap L^\infty(\R^2)$, and $u_0\in H^1(\R^2)$ be a divergence-free vector field. Then, for any given $T>0$, there exists a unique global solution $(\theta,u)$ to the 2D viscous Boussinesq system \eqref{BoussEq2D}
with
\begin{equation}\label{the-u-apes1}
  \theta\in L^\infty(0,T; L^2\cap L^\infty),\quad u \in C(0,T; H^1) \cap L^2(0,T; H^2) \cap L^1(0,T; C^{1,\gamma}),\quad \forall \gamma \in (0,1).
\end{equation}
\end{proposition}

In light of Proposition \ref{prop:gwp}, we go back to the temperature patch problem of system \ref{BoussEq2D}
and show the persistence of $C^{1,\gamma}$-boundary regularity.
Indeed, recalling that $X_t$ is the particle trajectory given by \eqref{eq:flow0},
the expression formula $\varphi(x,t) =\varphi_0(X^{-1}_t(x))$ and Lemma \ref{lem:flow} guarantee the desired result $\varphi\in L^\infty(0,T; C^{1,\gamma})$ with
\begin{equation}\label{nab-varp-es}
  \|\nabla\varphi(t)\|_{C^\gamma} \lesssim \|\nabla \varphi_0\|_{C^\gamma} \|\nabla X^{-1}_t\|_{L^\infty}^{1+\gamma} + \|\nabla \varphi_0\|_{L^\infty} \|\nabla X^{-1}_t\|_{C^\gamma} \leq C e^{C(1+T)^2},
\end{equation}
where we have used estimates \eqref{eq:uC1-gam}-\eqref{eq:u-Lip-es} below.

\begin{proof}[Proof of Proposition \ref{prop:gwp}]
The existence part is standard: we first regularize the initial data as $(\theta_{0,\epsilon},u_{0,\epsilon}) = \rho_\epsilon * (\theta_0, u_0)$
with $\rho_\epsilon = \epsilon^{-2}\rho(\frac{\cdot}{\epsilon})$, $\epsilon>0$, then the previous work (e.g. \cite{Cha06})
implies there exists a unique global smooth solution $(\theta_\epsilon, u_\epsilon)$ to the system \eqref{BoussEq2D} associated with
$(\theta_{0,\epsilon}, u_{0,\epsilon})$; moreover, the \textit{a priori} estimates below guarantee that $(\theta_\epsilon,u_\epsilon)$ satisfies \eqref{the-u-apes1} uniformly in $\epsilon$
and also the particle-trajectory $X_{t,\epsilon}$ associated with $u_\epsilon$ belongs to $L^\infty(0,T; C^{1,\gamma})$ uniformly in $\epsilon$;
thus combined with the standard compactness procedure (e.g. \cite[Chap. 8]{MB02}), one can pass $\epsilon\rightarrow 0$ (up to a subsequence) to show that there exist functions $(\theta,u)$ satisfying \eqref{the-u-apes1}
solve 2D Boussinesq system \eqref{BoussEq2D} in the distributional sense.

The uniqueness part can be proved exactly as \cite[Theorem 2.1]{GGJ17}.

In the following we only focus on the \textit{a priori} estimates.
From the equation of $\theta$, we directly have
\begin{equation}\label{the-es}
  \|\theta(t)\|_{L^2\cap L^\infty(\R^2)} \leq \|\theta_0\|_{L^2\cap L^\infty(\R^2)},\quad \forall t\geq 0.
\end{equation}
Then the classical energy estimate of system \eqref{BoussEq2D} gives that
\begin{equation}\label{eq:uL2}
  \|u\|_{L^\infty_T(L^2)}^2 + \|\nabla u\|_{L^2_T(L^2)}^2 \leq C_0 (1+T)^2 (\|u_0\|_{L^2}^2 + \|\theta_0\|_{L^2}^2).
\end{equation}

Now we consider the energy estimate of vorticity $\omega$. By taking the inner product of the equation \eqref{vort-eq} with $\omega$ itself,
and using the integration by parts, we see that
\begin{align*}
  \frac{1}{2}\frac{\dd}{\dd t} \|\omega(t)\|_{L^2}^2 + \|\nabla \omega(t)\|_{L^2}^2 \leq \Big| \int_{\R^2} \theta \, \partial_1 \omega(x,t) \dd x \Big|
  \leq \frac{1}{2}\|\theta(t)\|_{L^2}^2 + \frac{1}{2} \|\nabla \omega(t)\|_{L^2}^2.
\end{align*}
Integrating in the time variable leads to
\begin{equation}\label{vortL2-es}
  \|\omega\|_{L^\infty_T(L^2)}^2 + \|\nabla\omega\|_{L^2_T(L^2)}^2 \leq C_0 (\|\omega_0\|_{L^2}^2 + (1+ T)\|\theta_0\|_{L^2}^2),
\end{equation}
which in combination with \eqref{eq:uL2} and the interpolation ensures that for every $\rho \in [2,\infty]$,
\begin{equation}\label{eq:uH1-es}
  \|u\|_{L^\infty_T (H^1)}^2 + \|\nabla u\|_{L^2_T (H^1)}^2 + \|u\|_{L^\rho_T (H^{1+ \frac{2}{\rho}})}^2 \leq C (1+T)^2 (\|u_0\|_{H^1}^2 + \|\theta_0\|_{L^2}^2).
\end{equation}

Next based on estimates \eqref{the-es} and \eqref{eq:uH1-es}, we intend to obtain the $L^1_T (B^\gamma_{\infty,1})$-estimate of $\omega$.
We use the high-low frequency decomposition,
and due to the influence of forcing term $\partial_1\theta$ in equation \eqref{vort-eq},
it seems better to consider the estimation of $\Gamma$ and then use the relation $\omega = \Gamma + \mathcal{R}_{-1}\theta$ in the high-frequency part.
Applying Lemma \ref{lem:TD-sm} to the equation \eqref{Gamm-eq}, and using Lemma \ref{lem:comm-es}, we get that for every $q\in \N$ and $\rho\in [1,\infty]$,
\begin{align}\label{Gam-es3}
  \|\Delta_q \Gamma\|_{L^\rho_T (L^2)} \lesssim &   2^{-\frac{2}{\rho}q} \Big(\sup_{q\in \N} \|\Delta_q \Gamma_0\|_{L^2} +  \|u\|_{L^\infty_T (H^1)} \big(\|(\mathrm{Id}-\Delta_{-1})\Gamma\|_{L^1_T (B^0_{\infty,\infty})}
  + \|\nabla \Delta_{-1}\Gamma\|_{L^1_T (L^\infty)}\big) \Big) \nonumber \\
  &  +  2^{-\frac{2}{\rho}q} \|[\mathcal{R}_{-1},u\cdot\nabla]\theta\|_{L^1_T (L^2)} \nonumber \\
  \lesssim &  2^{-\frac{2}{\rho}q} \Big(\|(\omega_0,\theta_0)\|_{L^2} + \|u\|_{L^\infty_T (H^1)}
  \big(\|\omega\|_{L^1_T (B^0_{\infty,\infty})} + \|\theta\|_{L^1_T (L^2)}\big)+ T \|u\|_{L^\infty_T (H^1)} \|\theta\|_{L^\infty_T(L^2\cap L^\infty)}  \Big) \nonumber \\
  \leq & C 2^{-\frac{2}{\rho}q} \Big( (1+T^2) + (1+T) \|\omega\|_{L^1_T (B^0_{\infty,\infty})} \Big) ,
\end{align}
with $C>0$ depending only on the initial data.
Then let $N\in \N$ be an integer chosen later, and by virtue of Bernstein's inequality we find that
\begin{align*}
  \|\omega\|_{L^1_T (B^0_{\infty,1})} & \leq \sum_{-1\leq q \leq N}  \|\Delta_q \omega\|_{L^1_T (L^\infty)}
  + \sum_{q>N}  \|\Delta_q \omega\|_{L^1_T (L^\infty)} \\
  & \lesssim  \sum_{-1\leq q\leq N} 2^q \|\Delta_q \omega\|_{L^1_T (L^2)}
  + \sum_{q>N} 2^q  \|\Delta_q \Gamma\|_{L^1_T (L^2)}  + \sum_{q>N}  \|\Delta_q \RR_{-1}\theta\|_{L^1_T (L^\infty)} \nonumber \\
  & \lesssim 2^N \|\omega\|_{L^\infty_T(L^2)} T
  + (1+ T) \sum_{q>N} 2^{ -q }  \big(1 + T + \|\omega\|_{L^1_T (B^0_{\infty,\infty})} \big)  + \sum_{q>N} 2^{-q }  \|\theta\|_{L^1_T (L^\infty)} \\
  & \leq C 2^N  (1+T) + C 2^{-N}  (1+T) \big(1 + T + \|\omega\|_{L^1_T (B^0_{\infty,1})}\big).
\end{align*}
By choosing $N \in \N$ large enough so that $C 2^{-N}  (1+ T) \approx \frac{1}{2}$, we conclude that
\begin{equation}\label{vort-B0es}
  \|\omega\|_{L^1_T (B^0_{\infty,1})} \leq C (1+T)^2,
\end{equation}
with $C= C(\|u_0\|_{H^1}, \|\theta_0\|_{L^2\cap L^\infty})$.
Moreover, it also yields that for every $\gamma\in (0,1) $,
\begin{align}\label{vort-Bs-es}
  \|\omega\|_{L^1_T (B^\gamma_{\infty,1})} & \leq C \|\Delta_{-1} \omega\|_{L^1_T(L^\infty)}
  + \sum_{q\in \N} 2^{q\gamma} \big( \|\Delta_q \Gamma\|_{L^1_T(L^\infty)} + \|\Delta_q \RR_{-1}\theta\|_{L^1_T(L^\infty)} \big) \nonumber \\
  & \leq C T \|\omega\|_{L^\infty_T(L^2)} + \sum_{q\in \N} 2^{q(\gamma -1)} \big( (1+T)^2 + (1+T) \|\omega\|_{L^1_T(B^0_{\infty,1})}
  + \|\theta\|_{L^1_T(L^\infty)}  \big) \nonumber \\
  & \leq C  (1+ T)^3 .
\end{align}
On the other hand, we use \eqref{Gam-es3} and \eqref{vort-B0es} to deduce that for every $\varrho\in [1,2)$,
\begin{align}\label{vort-B0es2}
  \|\omega\|_{L^\varrho_T (B^0_{\infty,1})} & \leq \|\Delta_{-1} \omega\|_{L^\varrho_T (L^\infty)} + \sum_{q\in \N} \|\Delta_q \Gamma\|_{L^\varrho_T (L^\infty)} + \sum_{q\in \N} \|\Delta_q \RR_{-1}\theta\|_{L^\varrho_T (L^\infty)}  \nonumber \\
  & \lesssim  \|\omega\|_{L^\varrho_T (L^2)} + \sum_{q\in \N} 2^{q(1-\frac{2}{\varrho})} \big((1+T)^2  + (1+T)\|\omega\|_{L^1_T (B^0_{\infty,1})}\big) + \sum_{q\in \N} 2^{-q} \|\theta\|_{L^\varrho_T (L^\infty)} \nonumber \\
  & \lesssim T^{\frac{1}{\varrho}}(1+T) + (1+T)^3 \lesssim (1+T)^3.
\end{align}

As a direct consequence of estimates \eqref{vort-Bs-es}-\eqref{vort-B0es2} and \eqref{eq:uL2}, we have that for every $\gamma\in (0,1)$,
\begin{equation}\label{eq:uC1-gam}
  \|u\|_{L^1_T (B^{1+\gamma}_{\infty,1})}
  \lesssim \|\Delta_{-1} u\|_{L^1_T (L^\infty)} + \|\omega\|_{L^1_T (B^\gamma_{\infty,1})}
  \lesssim (1+ T)^3,
\end{equation}
and
\begin{equation}\label{eq:u-Lip-es}
  \|u\|_{L^1_T (W^{1,\infty})} \lesssim \|u\|_{L^1_T (B^1_{\infty,1})} \lesssim \|\Delta_{-1} u\|_{L^1_T (L^\infty)} + \|\omega\|_{L^1_T (B^0_{\infty,1})} \lesssim (1+T)^2 ,
\end{equation}
and
$\|u\|_{L^\varrho_T (W^{1,\infty})} \lesssim \|u\|_{L^\varrho_T (B^1_{\infty,1})} \lesssim (1+T)^3$ with $\varrho\in [1,2)$.

\end{proof}

\subsection{Control of curvature: persistence of $W^{2,\infty}$-boundary regularity}
According to Lemma \ref{lem:flow} and in light of the estimate \eqref{es:nab2u-Linf} below,
the particle-trajectory $X_t$ and its inverse $X^{-1}_t$ satisfy
$X^{\pm1}_t \in L^\infty(0,T; W^{2,\infty})$, thus the level-set characterization $\varphi(x,t) = \varphi_0(X^{-1}_t(x))$ fulfills that $\varphi\in L^\infty(0,T; W^{2,\infty})$ with
\begin{align}\label{eq:nab2-vap-es}
  \|\nabla^2 \varphi\|_{L^\infty_T (L^\infty)} \lesssim \|\nabla^2 \varphi_0\|_{L^\infty} \|\nabla X^{-1}_t\|_{L^\infty_T (L^\infty)}^2 + \|\nabla \varphi_0\|_{L^\infty} \|\nabla^2 X^{-1}_t\|_{L^\infty_T (L^\infty)} \leq C e^{C(1+T)^2}.
\end{align}

Hence in order to show the $W^{2,\infty}$-regularity persistence property,
it suffices to prove \eqref{es:nab2u-Linf} below, which means that $\nabla^2 u \in L^\rho(0,T; L^\infty(\R^2))$ with $\rho \in [1,\frac{2p}{p+2})$.
By virtue of the Biot-Savart law and relation $\omega = \Gamma + \RR_{-1}\theta$ (recalling $\RR_{-1}= \partial_1 \Lambda^{-2}$), we see that
\begin{equation}\label{exp:nab2u}
  \nabla^2 u = \nabla^2 \nabla^\perp \Lambda^{-2} \omega = \nabla^2 \nabla^\perp \Lambda^{-2} \Gamma + \nabla^2 \nabla^\perp \partial_1 \Lambda^{-4} \theta.
\end{equation}

First under the assumption $u_0\in H^1\cap W^{1,p}(\R^2)$ with some $p>2$, we show that $\Gamma$ has some more refined estimates.
Multiplying both sides of equation \eqref{Gamm-eq} with $|\Gamma|^{p-2}\Gamma$ and integrating on the space variable, we get
\begin{align*}
  \frac{1}{p}\frac{\dd }{\dd t}\|\Gamma(t)\|_{L^p}^p + (p-1)\int_{\R^2} |\nabla \Gamma|^2 |\Gamma|^{p-2}(x,t) \dd x \leq \|[\RR_{-1}, u\cdot\nabla]\theta(t)\|_{L^p} \|\Gamma(t)\|_{L^p}^{p-1}.
\end{align*}
It follows that
\begin{align*}
  \|\Gamma\|_{L^\infty_T (L^p)} \leq \|\Gamma_0\|_{L^p} + \|[\RR_{-1}, u\cdot\nabla]\theta\|_{L^1_T (L^p)} \leq C  + \|[\RR_{-1}, u\cdot\nabla]\theta\|_{L^1_T (L^p)},
\end{align*}
where we have used the fact that 
\begin{equation}\label{eq:data-es}
  \|\Gamma_0\|_{L^p} \leq \|\omega_0\|_{L^p} + \|\RR_{-1}\theta_0\|_{L^p} \leq \|\omega_0\|_{L^p} + C_p \|\theta_0\|_{L^{\frac{2p}{2+p}}} \lesssim 1.
\end{equation}
Observe that Lemma \ref{lem:comm-es} and estimates \eqref{the-es}, \eqref{eq:uH1-es} guarantee that
\begin{equation*}
  \|[\RR_{-1},u\cdot\nabla]\theta\|_{L^\infty_T (B^1_{2,\infty})} \leq C_0 \Big(\|\omega\|_{L^\infty_T(L^2)} \|\theta\|_{L^\infty_T (L^\infty)}
  + \|u\|_{L^\infty_T(L^2)} \|\theta\|_{L^\infty_T(L^2)} \Big) \lesssim (1+T).
\end{equation*}
Thus the embedding $B^1_{2,\infty}\subset L^p$ implies
\begin{align}\label{es:Gam-Lp}
  \|\Gamma\|_{L^\infty_T (L^p)} \leq C(1+T),
\end{align}
which combined with \eqref{eq:data-es} leads to that
\begin{align}\label{ome-Lp-es}
  \|\omega\|_{L^\infty_T (L^p)} \leq \|\Gamma\|_{L^\infty_T (L^p)} + \|\RR_{-1}\theta\|_{L^\infty_T (L^p)} \leq \|\Gamma\|_{L^\infty_T (L^p)} + C_p \|\theta_0\|_{L^{\frac{2p}{2+p}}} \leq C(1+T).
\end{align}
We also have
\begin{equation}\label{u-W1p-es}
  \|u\|_{L^\infty_T (W^{1,p})} \lesssim \|\Delta_{-1} u\|_{L^\infty_T (W^{1,p})} + \|\omega\|_{L^\infty_T (L^p)} \lesssim 1+ T.
\end{equation}

Now applying \eqref{TD-sm-es} in Lemma \ref{lem:TD-sm} yields that for every $q\in \N$ and $\rho\geq 1$,
\begin{align}
   & \quad\|\Delta_q \Gamma\|_{L^\rho_T (L^p)} \lesssim 2^{-\frac{2}{\rho}q} \Big( \|\Gamma_0\|_{L^p} + \|\nabla u\|_{L^\infty_T (L^p)} \|\Gamma\|_{L^1_T (B^0_{\infty,\infty})}
   + \|[\RR_{-1},u\cdot\nabla]\theta\|_{L^1_T (L^p)}\Big) \nonumber \\
   & \lesssim 2^{-\frac{2}{\rho}q} \Big( \|\Gamma_0\|_{L^p} + \|\omega\|_{L^\infty_T (L^p) } \big(\|\omega\|_{L^1_T (B^0_{\infty,\infty})}
   + \|\RR_{-1}\theta\|_{L^1_T (B^0_{\infty,\infty})}\big) + \|[\RR_{-1},u\cdot\nabla]\theta\|_{L^1_T B^1_{2,\infty}}\Big) \nonumber \\
   & \lesssim 2^{-\frac{2}{\rho}q} \big((1+T)^2 + (1+T)\|\omega\|_{L^1_T (B^0_{\infty,\infty})} \big) \lesssim 2^{-\frac{2}{\rho}q} (1+T)^3 , \label{Gam-es2}
\end{align}
where in the last line we have used \eqref{vort-B0es2}. By virtue of the high-low frequency decomposition, we find that for every $\rho\in [1, \frac{2p}{p+2})$,
\begin{align}\label{Gam-Lip-es}
  \|\nabla^2\nabla^\perp \Lambda^{-2} \Gamma\|_{L^\rho_T (L^\infty)}
  & \leq \|\nabla^2\nabla^\perp \Lambda^{-2} \Delta_{-1} \Gamma\|_{L^\rho_T (L^\infty)}
  + \sum_{q\in \N}\|\nabla^2\nabla^\perp \Lambda^{-2} \Delta_q \Gamma \|_{L^\rho_T (L^\infty)} \nonumber \\
  & \lesssim \|\Delta_{-1}\Gamma\|_{L^\rho_T (L^p)} + \sum_{q\in \N} 2^{q(1+ 2/p)} \|\Delta_q \Gamma\|_{L^\rho_T (L^p)} \nonumber \\
  & \lesssim (1+T)^2 + \sum_{q\in \N} 2^{q(1+ \frac{2}{p} -\frac{2}{\rho})} (1+T)^3 \lesssim (1+T)^3.
\end{align}

Next, since $\theta(x,t) = \bar{\theta}_0(X^{-1}_t(x)) \,1_{D(t)}(x)$ for every $t\in [0,T]$ and $D(t)= X_t(D_0)$ satisfies $D(t) = \{x\in \R^2: \varphi(x,t)>0\}$ 
(with $\varphi$ given by \eqref{varphi-eq}),
we claim that $\nabla^2\nabla^\perp \partial_1 \Lambda^{-4} \theta \in L^\infty(\R^2 \times [0,T] )$ and
\begin{equation}\label{eq:claim1}
  \|\nabla^2\nabla^\perp \partial_1 \Lambda^{-4} \theta \|_{L^\infty_T (L^\infty) } \leq C e^{C(1+T)^2}.
\end{equation}
Indeed, recalling that one needs to control the $L^\infty$-norm of $\nabla u$ with $\nabla u = \nabla \nabla^\perp \Lambda^{-2} \omega = \nabla \nabla^\perp \Lambda^{-2} 1_{D(t)}$ in the vorticity patch problem of 2D Euler equations (e.g. see \cite{BerC,Chem91}),
the proof of \eqref{eq:claim1} is quite analogous. Below we mainly argue as \cite[Proposition 1]{BerC}.
According to Lemma \ref{lem:op}, we have that for every $x_0\in \R^2$ and $i,j,k=1,2$,
\begin{equation}\label{eq:SIO-exp}
\begin{split}
  \partial_i \partial_j \partial_k^\perp \partial_1 \Lambda^{-4} \theta(x_0,t) = \, \mathrm{p.v.}\int_{D(t)} \frac{\sigma_{ijk}(x_0 -y)}{|x_0 -y|^2} \bar{\theta}_0(X^{-1}_t(y)) \dd y + a_{ijk}\, \theta(x_0,t).
\end{split}
\end{equation}
where $\nabla^\perp = (\partial_1^\perp,\partial_2^\perp)^{\mathrm{T}} = (-\partial_2,\partial_1)^{\mathrm{T}}$. 
We need only to estimate the integral part. Denote by
\begin{equation}
  d(x_0,t):= \inf_{x\in\partial D(t)} \{|x-x_0|\},\quad \delta(t):= \min\bigg\{1,\bigg( \frac{\|\nabla \varphi(t)\|_{\inf}}{\|\nabla\varphi(t)\|_{\dot C^\gamma}} \bigg)^{1/\gamma}\bigg\},
\end{equation}
with $\|\nabla \varphi(t)\|_{\inf}:= \inf_{x\in \partial D(t)} |\nabla \varphi(x,t)|$.
Notice that $\varphi(x,t)$ belongs to $L^\infty(0,T; C^{1,\gamma})$ according to \eqref{nab-varp-es},
and it also satisfies that $\|\nabla \varphi(t)\|_{\inf} \geq \|\nabla\varphi_0\|_{\inf} e^{- \int_0^t \|\nabla u\|_{L^\infty} \dd \tau}$ (e.g. see \cite[Eq. (2.26)]{BerC}),
thus we deduce that for every $t\in [0,T]$,
\begin{align*}
  \delta(t)^{-1} \leq \|\nabla \varphi\|_{L^\infty_T (C^\gamma)}^{\frac{1}{\gamma}} \|\nabla \varphi_0\|_{\inf}^{-\frac{1}{\gamma}}\,
  e^{\gamma^{-1}\int_0^T \|\nabla u\|_{L^\infty}\dd \tau}
  \leq C e^{C(1+T)^2}.
\end{align*}
Now we split the integral term in \eqref{eq:SIO-exp} as
\begin{align}
  &\int_{D(t)\cap \{|x_0-y|\geq \delta(t)\}} \frac{\sigma_{ijk}(x_0 -y)}{|x_0 -y|^2} \bar{\theta}_0(X^{-1}_t(y)) \dd y
  + \mathrm{p.v.}\int_{D(t)\cap \{|x_0 -y|\leq \delta(t)\}} \frac{\sigma_{ijk}(x_0 -y)}{|x_0 -y|^2} \bar{\theta}_0(X^{-1}_t(y)) \dd y\nonumber \\
  & =: I_1(x_0,t) + I_2(x_0,t).
\end{align}
Since the area of patch domain $D(t)= X_t(D_0)$ remains constant in time, 
we directly obtain
\begin{align}\label{eq:I1-es}
  |I_1(x_0,t)| \leq C_0 \|\bar{\theta}_0\|_{L^\infty(\overline{D_0})}|D(t)| \delta(t)^{-2}
  \leq C |D_0| \delta(t)^{-2}\leq C e^{C(1+T)^2}.
\end{align}

For the estimation of $I_2$, if $x_0\in \overline{D(t)}$, we decompose $I_2$ as
\begin{align*}
  I_2(x_0,t) = \,& \mathrm{p.v.}\int_{D(t)\cap \{|x_0 -y|\leq \delta(t)\}} \frac{\sigma_{ijk}(x_0 -y)}{|x_0 -y|^2} \big(\bar{\theta}_0(X^{-1}_t(y))- \bar{\theta}_0(X^{-1}_t(x_0))\big) \dd y \\
  & + \bar{\theta}_0(X^{-1}_t(x_0)) \bigg(\mathrm{p.v.}\int_{D(t)\cap \{|x_0 -y|\leq \delta(t)\}} \frac{\sigma_{ijk}(x_0 -y)}{|x_0 -y|^2} \dd y \bigg) = : I_{21}(x_0,t) + I_{22}(x_0,t);
\end{align*}
on the other hand, if $x_0\notin \overline{D(t)}$, denote by $\tilde{x}_t\in \partial D(t)$ a point so that $|x_0 -\tilde{x}_t|=d(x_0,t)$, we have
\begin{align*}
  I_2(x_0,t) = \,& \mathrm{p.v.}\int_{D(t)\cap \{|x_0 -y|\leq \delta(t)\}} \frac{\sigma_{ijk}(x_0 -y)}{|x_0 -y|^2} \big(\bar{\theta}_0(X^{-1}_t(y))- \bar{\theta}_0(X^{-1}_t(\tilde{x}_t))\big) \dd y \\
  & + \bar{\theta}_0(X^{-1}_t(\tilde{x}_t)) \bigg(\mathrm{p.v.}\int_{D(t)\cap \{|x_0 -y|\leq \delta(t)\}} \frac{\sigma_{ijk}(x_0 -y)}{|x_0 -y|^2} \dd y \bigg) = : I_{23}(x_0,t) + I_{24}(x_0,t).
\end{align*}
Due to that $\bar{\theta}_0\in C^\mu(\overline{D_0})$, $0<\mu<1$, one directly gets
\begin{align}\label{I21-es}
  |I_{21}(x_0,t)| \leq  C \|\bar\theta_0\|_{C^\mu(\overline{D_0})} \|\nabla X^{-1}_t\|_{L^\infty}^\mu \int_{B_1(x_0)} \frac{1}{|x_0 -y|^{2-\mu}} \dd y
  \leq C e^{C(1+T)^2},
\end{align}
and
\begin{align}\label{I23-es}
  |I_{23}(x_0,t)| & \leq  C \|\bar\theta_0\|_{C^\mu(\overline{D_0})} \|\nabla X^{-1}_t\|_{L^\infty}^\mu \int_{B_1(x_0)} \frac{1}{|x_0 -y|^2} |\tilde{x}_t - y|^\mu \dd y \nonumber \\
  & \leq C \|\bar\theta_0\|_{C^\mu(\overline{D_0})} e^{\mu \int_0^T \|\nabla u\|_{L^\infty} \dd t} \int_{B_1(x_0)} \frac{1}{|x_0 -y|^{2-\mu}} \dd y \leq C e^{C(1+T)^2},
\end{align}
where in the last line we used the fact $|\tilde{x}_t - y|\leq |\tilde{x}_t -x_0| + |x_0-y| \leq 2 |x_0-y|$.
On the other hand, due to the zero-mean of $\sigma_{ijk}$, the principal-value integral in $I_{22}(x_0,t)$ and $I_{24}(x_0,t)$ vanishes if $d(x_0,t)> \delta(t)$,
so it suffices to consider the case $d(x_0,t)\leq \delta(t)$.
We then define
\begin{align*}
  \Sigma(x_0,t) := \{z|\,|z|=1, \nabla_x \varphi(\tilde{x}_t,t)\cdot z \geq 0\},
\end{align*}
and for every $\rho$ satisfying $\rho \geq d(x_0,t)$, define
\begin{align}
  S_\rho(x_0,t) & : =\{z|\, |z|=1,x_0 +\rho z \in D(t)\}, \nonumber\\
  R_\rho(x_0,t) & := \big(S_\rho(x_0,t)\setminus \Sigma(x_0,t)\big) \cup \big(\Sigma(x_0,t)\setminus S_\rho(x_0,t)\big). \label{eq:Rrho}
\end{align}
In terms of the polar coordinates centered at $x_0$, and using the fact that $\int_{\Sigma(x_0,t)} \sigma_{ijk}(z) \dd \mathcal{H}^1(z) =0$
we find
\begin{equation}\label{eq:I2-es1}
  |I_{22}(x_0,t)|, |I_{24}(x_0,t)|  \leq C_0 \|\bar{\theta}_0\|_{L^\infty(\overline{D_0})} \int_{d(x_0,t)}^{\delta(t)} \frac{\mathcal{H}^1(R_\rho(x_0,t))}{\rho} \dd \rho,
\end{equation}
with $\mathcal{H}^1$ the Hausdorff measure on the unit circle.
Concerning $\mathcal{H}^1(R_\rho)$, we recall the following ingenious result (for the proof see \cite[Geometric Lemma]{BerC}).
\begin{lemma}\label{lem:geo-lem}
  Let $R_\rho(x_0,t)$ be the symmetric difference defined in \eqref{eq:Rrho},
then for all $\rho\geq d(x_0,t)$, $\gamma\in (0,1)$ and $x_0\in\R^2$ so that $d(x_0,t) \leq \delta(t)$, we have
\begin{equation}\label{eq:geo-lem}
  \mathcal{H}^1(R_\rho(x_0,t)) \leq 2\pi \bigg((1+ 2^\gamma) \frac{d(x_0,t)}{\rho} + 2^\gamma \Big(\frac{\rho}{\delta(t)}\Big)^\gamma \bigg).
\end{equation}
\end{lemma}
\noindent Thus inserting \eqref{eq:geo-lem} into \eqref{eq:I2-es1} leads to
\begin{equation}\label{eq:I2-es2}
  |I_{22}(x_0,t)|, |I_{24}(x_0,t)|  \leq 2\pi C_0  \|\bar{\theta}_0\|_{L^\infty(\overline{D_0})}  \Big((1+ 2^\gamma) + \frac{2^\gamma}{\gamma}\Big).
\end{equation}
Hence, gathering \eqref{eq:SIO-exp}, \eqref{eq:I1-es}--\eqref{I23-es} and \eqref{eq:I2-es2} ensures the assertion \eqref{eq:claim1}.

Therefore, by combining \eqref{Gam-Lip-es} with \eqref{eq:claim1} we get the desired estimate that for every $\rho\in [1,\frac{2p}{p+2})$,
\begin{align}\label{es:nab2u-Linf}
  \|\nabla^2 u\|_{L^\rho_T (L^\infty)} \leq \|\nabla^2 \nabla^\perp \Lambda^{-2} \Gamma\|_{L^\rho_T(L^\infty)} + \|\nabla^2 \nabla^\perp \partial_1 \Lambda^{-4} \theta\|_{L^\infty} \leq C e^{C(1+T)^2}.
\end{align}

\subsection{Persistence of $C^{2,\gamma}$-boundary regularity}\label{subsec:C2gam}

In this subsection we dedicate to proving the persistence of $C^{2,\gamma}$-regularity of the temperature front boundary.

Observing that $W=\nabla^\perp \varphi$ satisfies
\begin{equation}\label{W-eq}
  \partial_t W + u\cdot\nabla W = W\cdot\nabla u = \partial_W u,\quad W(0,x) = W_0(x) ,
\end{equation}
we infer that
\begin{align}\label{nabW-eq}
  \partial_t \nabla W + u\cdot \nabla (\nabla W) = \partial_W \nabla u + \nabla W\cdot \nabla u - \nabla u \cdot \nabla W.
\end{align}
Owing to \eqref{eq:T-sm2} and $B^{\gamma}_{\infty,\infty}(\R^2)\subset L^\infty(\R^2)$, we obtain
\begin{align}\label{nab-W-Cgam}
  \|\nabla W(t)\|_{B^\gamma_{\infty,\infty}}
  \leq C \bigg(\|\nabla W_0\|_{B^\gamma_{\infty,\infty}} + \int_0^t \|\partial_W\nabla u\|_{B^\gamma_{\infty,\infty}} \dd \tau
  + \int_0^t \|\nabla u\|_{B^\gamma_{\infty,\infty}} \|\nabla W\|_{B^\gamma_{\infty,\infty}} \dd \tau \bigg). &
\end{align}

Below we focus on estimating $\partial_W\nabla u$.
From the Biot-Savart law and relation $\omega =\Gamma + \RR_{-1}\theta =\Gamma + \partial_1 \Lambda^{-2}\theta$, it follows that
\begin{equation}\label{W-nab2u-decom}
  \partial_W \nabla u = \partial_W \nabla \nabla^\perp \Lambda^{-2} \omega = \partial_W \nabla\nabla^\perp \Lambda^{-2} \Gamma + \partial_W \nabla \nabla^\perp \partial_1 \Lambda^{-4} \theta.
\end{equation}
We first consider the estimation of $\partial_W \Gamma = W\cdot\nabla \Gamma$. Note that from equation \eqref{Gamm-eq}
and the fact that $[\partial_W,\partial_t + u\cdot\nabla]=0$, $\partial_W \Gamma$ solves the following equation
\begin{equation}\label{eq:parW-Gam}
\begin{split}
  \partial_t (\partial_W \Gamma) + u\cdot \nabla (\partial_W \Gamma) - \Delta (\partial_W \Gamma) & = - [\Delta,\partial_W]\Gamma + \partial_W ([\RR_{-1},u\cdot\nabla]\theta) \\
  & = - \Delta W\cdot \nabla \Gamma -2 \nabla W\cdot \nabla^2\Gamma + \partial_W ([\RR_{-1},u\cdot\nabla]\theta).
\end{split}
\end{equation}
According to the smoothing estimate \eqref{eq:TD-sm2}, we find that for every $0<\gamma' < \min\{\gamma, 1-\frac{2}{p}\}$,
\begin{align}
  & \quad \|\partial_W \Gamma \|_{L^\infty_t (B^{\gamma'-1}_{\infty,1})} + \|\partial_W \Gamma\|_{L^1_t (B^{\gamma'+1}_{\infty,1})} \nonumber\\
  & \leq C (1+t) \bigg(\|\partial_{W_0}\Gamma_0\|_{B^{\gamma'-1}_{\infty,1}}
  + \int_0^t \|\nabla u(\tau)\|_{L^\infty} \|\partial_W \Gamma(\tau)\|_{B^{\gamma'-1}_{\infty,1}} \dd \tau
  + \|\Delta W \cdot \nabla \Gamma\|_{L^1_t(B^{\gamma'-1}_{\infty,1})}  \nonumber \\
  & \qquad \qquad \qquad + \|\nabla W \cdot\nabla^2 \Gamma\|_{L^1_t(B^{\gamma'-1}_{\infty,1})} + \|\partial_W ([\RR_{-1},u\cdot\nabla]\theta)\|_{L^1_t(B^{\gamma'-1}_{\infty,1})} \bigg).
\end{align}
In view of $\Gamma_0 = \omega_0 + \RR_{-1}\theta_0 = \omega_0 + \partial_1 \Lambda^{-2}\theta_0$ and the embedding $W^{1,p}\subset B^{\gamma'}_{\infty,1}$
with $0<\gamma' <1-\frac{2}{p}$, we get
\begin{align}
  \|\partial_{W_0} \Gamma_0\|_{B^{\gamma'-1}_{\infty,1}} & \leq \|\partial_{W_0} \nabla u_0\|_{B^{\gamma'-1}_{\infty,1}}
  + \|\partial_{W_0} \RR_{-1}\theta_0\|_{B^{\gamma'-1}_{\infty,1}} \nonumber \\
  & \leq C \|\partial_{W_0} u_0\|_{B^{\gamma'}_{\infty,1}} + C \|\nabla W_0\|_{L^\infty} \|\nabla u_0\|_{B^{\gamma'-1}_{\infty,1}}
  + C \|W_0\|_{L^\infty} \|\nabla \RR_{-1} \theta_0\|_{B^{\gamma'-1}_{\infty,1}} \nonumber \\
  & \leq C \|\partial_{W_0} u_0\|_{W^{1,p}} + C \|\varphi_0\|_{W^{2,\infty}} \|u_0\|_{W^{1,p}} + C \|\varphi_0\|_{W^{1,\infty}} \|\theta_0\|_{L^2 \cap L^\infty} <\infty.
\end{align}
Notice that by virtue of \eqref{Gam-es2},
\begin{align}\label{es:Gam-L1TBes}
  \|\Gamma\|_{L^1_T(B^{\gamma'+1}_{\infty,1})} & \leq C \|\Delta_{-1}\Gamma\|_{L^1_T(L^\infty)} + C \sum_{q\in \N} 2^{q(1+\gamma' + \frac{2}{p})} \|\Delta_q \Gamma\|_{L^1_T(L^p)} \nonumber \\
  & \leq C T \|\Gamma\|_{L^\infty_T(L^p)} + C \sum_{q\in \N} 2^{q(\gamma'+\frac{2}{p} -1)}  (1+ T)^3  \leq C (1+ T)^3,
\end{align}
thus \eqref{eq:prod-es} and \eqref{eq:nab2-vap-es} ensure that
\begin{align*}
  \|\nabla W \cdot \nabla^2 \Gamma\|_{L^1_t(B^{\gamma'-1}_{\infty,1})} & \leq C \|\nabla W\|_{L^\infty_t(L^\infty)} \|\nabla^2 \Gamma\|_{L^1_t(B^{\gamma'-1}_{\infty,1})}
  \leq C e^{C(1+t)^2},
\end{align*}
\begin{align*}
  \|\Delta W \cdot \nabla \Gamma\|_{L^1_t(B^{\gamma'-1}_{\infty,1})}
  \leq C \int_0^t \|\Delta W(\tau)\|_{B^{\gamma'-1}_{\infty,1}} \|\nabla \Gamma(\tau)\|_{L^\infty} \dd \tau.
\end{align*}
Taking advantage of \eqref{eq:comm-es1}, \eqref{u-W1p-es} and the embedding $B^1_{p,\infty} \subset B^{\gamma'}_{\infty,1}$ for every $0<\gamma' < 1-\frac{2}{p}$, we deduce that
\begin{align}\label{es:R-1comm1}
  \|[\RR_{-1},u\cdot\nabla]\theta\|_{L^1_t (B^1_{p,\infty})} \leq C \big(\|\nabla u\|_{L^1_t(L^p)} \|\theta\|_{L^\infty_t(L^\infty)} + t \|u\|_{L^\infty_t(L^2)}\|\theta\|_{L^\infty_t(L^2)} \big) \leq C (1+t)^2,
\end{align}
and
\begin{align*}
  \|\partial_W ([\RR_{-1},u\cdot\nabla]\theta)\|_{L^1_t(B^{\gamma'-1}_{\infty,1})} \leq C \|W\|_{L^\infty_t(L^\infty)} \|[\RR_{-1},u\cdot\nabla]\theta\|_{L^1_t(B^{\gamma'}_{\infty,1})}  \leq  C e^{C(1+ t)^2} .
\end{align*}
Gathering the above estimates yields
\begin{align}\label{es:parW-Gam1}
  & \quad \|\partial_W \Gamma(t)\|_{B^{\gamma'-1}_{\infty,1}} + \|\partial_W \Gamma\|_{L^1_t(B^{\gamma'+1}_{\infty,1})} \nonumber \\
  & \leq C e^{C (1+ t)^2}
  + C \int_0^t \Big( \|\nabla u(\tau)\|_{L^\infty} \|\partial_W \Gamma(\tau)\|_{B^{\gamma'-1}_{\infty,1}} + \| W(\tau)\|_{B^{\gamma'+1}_{\infty,1}} \|\nabla \Gamma(\tau)\|_{L^\infty}\Big) \dd \tau .
\end{align}
Thanks to estimates \eqref{eq:prod-es}, \eqref{eq:str-es1} and \eqref{es:parW-Gam1}, we also infer that
\begin{align}\label{es:parW-Gam2}
  & \quad \|\partial_W(\nabla \nabla^\perp \Lambda^{-2}\Gamma)\|_{L^1_t(B^{\gamma}_{\infty,\infty})} \nonumber \\
  & \lesssim \|\Delta_{-1} \partial_W(\nabla \nabla^\perp \Lambda^{-2}\Gamma)\|_{L^1_t(L^\infty)}
  + \|\nabla \partial_W (\nabla \nabla^\perp \Lambda^{-2}\Gamma)\|_{L^1_t(B^{\gamma-1}_{\infty,\infty})} \nonumber \\
  & \lesssim \|W\|_{L^1_t (L^\infty)} \|\nabla \nabla^\perp \Lambda^{-2}\Gamma\|_{L^1_t (L^p)} +  \|\nabla W\|_{L^\infty_t(L^\infty)} \|\nabla^2 \nabla^\perp \Lambda^{-2} \Gamma\|_{L^1_t (B^{\gamma-1}_{\infty,\infty})}
  +  \|\partial_W(\nabla^2 \nabla^\perp \Lambda^{-2} \Gamma)\|_{L^1_t(B^{\gamma-1}_{\infty,\infty})} \nonumber \\
  & \lesssim \|W\|_{L^1_t (L^\infty)} \|\Gamma\|_{L^1_t (L^p)}  + \|\nabla W\|_{L^\infty_t(L^\infty)} \|\Gamma\|_{L^1_t (B^{\gamma}_{\infty,\infty})} \nonumber \\
  & \quad + \|\partial_W \nabla \Gamma\|_{L^1_t(B^{\gamma-1}_{\infty,\infty})} + \int_0^t \|W\|_{B^1_{\infty,1}}
  \big(\|\Delta_{-1}\nabla^2 \nabla^\perp \Lambda^{-2} \Gamma\|_{L^\infty} + \|\nabla\Gamma\|_{B^{\gamma-1}_{\infty,\infty}}\big) \dd \tau \nonumber \\
  & \lesssim  e^{C(1+t)^2} + \|\partial_W \Gamma\|_{L^1_t(B^\gamma_{\infty,\infty})} + \int_0^t \|W\|_{B^1_{\infty,1}} \big( \|\Gamma\|_{L^p} + \|\Gamma\|_{B^\gamma_{\infty,\infty}}\big) \dd \tau \nonumber \\
  & \leq C  e^{C(1+t)^2} + C \int_0^t \big(\|W\|_{B^{\gamma'+1}_{\infty,1}} + \|\partial_W \Gamma\|_{B^{\gamma'-1}_{\infty,1}}\big)
  \big( \|\nabla u(\tau)\|_{L^\infty} + \|\Gamma(\tau)\|_{L^p \cap W^{1,\infty}} \big) \dd \tau.
\end{align}

For the estimation of the $\theta$-term in \eqref{W-nab2u-decom}, the argument is similar to that of obtaining \eqref{es:parW-Gam2}:
\begin{align}\label{es:parW-the0}
  & \quad \|\partial_W (\nabla \nabla^\perp \partial_1 \Lambda^{-4}\theta)\|_{L^1_t(B^\gamma_{\infty,\infty})} \nonumber \\
  & \lesssim \|\Delta_{-1} \partial_W (\nabla \nabla^\perp \partial_1 \Lambda^{-4} \theta)\|_{L^1_t(L^\infty)}
  + \|\nabla \partial_W (\nabla \nabla^\perp \partial_1 \Lambda^{-4} \theta)\|_{L^1_t(B^{\gamma-1}_{\infty,\infty})} \nonumber \\
  & \lesssim \|W\|_{L^\infty_t(L^\infty)} \|\theta\|_{L^1_t(L^2)} + \|\nabla W\|_{L^\infty_t(L^\infty)} \|\nabla^2 \nabla^\perp \partial_1 \Lambda^{-4}\theta\|_{L^1_t(B^{\gamma-1}_{\infty,\infty})}
  + \|\partial_W (\nabla^2 \nabla^\perp \partial_1 \Lambda^{-4}\theta)\|_{L^1_t(B^{\gamma-1}_{\infty,\infty})} \nonumber \\
  & \lesssim \|W\|_{L^\infty_t(W^{1,\infty})} \|\theta\|_{L^1_t(L^2\cap L^\infty)} + \|\partial_W \theta\|_{L^1_t(B^{\gamma-1}_{\infty,\infty})} +
  \int_0^t \|W\|_{B^1_{\infty,1}} (\|\theta\|_{L^2} + \|\theta\|_{B^{\gamma-1}_{\infty,\infty}}) \dd \tau \nonumber \\
  & \lesssim e^{C(1+t)^2} + \|\partial_W \theta\|_{L^1_t(B^{\gamma-1}_{\infty,\infty})} + \int_0^t \|W(\tau)\|_{B^1_{\infty,1}} \|\theta(\tau)\|_{L^2\cap L^\infty} \dd \tau .
\end{align}
Since the operator $\partial_W = W\cdot\nabla$ commutates with $\partial_t + u\cdot \nabla$, we see that
\begin{equation}\label{eq:par-W-the}
  \partial_t \partial_W \theta + u\cdot \nabla \partial_W \theta =0,\qquad \partial_W \theta|_{t=0} = \partial_{W_0}\theta_0,
\end{equation}
and then the regularity preservation estimate \eqref{eq:T-sm3} ensures that
\begin{equation}\label{es:parWthe1}
  \|\partial_W \theta(t)\|_{B^{\gamma-1}_{\infty,\infty}} \leq e^{C\int_0^t \|\nabla u\|_{L^\infty}\dd\tau} \|\partial_{W_0}\theta_0\|_{C^{-1,\gamma}}
  \leq \|\partial_{W_0}\theta_0\|_{C^{-1,\gamma}} e^{C(1+t)^3} .
\end{equation}
In the above $\partial_{W_0}\theta_0\in C^{-1,\gamma}(\R^2)$ is guaranteed by Lemma \ref{lem:str-reg}.

Therefore, by collecting estimates \eqref{nab-W-Cgam}, \eqref{es:parW-Gam1}--\eqref{es:parW-the0} and \eqref{es:parWthe1},
it follows from the embedding $B^{\gamma+1}_{\infty,\infty}\subset B^{\gamma'+1}_{\infty,1}$ (for every $0<\gamma' <\min\{\gamma, 1-\frac{2}{p}\}$)
that, 
\begin{align}
  & \quad \|W(t)\|_{B^{\gamma+1}_{\infty,\infty}} +  \|\partial_W \Gamma(t)\|_{B^{\gamma'-1}_{\infty,1}} + \|\partial_W \Gamma\|_{L^1_t(B^{\gamma'+1}_{\infty,1})} + \|\partial_W \nabla u\|_{L^1_t (B^\gamma_{\infty,\infty})} \nonumber \\
  & \leq C e^{C(1+ t)^2} +
  C \int_0^t \big(\|W\|_{B^{\gamma+1}_{\infty,\infty}} + \|\partial_W \Gamma\|_{B^{\gamma'-1}_{\infty,1}}\big)
  \big(1+ \|\nabla u(\tau)\|_{B^\gamma_{\infty,\infty}} + \|\Gamma(\tau)\|_{L^p \cap W^{1,\infty}} \big) \dd \tau .
\end{align}
Gronwall's inequality and \eqref{eq:uC1-gam}, \eqref{es:Gam-Lp}, \eqref{es:Gam-L1TBes} guarantee that
\begin{align}\label{nab-W-Cgam-es}
  \|W\|_{L^\infty_T(B^{\gamma+1}_{\infty,\infty})} + \|\partial_W \Gamma \|_{L^1_T(B^{\gamma'-1}_{\infty,1})} + \|\partial_W \Gamma\|_{L^1_T(B^{\gamma'+1}_{\infty,1})} + \|\partial_W \nabla u\|_{L^1_T (B^\gamma_{\infty,\infty})}
  \leq C e^{C (1+T)^3},
\end{align}
which directly implies $\varphi \in L^\infty(0,T; C^{2,\gamma}(\R^2))$, as desired.

In terms of the notations \eqref{norm:BBsln2}-\eqref{eq:abbr}, the estimates \eqref{eq:uC1-gam}, \eqref{es:Gam-Lp}, \eqref{es:Gam-L1TBes}, \eqref{nab-W-Cgam-es} imply that
\begin{align}\label{eq:W-Gam-step1}
  & \quad \|W\|_{L^\infty_T(\bb^{\gamma+1,0}_{\infty,W})} + \|\nabla u\|_{L^1_T(\bb^{\gamma,1}_{\infty,W})}
  + \|\Gamma\|_{L^\infty_T(\bb^{\gamma'-1,1}_W)} + \|\Gamma\|_{L^1_T(\bb^{\gamma'+1,1}_W)}  \nonumber \\
  & = \|W\|_{L^\infty_T(B^{\gamma+1}_{\infty,\infty})} + \|\nabla u\|_{L^1_T(B^{\gamma}_{\infty,\infty})}
  + \|\partial_W \nabla u\|_{L^1_T(B^{\gamma}_{\infty,\infty})} \nonumber \\
  & \quad + \|\Gamma\|_{L^\infty_T(B^{\gamma'-1}_{\infty,1})} + \|\partial_W \Gamma\|_{L^\infty_T(B^{\gamma'-1}_{\infty,1})}
  + \|\Gamma\|_{L^1_T(B^{\gamma'+1}_{\infty,1})} + \|\partial_W \Gamma\|_{L^1_T(B^{\gamma'+1}_{\infty,1})} \nonumber \\
  & \leq C e^{C (1+T)^3}.
\end{align}

\section{Persistence of $C^{k,\gamma}$-boundary regularity with $k\in \N\cap [3,\infty)$}\label{sec:Ck-gam}

In this section, assuming that $\partial D_0\in C^{k,\gamma}(\R^2)$ and $\bar{\theta}_0\in C^{k-2,\gamma}(\overline{D_0})$
for every $k \in \N \cap [ 3,\infty)$ and $\gamma\in (0,1)$, we intend to prove
\begin{align}\label{eq:goal0}
  \partial D(t)\in  C^{k,\gamma},\quad \forall t\in [0,T].
\end{align}

In order to study the higher regularity of front boundary $\partial D(t)$, we first establish its deep connection with the striated regularity of
$W=\nabla^\perp \varphi$ (e.g. see \cite{Chem91,LZ16}). Recalling that $\partial D_0$ has the parameterization \eqref{patch-para-exp},
$\partial D(t)$ can thus be expressed as $X_t(z_0(\alpha))$ with $\alpha\in \mathbb{S}^1$ and $X_t$ the particle trajectory given by \eqref{eq:flow0},
then
\begin{equation}
  \partial_\alpha \big( X_t(z_0(\alpha)) \big)= W_0(z_0(\alpha)) \cdot \nabla_x X_t(z_0(\alpha)) = (\partial_{W_0} X_t)(z_0(\alpha)) .
\end{equation}
On the other hand, noting that $W=\nabla^\perp \varphi$ solves equation \eqref{W-eq}, Lemma 1.4 of \cite{MB02} ensures
\begin{equation}\label{W-rela1}
  W(X_t(x),t) = W_0(x)\cdot \nabla X_t(x) = \big(\partial_{W_0} X_t\big)(x).
\end{equation}
Combining the above two formulas leads to
\begin{equation}
  \partial_\alpha \big(X_t(z_0(\alpha))\big) = W(X_t(z_0(\alpha)),t) = \big(\partial_{W_0}X_t\big)(z_0(\alpha)).
\end{equation}
Moreover, by iteration, it follows that for any $k\in\Z^+ $,
\begin{align}\label{par-k-bdr-rela}
  \partial^k_\alpha \big(X_t(z_0(\alpha))\big) = \partial_\alpha^{k-1}\big((\partial_{W_0} X_t)(z_0(\alpha)) \big)= \cdots = (\partial_{W_0}^k X_t)(z_0(\alpha)) .
\end{align}
From \eqref{W-rela1}, we get
\begin{align*}
  \big(\partial_{W_0}^2 X_t\big)(x) = \partial_{W_0}X_t(x)\,\cdot (\nabla W)(X_t(x),t)
  = W(X_t(x),t)\cdot(\nabla W)(X_t(x),t)= (\partial_W W)(X_t(x),t),
\end{align*}
and by induction,
\begin{align}\label{W-rela2}
  \big(\partial_{W_0}^k X_t\big)(x) = (\partial_W^{k-1} W)(X_t(x),t), \quad \forall k\in \Z^+.
\end{align}

Hence in light of \eqref{par-k-bdr-rela} and \eqref{W-rela2}, in order to prove the persistence result \eqref{eq:goal0},
that is, $\partial_\alpha^k\big(X_t(z_0(\alpha))\big)\in C^\gamma$, it suffices to show that
\begin{align}\label{eq:target0}
  \big(\partial_{W_0}^k X_t\big)(x) \in L^\infty(0,T; C^\gamma(\R^2)),
\end{align}
which in turn remains to verify that (thanks to Lemma \ref{lem:flow}) 
\begin{equation}\label{eq:target}
  \big(\partial_W^{k-1} W\big)(\cdot,t) \in L^\infty(0,T; C^\gamma(\R^2)),\quad \forall k\in\N\cap [3,\infty),\gamma\in(0,1).
\end{equation}
\vskip0.15cm

In the following we mainly will prove that
\begin{align}\label{eq:Targ-k}
  \| W\|_{L^\infty_T (\bb^{\gamma+1,k-2}_{\infty,W})}  + \|\nabla u\|_{L^1_T ( \bb^{\gamma,k-1}_{\infty,W})}
  & + \|\Gamma\|_{L^\infty_T (\bb^{\gamma'-1,k-1}_W)} + \|\Gamma\|_{L^1_T(\bb^{\gamma'+1,k-1}_W)}  \leq H_{k-1}(T),
\end{align}
with $H_{k-1}(T)$ depending on $k-1$ and $T$. 
A direct consequence of \eqref{eq:Targ-k} is that
\begin{align*}
  \|\partial_W^{k-1} W\|_{L^\infty_T (C^\gamma)} \leq \|W\cdot\nabla \partial_W^{k-2} W\|_{L^\infty_T (C^\gamma)}
  \lesssim \|W\|_{L^\infty_T (C^\gamma)}
  \|\partial_W^{k-2} W\|_{L^\infty_T (B^{\gamma+1}_{\infty,\infty})} \lesssim H_{k-1}(T),
\end{align*}
which corresponds to the wanted result \eqref{eq:target} with $k\in \N\cap [3,\infty)$.

In order to show the target estimate \eqref{eq:Targ-k},
we apply the induction method. Assume that for some $\ell \in \{1,\cdots, k-2\}$, we have
\begin{align}\label{assum-el}
  \| W\|_{L^\infty_T (\bb^{\gamma+1,\ell-1}_{\infty,W})}  + \|\nabla u\|_{L^1_T ( \bb^{\gamma,\ell}_{\infty,W})}
  + \|\Gamma\|_{L^\infty_T (\bb^{\gamma'-1,\ell}_W)}  + \|\Gamma\|_{L^1_T(\bb^{\gamma'+1,\ell}_W)}  \leq H_\ell(T),
\end{align}
we intend to prove that it also holds for $\ell$ replaced by $\ell+1$, that is,
\begin{align}\label{eq:target-ell}
  \| W\|_{L^\infty_T (\bb^{\gamma+1,\ell}_{\infty,W})}  + \|\nabla u\|_{L^1_T ( \bb^{\gamma,\ell+1}_{\infty,W})}
  + \|\Gamma\|_{L^\infty_T (\bb^{\gamma'-1,\ell+1}_W)}
  + \|\Gamma\|_{L^1_T(\bb^{\gamma'+1,\ell+1}_W)}  \leq H_{\ell+1}(T).
\end{align}

Note that estimate \eqref{eq:W-Gam-step1} in subsection \ref{subsec:C2gam} corresponds to \eqref{assum-el} with $\ell=1$.
Notice also that under \eqref{assum-el}, Lemma \ref{lem:prod-es2} can be applied with the $k$-index replaced by $\ell$.

We start with the estimation of the $L^\infty_T(B^{\gamma-1}_{\infty,\infty})$-norm of $\partial_W^\ell \nabla^2 W$.
From equation \eqref{nabW-eq} and the fact that $[\partial_W, \partial_t + u\cdot\nabla]=0$,
we see that
\begin{equation}\label{par-W-Wel-eq}
\begin{split}
  \partial_t (\partial_W^\ell \nabla^2 W) + u\cdot\nabla (\partial_W^\ell \nabla^2 W) & = \partial_W^{\ell+1} \nabla^2 u + 2 \partial_W^\ell (\nabla W\cdot\nabla^2 u) + \partial_W^\ell (\nabla^2 W\cdot\nabla u) \\
  &\quad - \partial_W^\ell (\nabla^2 u\cdot\nabla W) - \partial_W^\ell (\nabla u \cdot\nabla^2 W) .
\end{split}
\end{equation}
Thanks to Lemma \ref{lem:TD-sm2}, we find that for every $\gamma\in (0,1)$,
\begin{align}\label{es:parW2W-3}
  \|\partial_W^\ell \nabla^2 W(t)\|_{B^{\gamma-1}_{\infty,\infty}}
  \leq & C \| \partial_{W_0}^\ell \nabla^2 W_0\|_{B^{\gamma-1}_{\infty,\infty}}
  + C \int_0^t \|\nabla u\|_{L^\infty} \| \partial_W^\ell \nabla^2 W\|_{B^{\gamma-1}_{\infty,\infty}} \dd \tau
  \nonumber \\
  & + C \int_0^t \|\partial_W^{\ell+1} \nabla^2 u\|_{B^{\gamma-1}_{\infty,\infty}} \dd \tau
  + C \int_0^t \|\big(\partial_W^\ell (\nabla W \cdot \nabla^2 u), \partial_W^\ell (\nabla^2 u\cdot\nabla W)\big)\|_{B^{\gamma-1}_{\infty,\infty}} \dd \tau \nonumber \\
  & + C \int_0^t \|\big(\partial_W^\ell (\nabla^2 W\cdot\nabla u),
  \partial_W^\ell (\nabla u\cdot\nabla^2 W)\big)\|_{B^{\gamma-1}_{\infty,\infty}} \dd \tau .
\end{align}
From $\varphi_0\in C^{k,\gamma}(\R^2)$, 
and by arguing as \eqref{es:parW0-k-the} we get
\begin{align}
  \|\partial_{W_0}^\ell \nabla^2 W_0\|_{B^{\gamma-1}_{\infty,\infty}} \lesssim_{\|W_0\|_{W^{\ell-1,\infty}}}
  \|\nabla^2 W_0\|_{B^{\ell-1+\gamma}_{\infty,\infty}} \lesssim \|\varphi_0\|_{B^{\ell+2+\gamma}_{\infty,\infty}}
  \lesssim \|\varphi_0\|_{B^{k+\gamma}_{\infty,\infty}}.
\end{align}
Taking advantage of Lemma \ref{lem:prod-es2} and \eqref{eq:BB-nab-es}, \eqref{normBB-equiv}, \eqref{eq:W-BBbb-equ}, the last two integrals on the right-hand side of \eqref{es:parW2W-3} can be treated as follows
\begin{align}\label{es:parW2W-4}
  & \quad \int_0^t \|\big(\partial_W^\ell (\nabla W \cdot \nabla^2 u), \partial_W^\ell (\nabla^2 u\cdot\nabla W)\big)\|_{B^{\gamma-1}_{\infty,\infty}} \dd \tau \nonumber \\
  & \lesssim  \int_0^t \|(\nabla W \cdot \nabla^2 u, \nabla^2 u\cdot\nabla W)\|_{\bb^{\gamma-1,\ell}_{\infty,W}} \dd \tau \nonumber \\
  & \lesssim \int_0^t \|\nabla W\|_{\bb^{0,\ell}_W} \|\nabla^2 u\|_{\bb^{\gamma-1,\ell}_{\infty,W}} \dd \tau
  \lesssim \int_0^t \| W\|_{\bb^{1,\ell}_W} \|\nabla u\|_{\bb^{\gamma,\ell}_{\infty,W}} \dd \tau,
\end{align}
and
\begin{align}\label{es:parW2W-5}
  & \quad \int_0^t \|\big(\partial_W^\ell (\nabla^2 W \cdot \nabla u), \partial_W^\ell (\nabla u\cdot\nabla^2 W)\big)\|_{B^{\gamma-1}_{\infty,\infty}} \dd \tau \nonumber \\
  & \lesssim  \int_0^t \|(\nabla^2 W \cdot \nabla u, \nabla u\cdot\nabla^2 W)\|_{\bb^{\gamma-1}_{\infty,W}} \dd \tau \nonumber \\
  & \lesssim \int_0^t \|\nabla^2 W\|_{\bb^{\gamma-1,\ell}_{\infty,W}} \|\nabla u\|_{\bb^{0,\ell}_W} \dd \tau
  \lesssim \int_0^t \|W\|_{\bb^{\gamma+1,\ell}_{\infty,W}} \|\nabla u\|_{\bb^{\gamma,\ell}_{\infty,W}} \dd \tau.
\end{align}

For the third term on the right-hand side of \eqref{es:parW2W-3}, it follows from equality \eqref{exp:nab2u} that
\begin{align}\label{es:parWel-nab2u}
  \|\partial_W^{\ell+1} \nabla^2 u\|_{L^1_t(B^{\gamma-1}_{\infty,\infty})}
  & \leq \|\partial_W^{\ell+1} \nabla^2 \nabla^\perp \Lambda^{-2}\Gamma\|_{L^1_t(B^{\gamma-1}_{\infty,\infty})}
  + \|\partial_W^{\ell+1} \nabla^2 \nabla^\perp \partial_1 \Lambda^{-4}\theta \|_{L^1_t(B^{\gamma-1}_{\infty,\infty})} \nonumber \\
  & \leq \| \nabla \nabla^\perp \Lambda^{-2}(\nabla\Gamma)\|_{L^1_t(\bb^{\gamma-1,\ell+1}_{\infty,W})} + \|\nabla^2 \nabla^\perp \partial_1 \Lambda^{-4}\theta \|_{L^1_t(\bb^{\gamma-1,\ell+1}_{\infty, W})}.
\end{align}
By virtue of \eqref{eq:prod-es6}, we get
\begin{align*}
  \|\nabla^2 \nabla^\perp \partial_1 \Lambda^{-4} \theta \|_{L^1_t(\bb^{\gamma-1,\ell+1}_{\infty,W})}
  \lesssim \|\theta\|_{L^1_t(\bb^{\gamma-1,\ell+1}_{\infty,W})} + \int_0^t (\|W\|_{\bb^{1,\ell}_W} +1 )
  \big( \|\theta\|_{\bb^{\gamma-1,\ell}_{\infty,W}} + \|\theta\|_{L^2} \big) \dd \tau.
\end{align*}
Since $\partial_W^i \theta$ for every $i\in \{1,\cdots,\ell+1\}$ satisfies
$\partial_t (\partial_W^i \theta) + u\cdot\nabla (\partial_W^i \theta) = 0$, we use \eqref{eq:T-sm3} and Lemma \ref{lem:str-reg} to infer that for every $i\in \{1,\cdots,\ell+1\}$,
\begin{equation}\label{es:parW-the-2}
  \|\partial_W^i \theta\|_{L^\infty_t (B^{\gamma-1}_{\infty,\infty})} \leq C e^{C \|\nabla u\|_{L^1_t (L^\infty)}}
  \|\partial_{W_0}^i \theta_0\|_{ B^{\gamma-1}_{\infty,\infty}} \leq C e^{C(1+ t)^3},
\end{equation}
and
\begin{align}
  \|\theta\|_{L^\infty_t (\bb^{\gamma-1,\ell+1}_{\infty,W})}
  \leq C e^{C \|\nabla u\|_{L^1_t(L^\infty)}} \|\theta_0\|_{\bb^{\gamma-1,\ell+1}_{\infty,W}} \leq C e^{C(1+t)^3}. \nonumber
\end{align}
Collecting the above estimates leads to
\begin{align}\label{eq:parW-el-the2}
  \| \nabla^2\nabla^\perp \partial_1 \Lambda^{-4}\theta\|_{L^1_t(\bb^{\gamma-1,\ell+1}_{\infty,W})}
  \leq C e^{C(1+t)^3} + C e^{C(1+t)^3} \int_0^t \|W(\tau)\|_{\bb^{1,\ell}_W}  \dd \tau.
\end{align}

For the first term of the right-hand side of \eqref{es:parWel-nab2u},
we use \eqref{eq:prod-es6} to deduce that
\begin{align}\label{es:parW-el-nabGam}
  \|\nabla\nabla^\perp \Lambda^{-2} (\nabla \Gamma)\|_{L^1_t(\bb^{\gamma-1,\ell+1}_{\infty,W})}
  & \lesssim \|\nabla \Gamma\|_{L^1_t (\bb^{\gamma-1,\ell+1}_{\infty,W})} + \int_0^t (1+ \|W\|_{\bb^{1,\ell}_W})
  (\|\nabla \Gamma\|_{\bb^{\gamma-1,\ell}_{\infty,W}} + \|\Gamma\|_{L^p} )\dd \tau \nonumber \\
  & \lesssim \|\Gamma\|_{L^1_t (\bb^{\gamma,\ell+1}_{\infty,W})}
  + \int_0^t \|W(\tau)\|_{\bb^{1,\ell}_W} (\|\Gamma(\tau)\|_{\bb^{\gamma,\ell}_{\infty,W}}
  + 1) \dd \tau + 1,
\end{align}
where in the last line we have used the following estimate (in light of Lemma \ref{lem:Tw-es} and \eqref{u0-comm-2} below)
\begin{align}\label{es:nab-Gam}
  \|\nabla \Gamma\|_{\bb^{\gamma-1,\ell+1}_{\infty,W}} & = \|\nabla \Gamma\|_{\bb^{\gamma-1,\ell}_{\infty,W}}
  +  \|\partial_W^{\ell+1} \nabla \Gamma\|_{B^{\gamma-1}_{\infty,\infty}} \nonumber \\
  & \lesssim \|\Gamma\|_{\bb^{\gamma,\ell}_{\infty,W}} +  \|\partial_W^{\ell+1}\Gamma\|_{B^\gamma_{\infty,\infty}}
  + \|[\nabla,\partial_W^{\ell+1}]\Gamma\|_{B^{\gamma-1}_{\infty,\infty}} \nonumber \\
  & \lesssim \|\Gamma\|_{\bb^{\gamma,\ell+1}_{\infty,W}}
  + \sum_{i=0}^\ell \|\nabla W \cdot \nabla \partial_W^{\ell-i} \Gamma \|_{\bb^{\gamma-1,i}_{\infty,W}} \nonumber \\
  & \lesssim \|\Gamma\|_{\bb^{\gamma,\ell+1}_{\infty,W}} + \sum_{i=1}^\ell \|\nabla W\|_{\bb^{0,i}_W}
  \|\partial_W^{\ell-i} \Gamma\|_{\bb^{\gamma,i}_{\infty,W}}
  \lesssim \|\Gamma\|_{\bb^{\gamma,\ell+1}_{\infty,W}} +  \|W\|_{\bb^{1,\ell}_W} \|\Gamma\|_{\bb^{\gamma,\ell}_{\infty,W}}.
\end{align}
In the following we consider the smoothing estimate of $\partial_W^{\ell+1} \Gamma$. 
From \eqref{Gamm-eq} and $[\partial_W^{\ell+1},\partial_t  + u\cdot\nabla]=0$, we see that
\begin{equation}\label{par-Wel-Gam-eq}
\begin{split}
  \partial_t (\partial_W^{\ell+1}\Gamma) + u\cdot \nabla (\partial_W^{\ell+1} \Gamma) - \Delta (\partial_W^{\ell+1}\Gamma) =
  [\Delta, \partial_W^{\ell+1}] \Gamma +   \partial_W^{\ell+1} ([\RR_{-1}, u\cdot\nabla]\theta)
   =: F_{\ell+1},
\end{split}
\end{equation}
with
\begin{align}\label{Gam-comm-2}
  [\Delta,\partial_W^{\ell+1} ]\Gamma  & = [\Delta,\partial_W]\partial_W^\ell \Gamma  + \partial_W \big([\Delta, \partial_W^\ell] \Gamma\big)
  = \sum_{i=0}^\ell \partial_W^i\big( [\Delta,\partial_W]\partial_W^{\ell-i} \Gamma  \big) \nonumber \\
  & = \sum_{i=0}^\ell \partial_W^i \big(\Delta W \cdot \nabla \partial_W^{\ell-i} \Gamma  \big)
  +   \sum_{i=0}^\ell \partial_W^i \big(2 \nabla W : \nabla^2 \partial_W^{\ell-i} \Gamma  \big).
\end{align}

According to Lemma \ref{lem:TD-sm2}, we infer that for every $\gamma'\in (0, \min \{\gamma,1- \frac{2}{p}\})$,
\begin{equation}\label{par-Wel-u-es}
\begin{split}
  & \|\partial_W^{\ell+1} \Gamma \|_{L^1_t  (B^{\gamma'+1}_{\infty,1})} + \|\partial_W^{\ell+1} \Gamma(t) \|_{B^{\gamma'-1}_{\infty,1}} \\
  \leq & C (1+t)\bigg(\|\partial_{W_0}^{\ell+1} \Gamma _0\|_{ B^{\gamma'-1}_{\infty,1}}
  + \|F_{\ell+1}\|_{L^1_t  (B^{\gamma'-1}_{\infty,1})}+ \int_0^t \|\nabla u \|_{L^\infty}
  \|\partial_W^{\ell+1} \Gamma \|_{ B^{\gamma'-1}_{\infty,1}} \dd \tau\bigg) .
\end{split}
\end{equation}
By virtue of the relation $\Gamma_0 = \omega_0 -\RR_{-1}\theta_0$ and the following equality
\begin{equation}\label{u0-comm-2}
\begin{split}
  [\nabla, \partial_W^{\ell+1}]f = [\nabla, \partial_W] \partial_W^\ell f + \partial_W [\nabla,\partial_W]\partial_W^{\ell-1} f + \partial_W^2 [\nabla,\partial_W] \partial_W^{\ell-2} f
  = \sum_{i=0}^\ell \partial_W^i \big( \nabla W \cdot\nabla \partial_W^{\ell-i} f \big) ,
\end{split}
\end{equation}
and using Lemma \ref{lem:prod-es2}, we deduce that
\begin{align}\label{eq:parW-ell-Gam0}
  \|\partial_{W_0}^{\ell+1} \Gamma_0 \|_{B^{\gamma'-1}_{\infty,1}} & \leq \|\partial_{W_0}^{\ell+1}\nabla u_0\|_{B^{\gamma'-1}_{\infty,1}}
  + \|\partial_{W_0}^{\ell+1} \RR_{-1}\theta_0\|_{B^{\gamma'-1}_{\infty,1}} \nonumber \\
  & \lesssim \|\nabla \partial_{W_0}^{\ell+1} u_0\|_{B^{\gamma'-1}_{\infty,1}}
  + \sum_{i=0}^\ell \|\nabla W_0\cdot \nabla \partial_{W_0}^{\ell-i}u_0\|_{\bb^{\gamma'-1,i}_{W_0}}
  + \|W_0\cdot \nabla \RR_{-1}\theta_0\|_{\bb^{\gamma'-1,\ell}_{W_0}} \nonumber \\
  & \lesssim \|\partial_{W_0}^{\ell+1} u_0\|_{W^{1,p}}
  + \sum_{i=0}^\ell \|\nabla W_0\|_{\bb^{0,i}_{W_0}} \|\nabla \partial_{W_0}^{\ell-i}u_0\|_{\bb^{\gamma'-1,i}_{W_0}}
  + \|W_0\|_{\bb^{0,\ell}_{W_0}} \|\nabla \RR_{-1} \theta_0\|_{\bb^{\gamma'-1,\ell}_{W_0}} \nonumber \\
  & \lesssim \|\partial_{W_0}^{\ell+1} u_0\|_{W^{1,p}} + \|W_0\|_{\bb^{1,\ell}_{W_0}}
  \|u_0\|_{\bb^{\gamma',\ell}_{W_0}}  \nonumber \\
  & \quad + \|W_0\|_{\bb^{0,\ell}_{W_0}} \big( 1+ \| W_0\|_{\bb^{1,\ell-1}_{W_0}} \big)
  \big( \|\theta_0\|_{\bb^{\gamma'-1,\ell}_{W_0}} + \|\theta_0\|_{L^2} \big) \nonumber \\
  & \lesssim (1+ \|W_0\|_{\bb^{1,\ell}_{W_0}}) \Big(\sum_{i=0}^{\ell+1} \|\partial_{W_0}^i u_0\|_{W^{1,p}} +
  \sum_{i=0}^\ell \|\partial_{W_0}^i \theta_0\|_{C^{-1,\gamma}} + \|\theta_0\|_{L^2} \Big) \lesssim 1.
\end{align}
For the term $[\Delta,\partial_W^{\ell+1}]\Gamma$ given by \eqref{Gam-comm-2}, taking advantage of Lemmas \ref{lem:prod-es2}
and \ref{lem:Tw-es}, we find
\begin{align}\label{es:Lap-parWel-Gam}
  \|[\Delta,\partial_W^{\ell+1}]\Gamma \|_{B^{\gamma'-1}_{\infty,1}}
  & \lesssim \sum_{i=0}^\ell \|\Delta W \cdot \nabla \partial_W^{\ell-i} \Gamma \|_{\bb^{\gamma'-1,i}_W}
  + \sum_{i=0}^\ell \| \nabla W \cdot \nabla^2 \partial_W^{\ell-i} \Gamma \|_{\bb^{\gamma'-1,i}_W} \nonumber \\
  & \lesssim \sum_{i=0}^\ell \Big( \|\Delta W\|_{\bb^{\gamma'-1,i}_W} \|\nabla \partial_W^{\ell-i} \Gamma\|_{\bb^{0,i}_W}
  + \sum_{j=0}^i \|\nabla W\|_{\bb^{0,j}_W} \|\nabla^2 \partial_W^{\ell-i} \Gamma\|_{\bb^{\gamma'-1,i-j}_W}\Big) \nonumber \\
  & \lesssim \sum_{i=0}^\ell  \| W\|_{\bb^{\gamma'+1,i}_W} \|\Gamma\|_{\bb^{1,\ell}_W}
  + \sum_{i=0}^\ell \sum_{j=0}^i \| W\|_{\bb^{1,j}_W} \|\partial_W^{\ell-i}\Gamma \|_{\BB^{\gamma'+1,i-j }_W}  \nonumber \\
  & \leq C  \| W\|_{\bb^{\gamma'+1,\ell}_W} \|\Gamma\|_{\bb^{\gamma'+1,\ell}_W},
\end{align}
where in the last line we have used the following estimate
\begin{align}
  & \quad \sum_{i=0}^\ell \sum_{j=0}^i \| W\|_{\bb^{1,j}_W} \|\partial_W^{\ell-i}\Gamma \|_{\BB^{\gamma'+1,i-j}_W} \nonumber \\
  & \lesssim \sum_{i=0}^\ell \sum_{j=0}^i \|W\|_{\bb^{1,j}_W} \Big( \|\partial_W^{\ell-i}\Gamma\|_{\bb^{\gamma'+1,i-j}_W}
  + \|\partial_W^{\ell-i}\Gamma\|_{\bb^{1,i-j}_W} \|W\|_{\bb^{\gamma'+1,i-j}_W} \Big) \nonumber \\
  & \lesssim  \sum_{j=0}^\ell \|W\|_{\bb^{1,j}_W} \Big( \|\Gamma\|_{\bb^{\gamma'+1,\ell-j}_W}
  + \|\Gamma\|_{\bb^{1,\ell-j}_W} \|W\|_{\bb^{\gamma'+1,\ell-j}_W} \Big)
  \nonumber \\
  & \lesssim \|W\|_{\bb^{1,\ell}_W} \|\Gamma\|_{\bb^{\gamma'+1,0}_W} \|W\|_{\bb^{\gamma'+1,0}_W}
  + \|W\|_{\bb^{1,\ell-1}_W} \big(\|\Gamma\|_{\bb^{\gamma'+1,\ell}_W} +
  \|\Gamma\|_{\bb^{1,\ell}_W} \|W\|_{\bb^{\gamma'+1,\ell}_W}\big) \nonumber \\
  & \leq C \|W\|_{\bb^{\gamma'+1,\ell}_W} \|\Gamma\|_{\bb^{\gamma'+1,\ell}_W}   \nonumber ,
\end{align}
and in the above $C>0$ depends on $H_\ell(T)$ that is the upper bound of $\|W\|_{\bb^{\gamma'+1,\ell-1}_W}$.
For the second term in $F_{\ell+1}$, by using the formula
$\nabla ([\RR_{-1},u\cdot\nabla]\theta) = [\nabla \RR_{-1}, u\cdot\nabla] \theta - (\nabla u) \cdot \nabla \RR_{-1}\theta$,
we see
\begin{align}
  & \quad \|\partial_W^{\ell+1}([\RR_{-1},u\cdot\nabla]\theta)\|_{L^1_t (B^{\gamma'-1}_{\infty,1})}
  \lesssim \|W\|_{L^\infty_t(L^\infty)} \|\nabla \partial_W^\ell ([\RR_{-1}, u\cdot\nabla]\theta)\|_{L^1_t(B^{\gamma'-1}_{\infty,1})} \nonumber \\
  & \lesssim \|\partial_W^\ell ([\nabla \RR_{-1},u\cdot\nabla]\theta)\|_{L^1_t(B^{\gamma'-1}_{\infty,1})}
  + \|\partial_W^\ell (\nabla u\cdot \nabla \RR_{-1}\theta)\|_{L^1_t (B^{\gamma'-1}_{\infty,1})} \nonumber \\
  & \quad + \|[\nabla, \partial_W^\ell] ([\RR_{-1}, u\cdot\nabla]\theta)\|_{L^1_t(B^{\gamma'-1}_{\infty,1})}\nonumber \\
  & = : N_1 + N_2 + N_3.
\end{align}
For $N_1$ and $N_2$, according to Lemma \ref{lem:prod-es2} and \eqref{es:parW-the-2}, it follows that
\begin{align}\label{es:N1}
  N_1 \leq \|[\nabla \RR_{-1},u\cdot\nabla]\theta\|_{L^1_t(\bb^{\gamma'-1,\ell}_W)}
  & \lesssim \big(\|\nabla u\|_{L^1_t (\bb^{0,\ell}_W)} + \|u\|_{L^1_t (L^\infty)} \big) \|\theta\|_{L^\infty_t (\bb^{\gamma'-1,\ell}_W)} \nonumber \\
  & \lesssim \big( \|\nabla u\|_{L^1_t (\bb^{0,\ell}_W)} + t \big) \Big(\sum_{\lambda=0}^\ell \|\partial_W^\lambda \theta\|_{L^\infty_t (C^{-1,\gamma})} \Big)
  \leq C,
\end{align}
and
\begin{align}\label{es:N2}
  & \quad N_2 \leq \|\nabla u\cdot \nabla \RR_{-1} \theta\|_{L^1_t (\bb^{\gamma'-1,\ell}_W)}
  \lesssim \|\nabla u\|_{L^1_t (\bb^{0,\ell}_W)} \|\nabla \RR_{-1}\theta\|_{L^\infty_t(\bb^{\gamma'-1,\ell}_W)} \nonumber \\
  & \lesssim \|\nabla u\|_{L^1_t (\bb^{0,\ell}_W)} \Big(\|\theta\|_{L^\infty_t(\bb^{\gamma'-1,\ell}_W)}
  + \big(1 + \|W\|_{L^\infty_t(\bb^{1,\ell-1}_W)} \big)
  \big(\|\theta\|_{L^\infty_t(\bb^{\gamma'-1,\ell-1}_W)} + \|\theta\|_{L^\infty_t(L^2)}\big) \Big)
  \leq C ,
\end{align}
where $C>0$ depends on $ H_\ell(T)$ with $H_\ell(T) \gtrsim e^{C_0(1+t)^3}$.
For $N_3$, we use formula \eqref{u0-comm-2}, Lemma \ref{lem:prod-es2} and estimates \eqref{es:N1}-\eqref{es:N2} to obtain that
\begin{align}
  N_3 & \leq \sum_{i=0}^{\ell-1} \|\partial_W^i (\nabla W \cdot \nabla \partial_W^{\ell-1-i} [\RR_{-1},u\cdot\nabla]\theta)\|_{L^1_t(B^{\gamma'-1}_{\infty,1})} \nonumber \\
  & \lesssim \sum_{i=0}^{\ell-1} \|\nabla W \cdot \nabla \partial_W^{\ell-1-i}[\RR_{-1},u\cdot\nabla]\theta\|_{L^1_t(\bb^{\gamma'-1,i}_W)} \nonumber \\
  & \lesssim  \sum_{i=0}^{\ell-1} \|\nabla W\|_{L^\infty_t (\bb^{0,i}_W)} \|\nabla \partial_W^{\ell-1-i}[\RR_{-1},u\cdot\nabla]\theta\|_{L^1_t(\bb^{\gamma'-1,i}_W)} \nonumber \\
  & \lesssim \|W\|_{L^\infty_t(\bb^{1,\ell-1}_W)} \sum_{i=0}^{\ell-1} \|\nabla \partial_W^{\ell-1-i}[\RR_{-1},u\cdot\nabla]\theta\|_{L^1_t(\bb^{\gamma'-1,i}_W)}
  \nonumber \\
  & \lesssim  \|\nabla ([\RR_{-1},u\cdot\nabla]\theta)\|_{L^1_t(\bb^{\gamma'-1,\ell-1}_W)} + \sum_{i=0}^{\ell-2}  \|[\nabla, \partial_W^{\ell-1-i}]([\RR_{-1},u\cdot\nabla]\theta)\|_{L^1_t(\bb^{\gamma'-1,i}_W)}  \nonumber \\
  & \leq  C + C \sum_{i=0}^{\ell-2} \sum_{j=0}^{\ell-2-i} \|\partial_W^j (\nabla W \cdot \nabla \partial_W^{\ell-2-i-j} ([\RR_{-1},u\cdot\nabla]\theta) )\|_{L^1_t(\bb^{\gamma'-1,i}_W)}  \nonumber \\
  & \leq  C + C \sum_{0\leq i+j\leq \ell-2} \|\nabla W \cdot \nabla \partial_W^{\ell-2-i-j} ([\RR_{-1},u\cdot\nabla] \theta)\|_{L^1_t(\bb^{\gamma'-1,i+j}_W)}  \nonumber \\
  & \leq  C + C \sum_{0\leq i+j \leq \ell-2} \|\nabla \partial_W^{\ell-2-i-j} ([\RR_{-1}, u\cdot\nabla]\theta)\|_{L^1_t(\bb^{\gamma'-1,i+j}_W)} , \nonumber
\end{align}
where $C>0$ depends on $ H_\ell(t)$ with $H_\ell(T) \geq H_{\ell-1}(T) \gtrsim e^{C_0(1+t)^3}$.
Iteratively repeating the above process yields that
\begin{align}\label{es:N3}
  N_3 \lesssim C + C \|\nabla ([\RR_{-1},u\cdot\nabla]\theta)\|_{L^1_t(B^{\gamma'-1}_{\infty,1})} \leq C.
\end{align}
Thus it follows from the above estimates \eqref{es:Lap-parWel-Gam} and \eqref{es:N1}-\eqref{es:N3} that
\begin{equation}\label{eq:Fel-es2}
\begin{split}
  \|F_{\ell+1}\|_{L^1_t (B^{\gamma'-1}_{\infty,1})}
  \leq C  \int_0^t \|W(\tau)\|_{\bb^{\gamma'+1,\ell}_W}
  \|\Gamma(\tau) \|_{\bb^{\gamma'+1,\ell}_W}  \dd \tau  + C .
\end{split}
\end{equation}
In combination with \eqref{par-Wel-u-es} and \eqref{eq:parW-ell-Gam0} we obtain
\begin{align}\label{eq:Gam-bb-es}
  & \quad \|\Gamma(t)\|_{\bb^{\gamma'-1,\ell+1}_W} + \|\Gamma\|_{L^1_t (\bb^{\gamma'+1,\ell+1}_W)} \nonumber \\
  & = \|\partial_W^{\ell+1} \Gamma(t)\|_{B^{\gamma'-1}_{\infty,1}}
  + \|\partial_W^{\ell+1} \Gamma\|_{L^1_t (B^{\gamma'+1}_{\infty,1})}
  + \|\Gamma(t)\|_{\bb^{\gamma'-1,\ell}_W} + \|\Gamma\|_{L^1_t(\bb^{\gamma'+1,\ell}_W)} \nonumber \\
  & \leq C \int_0^t \big(\|\Gamma\|_{\bb^{\gamma'+1,\ell}_W} + \|\nabla u\|_{L^\infty}\big)
  \big(\|W\|_{\bb^{\gamma'+1,\ell}_W} + \|\Gamma\|_{\bb^{\gamma'-1,\ell+1}_W}\big) \dd \tau + C .
\end{align}

Hence inserting \eqref{eq:Gam-bb-es} into \eqref{es:parW-el-nabGam}, and gathering \eqref{es:parW2W-3}-\eqref{es:parWel-nab2u}
and \eqref{eq:parW-el-the2}-\eqref{es:parW-el-nabGam} lead to that for every $0<\gamma'<\min\{\gamma, 1-\frac{2}{p}\}$,
\begin{align}\label{eq:bb-el-key}
  & \quad \|\partial_W^\ell \nabla^2 W(t)\|_{B^{\gamma-1}_{\infty,\infty}} + \|\partial_W^{\ell+1} \nabla^2 u\|_{L^1_t(B^{\gamma-1}_{\infty,\infty})} \nonumber \\
  & \leq C \int_0^t \big(\|\Gamma\|_{\bb^{\gamma'+1,\ell}_W} + \|\nabla u\|_{\bb^{\gamma,\ell}_{\infty,W}} + 1 \big)
   \big(\|W\|_{\bb^{\gamma+1,\ell}_{\infty,W}} + \|\Gamma\|_{\bb^{\gamma'-1,\ell+1}_W}\big) \dd \tau + C ,
\end{align}
where $C>0$ depends on $H_\ell(T)$.

By using \eqref{Gam-comm-2}, \eqref{u0-comm-2} and Lemmas \ref{lem:prod-es2}, \ref{lem:Tw-es}, we also infer that
\begin{align}
  \|[\nabla^2,\partial_W^\ell]W(t)\|_{B^{\gamma-1}_{\infty,\infty}}
  & \leq  \sum_{i=0}^{\ell-1} \| \partial_W^i \big(\nabla^2 W \cdot \nabla \partial_W^{\ell-1-i} W  \big)\|_{B^{\gamma-1}_{\infty,\infty}}
  +  2 \sum_{i=0}^{\ell-1} \| \partial_W^i \big( \nabla W \cdot \nabla^2 \partial_W^{\ell-1-i}W \big)\|_{B^{\gamma-1}_{\infty,\infty}} \nonumber \\
  & \lesssim \sum_{i=0}^{\ell-1} \|\nabla^2 W \cdot\nabla \partial_W^{\ell-1-i} W\|_{\bb^{\gamma-1,i}_{\infty,W}}
  + \sum_{i=0}^{\ell-1} \|\nabla W\cdot\nabla^2 \partial_W^{\ell-1-i}W\|_{\bb^{\gamma-1,i}_{\infty,W}} \nonumber \\
  & \lesssim \sum_{i=0}^{\ell-1} \Big( \|\nabla^2 W\|_{\bb^{\gamma-1,i}_{\infty,W}}
  \|\nabla \partial_W^{\ell-1-i}W\|_{\bb^{0,i}_W}
  +  \|\nabla W\|_{\bb^{0,i}_W} \|\nabla^2 \partial_W^{\ell-1-i}W\|_{\bb^{\gamma-1,i}_{\infty,W}} \Big) \nonumber \\
  & \lesssim \|W\|_{\bb^{\gamma+1,\ell-1}_{\infty,W}} \|W\|_{\bb^{1,\ell-1}_W} \leq C , \nonumber
\end{align}
and
\begin{align}
  \|[\nabla,\partial_W^{\ell+1}] \nabla u\|_{L^1_t (B^{\gamma-1}_{\infty,\infty})} & \leq \sum_{i=0}^\ell \|\partial_W^i \big( \nabla W \cdot\nabla \partial_W^{\ell-i} \nabla u\big) \|_{L^1_t (B^{\gamma-1}_{\infty,\infty})} \nonumber \\
  & \leq \sum_{i=0}^\ell \|\nabla W \cdot \nabla \partial_W^{\ell-i} \nabla u\|_{L^1_t(\bb^{\gamma-1,i}_{\infty,W})} \nonumber \\
  & \lesssim \sum_{i=0}^\ell \int_0^t \|\nabla W\|_{\bb^{0,i}_W} \|\nabla \partial_W^{\ell-i} \nabla u\|_{\bb^{\gamma-1,i}_{\infty,W}} \dd \tau
  \lesssim \int_0^t \|W(\tau)\|_{\bb^{1,\ell}_W} \|\nabla u(\tau)\|_{\bb^{\gamma,\ell}_{\infty,W}} \dd \tau. \nonumber
\end{align}
Consequently, in view of the following estimates
\begin{align}\label{es:W-bb-el}
  \|W\|_{\bb^{\gamma+1,\ell}_{\infty,W}} & \leq \|\partial_W^\ell W\|_{B^{\gamma+1}_{\infty,\infty}} + \|W\|_{\bb^{\gamma+1,\ell-1}_{\infty,W}} \nonumber \\
  & \lesssim \|\nabla^2 \partial_W^\ell W\|_{B^{\gamma-1}_{\infty,\infty}} + \|\Delta_{-1} \partial_W^\ell W\|_{L^\infty} + C\nonumber \\
  & \lesssim \|\partial_W^\ell \nabla^2 W\|_{B^{\gamma-1}_{\infty,\infty}} + \|[\nabla^2, \partial_W^\ell] W\|_{B^{\gamma-1}_{\infty,\infty}} + \|W\|_{L^\infty} \|\partial_W^{\ell-1}W\|_{L^\infty}
  + C \nonumber \\
  & \lesssim \|\partial_W^\ell \nabla^2 W\|_{B^{\gamma-1}_{\infty,\infty}} + C,
\end{align}
and
\begin{align}\label{es:nab-u-bbel}
  \|\nabla u\|_{L^1_t(\bb^{\gamma,\ell+1}_{\infty,W})}
  & \leq \|\partial_W^{\ell+1} \nabla u\|_{L^1_t (B^{\gamma}_{\infty,\infty})}
  + \|\nabla u\|_{L^1_t(\bb^{\gamma,\ell}_{\infty,W})} \nonumber \\
  & \lesssim \|\nabla\partial_W^{\ell+1} \nabla u\|_{L^1_t(B^{\gamma-1}_{\infty,\infty})}
  +  \|\Delta_{-1}\partial_W^{\ell+1} \nabla u\|_{L^1_t(L^\infty)} + C \nonumber \\
  & \lesssim \|\partial_W^{\ell+1}\nabla^2 u\|_{L^1_t(B^{\gamma-1}_{\infty,\infty})}
  + \|[\nabla,\partial_W^{\ell+1}]\nabla u\|_{L^1_t(B^{\gamma-1}_{\infty,\infty})}
  + \|W\|_{L^\infty_t(L^\infty)} \|\partial_W^\ell \nabla u\|_{L^1_t(L^\infty)}
  + C \nonumber \\
  & \lesssim \|\partial_W^{\ell+1}\nabla^2 u\|_{L^1_t(B^{\gamma-1}_{\infty,\infty})}
  + \int_0^t \|W(\tau)\|_{\bb^{1,\ell}_W} \|\nabla u(\tau)\|_{\bb^{\gamma,\ell}_{\infty,W}} \dd \tau + C,
\end{align}
we collect estimates \eqref{eq:Gam-bb-es}-\eqref{es:nab-u-bbel} to yield that for every $0<\gamma'<\min\{\gamma, 1-\frac{2}{p}\}$,
\begin{align}\label{eq:final}
  & \quad \|W(t)\|_{\bb^{\gamma+1,\ell}_{\infty,W}} + \|\nabla u\|_{L^1_t(\bb^{\gamma,\ell+1}_{\infty,W})}
  + \|\Gamma(t)\|_{\bb^{\gamma'-1,\ell+1}_W} + \|\Gamma\|_{L^1_t(\bb^{\gamma'+1,\ell+1}_W)} \nonumber \\
  & \leq C \int_0^t \big(\|\Gamma(\tau)\|_{\bb^{\gamma'+1,\ell}_W} + \|\nabla u(\tau)\|_{\bb^{\gamma,\ell}_{\infty,W}} + 1\big)
  \big(\|W(\tau)\|_{\bb^{\gamma+1,\ell}_{\infty,W}}
  + \|\Gamma(\tau)\|_{\bb^{\gamma'-1,\ell+1}_W}\big) \dd \tau
  + C ,
\end{align}
where $C>0$ depends on $H_\ell(T)$.
Gronwall's inequality and assumption \eqref{assum-el} guarantee that
\begin{align}\label{key-es-3}
  & \quad \|W\|_{L^\infty_T (\bb^{\gamma+1,\ell}_{\infty,W})} + \|\nabla u\|_{L^\infty_T (\bb^{\gamma,\ell+1}_{\infty,W})}
  + \|\Gamma\|_{L^\infty_T (\bb^{\gamma'-1,\ell+1}_W)} + \|\Gamma\|_{L^1_T(\bb^{\gamma'+1,\ell+1}_W)} \nonumber \\
  & \leq \, C \exp \Big\{C  \|\nabla u \|_{L^1_T (\bb^{\gamma,\ell}_{\infty,W})} + C  \|\Gamma\|_{L^1_T (\bb^{\gamma'+1,\ell}_W)} + C T\Big\} \lesssim_{H_\ell(T)} 1,
\end{align}
which corresponds to \eqref{eq:target-ell}, as desired.
Hence, the target estimate \eqref{eq:Targ-k} is fulfilled and the proof is completed.

\section{Striated estimates: proof of Lemmas \ref{lem:prod-es2} and \ref{lem:prod-es}}\label{sec:prod-es}

\subsection{Proof of Lemma \ref{lem:prod-es2}}\label{subsec:str-es1}

Similar to \cite[Section 2.2]{Chem88} or \cite[Lemma 7.1]{LZ19}, denote by the operator $R_q$ as
\begin{align}\label{def:Rq}
  R_q(\alpha_1,\cdots,\alpha_m) := \int_{[0,1]^m} \int_{\R^d} \prod_{i=1}^m \alpha_i(x +f_i (\tau) 2^{-q}y) h(\tau,y) \dd y \dd \tau,
\end{align}
where $q\in\N$, $h\in C([0,1]^m; \mathcal{S}(\R^d))$, $f_i\in L^\infty([0,1]^m)$, $f_i(\tau)\neq 0$ for every $\tau\in (0,1)^m$.
Note that when $f_i \equiv 0$ and $\int_{\R^d} h(\tau,y) \dd y =1$, one has $R_q(\alpha_1,\cdots,\alpha_m) = \prod\limits_{i=1}^m \alpha_i(x)$.
We first have the following crucial result, and its detailed proof is postponed in the subsection \ref{subsec:lem5.1}.
\begin{lemma}\label{lem:Rq}
  Let $k\in \Z^+$, $\sigma\in (0,1)$, $N\in \mathbb{Z}^+$ and $\mathcal{W}=\{W_i\}_{1\leq i \leq N}$ be a set of regular divergence-free vector fields of $\R^d$ satisfying that
\begin{equation}\label{norm:W2}
\begin{split}
  \|\mathcal{W}\|_{\BB^{1+\sigma,k-1}_{\infty,\mathcal{W}}} := \sum_{\lambda=0}^{k-1}
  \|(T_{\mathcal{W}\cdot\nabla})^\lambda \mathcal{W}\|_{B^{1+\sigma}_{\infty,\infty}}
  = \sum_{\lambda=0}^{k-1} \sum_{\lambda_1 + \cdots \lambda_N =\lambda} \|(T_{W_1\cdot\nabla})^{\lambda_1}\cdots (T_{W_N\cdot \nabla})^{\lambda_N} \mathcal{W}\|_{B^{1+\sigma}_{\infty,\infty}} < \infty.
\end{split}
\end{equation}
Let $\alpha_i$ ($i=1,\cdots,m$) be such that $\mathrm{supp}\,\widehat{\alpha_i} \subset B(0,C_i 2^q)$,
and let $\psi$ be a smooth function with compact support in a ball.
Then we have that for every $s \in \R$, $(p,r)\in [1,\infty]^2$ and $\ell \in \{0,1,\cdots,k\}$,
\begin{align}\label{eq:Tw-Rq}
  \big\|(T_{\mathcal{W}\cdot\nabla})^\ell R_q(\alpha_1,\cdots,\alpha_m)\big\|_{L^p}
  \lesssim \min_{1\leq i\leq m} \Big(\sum_{|\mu|\leq \ell} \|(T_{\mathcal{W}\cdot\nabla})^{\mu_i} \alpha_i\|_{L^p} \prod_{1\leq j\leq m,j\neq i} \|(T_{\mathcal{W}\cdot\nabla})^{\mu_j} \alpha_j\|_{L^\infty}\Big),
\end{align}
with $\mu =(\mu_1,\cdots,\mu_m)$ and $|\mu|= \mu_1 + \cdots + \mu_m$,
and 
\begin{equation}\label{eq:Tw-phi-es1}
  \|(T_{\mathcal{W} \cdot\nabla})^\ell \psi(2^{-q}D) \phi\|_{L^p} \lesssim \sum_{\lambda=0}^\ell
  \|(T_{\mathcal{W} \cdot\nabla})^\lambda \phi\|_{L^p},
\end{equation}
and 
\begin{align}\label{eq:Tw-es-key}
  \big\|\big\{ 2^{q s}\|(T_{\mathcal{W}\cdot\nabla})^\ell \Delta_q \phi\|_{L^p}\big\}_{q\geq -1} \big\|_{\ell^r} +
  \big\|\big\{ 2^{q(s-1)}\|(T_{\mathcal{W}\cdot\nabla})^\ell \nabla \Delta_q \phi\|_{L^p}\big\}_{q\geq -1} \big\|_{\ell^r}
  \lesssim 
  \|\phi\|_{\BB^{s,\ell}_{p,r,\mathcal{W}}} ,
\end{align}
and
\begin{equation}\label{eq:BB-nab-es}
  \|\nabla \phi\|_{\BB^{s,\ell}_{p,r,\mathcal{W}}} \lesssim \|\phi\|_{\BB^{s+1,\ell}_{p,r,\mathcal{W}}}.
\end{equation}
In the above all the hidden constants depend on $\|\mathcal{W}\|_{\BB^{1+\sigma,k-1}_{\infty,\mathcal{W}}}$.
\end{lemma}

Based on Lemma \ref{lem:Rq}, we get the following useful striated estimates, whose proof is placed in Subsection \ref{subsec:lem5.2}.
\begin{lemma}\label{lem:Tw-es}
  Let $\mathcal{W}=\{W_i\}_{1\leq i\leq N}$ ($N\in \mathbb{Z}^+$)
be a set of regular divergence-free vector fields of $\R^d$ satisfying \eqref{norm:W2} with $k\in \Z^+$, $\sigma\in (0,1)$.
Let $m(D)$ be a zero-order pseudo-differential operator with $m(\xi)\in C^\infty(\R^d\setminus\{0\})$.
Then there exist positive constants depending on $\|\mathcal{W}\|_{\BB^{1+\sigma,k-1}_{\infty,\mathcal{W}}}$
such that the following statements hold true for every $\ell\in \{0,1,\cdots, k\}$.
\begin{enumerate}[(1)]
\item
We have that
for every $q\geq -1$,
\begin{align}\label{eq:Tw-es-key2}
  \|\Delta_q (T_{\mathcal{W}\cdot\nabla})^\ell \nabla m(D) \phi\|_{L^p} \lesssim\sum_{\lambda=0}^\ell 2^q \|(T_{\mathcal{W}\cdot\nabla})^\lambda \phi\|_{L^p},
\end{align}
and for every $q\in \N$,
\begin{align}\label{eq:Tw-es-key3}
  2^q \| (T_{\mathcal{W}\cdot\nabla})^\ell \Delta_q \phi \|_{L^p}
  \lesssim \sum_{q_1\in \N,|q_1-q|\leq N_\ell} \sum_{\lambda=0}^\ell \|(T_{\mathcal{W}\cdot\nabla})^\lambda \Delta_{q_1} \nabla \phi\|_{L^p},
\end{align}
with $N_\ell\in \N$ depending only on $\ell$.
\item 
We have that for every $s< 0$,
\begin{equation}\label{es:Tw-parap-1}
  \|T_v w\|_{\BB^{s,\ell}_{p,r,\mathcal{W}}} 
  \lesssim   \min\bigg\{\sum_{\mu=0}^\ell \|v\|_{\BB^{0,\mu}_\mathcal{W}} \| w\|_{\BB^{s,\ell-\mu}_{p,r,\mathcal{W}}} ,
  \sum_{\mu=0}^\ell \|v\|_{\BB^{s,\mu}_{p,r,\mathcal{W}}} \| w\|_{\BB^{0,\ell-\mu}_\mathcal{W}}\bigg\},
\end{equation}
and
\begin{equation}\label{es:Tw-parap-1b}
  \|T_v w\|_{\BB^{0,\ell}_{p,r,\mathcal{W}}} \lesssim \sum_{\mu=0}^\ell \|v\|_{\BB^{0,\mu}_\mathcal{W}} \|w\|_{\BB^{0,\ell-\mu}_{p,r,\mathcal{W}}},
\end{equation}
and for every $s < 1$,
\begin{equation}\label{es:Tw-parap-2}
  \|T_{\nabla w} v\|_{\BB^{s,\ell}_{p,r,\mathcal{W}}} 
  \lesssim  \sum_{\mu=0}^\ell  \|w\|_{\BB^{s,\mu}_{\infty,\mathcal{W}}} \| v\|_{\BB^{1,\ell-\mu}_{p,r,\mathcal{W}}},
\end{equation}
and for every $s\in \R$,
\begin{equation}\label{es:Tw-parap-2c}
  \|T_{\nabla w}v\|_{\BB^{s,\ell}_{p,r,\mathcal{W}}} \lesssim \sum_{\mu=0}^\ell \|\nabla w\|_{\BB^{0,\mu}_\mathcal{W}} \|v\|_{\BB^{s,\ell-\mu}_{p,r,\mathcal{W}}}.
\end{equation}
\item Assume that $v$ is a divergence-free vector field of $\R^d$, then we have that for every $s>-1$,
\begin{equation}\label{es:Tw-rem-1}
\begin{split}
  \|R(v\cdot,\nabla w)\|_{\BB^{s,\ell}_{p,r,\mathcal{W}}}
  \lesssim
  \min\bigg\{\sum_{\mu=0}^\ell \|v\|_{\BB^{0,\mu}_\mathcal{W}} \|\nabla w\|_{\BB^{s,\ell-\mu}_{p,r,\mathcal{W}}} ,
  \sum_{\mu=0}^\ell \|v\|_{\BB^{s,\mu}_{p,r,\mathcal{W}}} \|\nabla w\|_{\BB^{0,\ell-\mu}_\mathcal{W}},
  \sum_{\mu=0}^\ell \|v\|_{\BB^{1,\mu}_\mathcal{W}} \| w\|_{\BB^{s,\ell-\mu}_{p,r,\mathcal{W}}}
  \bigg\}.
\end{split}
\end{equation}
\item We have that for every $s\in (-1,1)$,
\begin{equation}\label{normBB-equiv}
  \|\phi\|_{\bb^{s,\ell}_{p,r,\mathcal{W}}} \lesssim \|\phi\|_{\BB^{s,\ell}_{p,r,\mathcal{W}}} \lesssim \|\phi\|_{\bb^{s,\ell}_{p,r,\mathcal{W}}},
\end{equation}
and
\begin{equation}\label{normBB-equivb}
  \|\phi\|_{\bb^{1,\ell}_\mathcal{W}} \lesssim \|\phi\|_{\BB^{1,\ell}_\mathcal{W}} \lesssim \|\phi\|_{\bb^{1,\ell}_\mathcal{W}},
\end{equation}
and for every $s>-1$,
\begin{align}\label{eq:W-BBbb-equ}
  \|\mathcal{W}\|_{\BB^{s,\ell}_{p,r,\mathcal{W}}} \lesssim \|\mathcal{W}\|_{\bb^{s,\ell}_{p,r,\mathcal{W}}},
\end{align}
and for every $s\geq 1$,
\begin{align}\label{eq:BBbb-equiv1}
  \|\phi\|_{\BB^{s,\ell}_{p,r,\mathcal{W}}} \lesssim \|\phi\|_{\bb^{s,\ell}_{p,r,\mathcal{W}}} + \|\phi\|_{\bb^{1,\ell}_\mathcal{W}} \|\mathcal{W}\|_{\bb^{s,\ell}_{p,r,\mathcal{W}}}.
\end{align}
\end{enumerate}
\end{lemma}

Now we turn to the proof of Lemmas \ref{lem:prod-es2}.
\begin{proof}[Proof of Lemma \ref{lem:prod-es2}]
  (1) By using Bony's decomposition and \eqref{normBB-equiv}, \eqref{es:Tw-parap-1}, \eqref{es:Tw-rem-1}, we get
\begin{align*}
  \|u\cdot\nabla \phi\|_{\bb^{-\epsilon,k}_{p,r,\mathcal{W}}} & \lesssim \|u\cdot \nabla \phi\|_{\BB^{-\epsilon,k}_{p,r,\mathcal{W}}} \\
  & \lesssim \|T_{u\cdot}\nabla \phi\|_{\BB^{-\epsilon,k}_{p,r,\mathcal{W}}}
  + \|T_{\nabla\phi}\cdot u\|_{\BB^{-\epsilon,k}_{p,r,\mathcal{W}}}
  + \|R(u\cdot,\nabla \phi)\|_{\BB^{-\epsilon,k}_{p,r,\mathcal{W}}} \\
  & \lesssim  \min\Big\{\sum_{\mu=0}^k \|u\|_{\BB^{0,\mu}_\mathcal{W}}
  \|\nabla \phi\|_{\BB^{-\epsilon,k-\mu}_{p,r,\mathcal{W}}},
  \sum_{\mu=0}^k \|u\|_{\BB^{-\epsilon,\mu}_{p,r,\mathcal{W}}}
  \|\nabla \phi\|_{\BB^{0,k-\mu}_\mathcal{W}} \Big\} \\
  &  \lesssim  \min\Big\{\sum_{\mu=0}^k \|u\|_{\bb^{0,\mu}_\mathcal{W}}
  \|\nabla \phi\|_{\bb^{-\epsilon,k-\mu}_{p,r,\mathcal{W}}},
  \sum_{\mu=0}^k \|u\|_{\bb^{-\epsilon,\mu}_{p,r,\mathcal{W}}} \|\nabla \phi\|_{\bb^{0,k-\mu}_\mathcal{W}} \Big\}.
\end{align*}

(2) We are devoted to proving \eqref{eq:prod-es6} by induction on the index $k$.
For $k=0$, \eqref{eq:prod-es6} is explicitly estimated by \eqref{eq:str-es1}
where $C>0$ is a universal constant (the norm $\|\mathcal{W}\|_{\bb^{1+\sigma,k-1}_{\infty,\mathcal{W}}}$ plays no role).

Assume that \eqref{eq:prod-es6} holds for $\ell\in \{0,1,\cdots,k-1\}$ (when $k=0$, for $\ell=0$) with $\ell$ in place of the $k$-index, we intend to show that it also holds for $\ell+1$.
By using \eqref{eq:Bes-prop3} and \eqref{normBB-equiv}, we see that
\begin{align}\label{es:mD-phi1}
  \|m(D)\phi\|_{\bb^{-\epsilon,\ell+2}_{p,r,\mathcal{W}}} &
  = \|\partial_\mathcal{W} (m(D)\phi)\|_{\bb^{-\epsilon,\ell+1}_{p,r,\mathcal{W}}}
  + \|m(D)\phi\|_{B^{-\epsilon}_{p,r}} \nonumber \\
  & \lesssim \|\mathcal{W}\cdot \nabla(m(D)\phi)\|_{\BB^{-\epsilon,\ell+1}_{p,r,\mathcal{W}}} + \|\phi\|_{B^{-\epsilon}_{p,r}} + \|\Delta_{-1}m(D)\phi\|_{L^p}.
\end{align}
Noting that there exists a bump function $\widetilde{\psi} \in C^\infty_c(\R^d)$ supported on an annulus so that
$S_{j-1}\mathcal{W}\cdot \Delta_j f = \widetilde{\psi}(2^{-j}D) (S_{j-1}\mathcal{W}\cdot \Delta_j f )$ for every $j\in \N$,
we have
\begin{align}\label{exp:Tw-mD0}
  (T_{\mathcal{W}\cdot\nabla})m(D)f & = \sum_{j\in \N} S_{j-1} \mathcal{W} \cdot \nabla \Delta_j m(D) f \nonumber \\
  & = - \sum_{j\in\N} [m(D)\widetilde{\psi}(2^{-j}D), S_{j-1}\mathcal{W}\cdot]\nabla \Delta_j f + m(D) (T_{\mathcal{W}\cdot\nabla}) f,
\end{align}
then Bony's decomposition yields that
\begin{align}\label{decom:Xi}
  & \quad \mathcal{W}\cdot\nabla (m(D)\phi)
  = T_{\nabla m(D)\phi}\cdot \mathcal{W} + R(\mathcal{W}\cdot,\nabla m(D)\phi)
  + (T_{\mathcal{W}\cdot\nabla})(m(D)\phi) \nonumber \\
  & =  T_{\nabla m(D)\phi}\cdot \mathcal{W} +  R(\mathcal{W}\cdot,\nabla m(D)\phi) -  \sum_{j\in\N} [m(D)\widetilde{\psi}(2^{-j}D), S_{j-1}\mathcal{W}\cdot]\nabla \Delta_j  \phi  + m(D)(T_{\mathcal{W}\cdot\nabla}) \phi \nonumber \\
  & = : \Xi_1 + \Xi_2 + \Xi_3 + \Xi_4 .
\end{align}
In view of \eqref{es:Tw-parap-2}, \eqref{es:Tw-rem-1} and the induction assumption, we find
\begin{align}\label{es:X12}
  \|\Xi_1\|_{\BB^{-\epsilon,\ell+1}_{p,r,\mathcal{W}}} + \|\Xi_2\|_{\BB^{-\epsilon,\ell+1}_{p,r,\mathcal{W}}}
  & \lesssim \|m(D)\phi\|_{\BB^{-\epsilon,\ell+1}_{p,r,\mathcal{W}}} \|\mathcal{W}\|_{\BB^{1,\ell+1}_\mathcal{W}} \nonumber \\
  & \lesssim \|m(D)\phi\|_{\bb^{-\epsilon,\ell+1}_{p,r,\mathcal{W}}} \|\mathcal{W}\|_{\bb^{1,\ell+1}_\mathcal{W}} \nonumber \\
  & \lesssim \Big( \|\phi\|_{\bb^{-\epsilon,\ell+1}_{p,r,\mathcal{W}}} +
  \big(1+ \| \mathcal{W}\|_{\bb^{1,\ell}_\mathcal{W}}\big) \big(\|\phi\|_{\bb^{-\epsilon,\ell}_{p,r,\mathcal{W}}} + \|\Delta_{-1}m(D)\phi\|_{L^p}\big) \Big)
  \|\mathcal{W}\|_{\bb^{1,\ell+1}_\mathcal{W}} \nonumber \\
  & \lesssim \big(\|\phi\|_{\bb^{-\epsilon,\ell+1}_{p,r,\mathcal{W}}} + \|\Delta_{-1}m(D)\phi\|_{L^p} \big) \|\mathcal{W}\|_{\bb^{1,\ell+1}_\mathcal{W}},
\end{align}
where in the last line we have used the fact that
$\|\mathcal{W}\|_{\bb^{1,\ell}_\mathcal{W}} \leq \|\mathcal{W}\|_{\bb^{1,k-1}_\mathcal{W}} \lesssim 1$.
Noticing that $m(D) \widetilde{\psi}(2^{-j} D) = 2^{kd} \widetilde{h}(2^j\cdot)*$ with $\widetilde{h}= \mathcal{F}^{-1}(m \widetilde{\psi})\in \mathcal{S}(\R^d)$, we infer that
\begin{align}\label{exp:mD-comm0}
  &\quad  [m(D)\widetilde{\psi}(2^{-j}D), S_{j-1}\mathcal{W}\cdot]\nabla \Delta_j \phi(x) \nonumber \\
  & = \int_{\R^d} \widetilde{h}(y)  \big(S_{j-1}\mathcal{W}(x+2^{-j}y) - S_{j-1}\mathcal{W}(x)\big) \cdot \nabla \Delta_j \phi(x+2^{-j}y) \dd y \nonumber \\
  & = 2^{-j}\int_0^1 \int_{\R^d} \widetilde{h}(y) y\cdot \nabla S_{j-1}\mathcal{W}(x+ \tau 2^{-j}y)\cdot \nabla \Delta_j \phi(x+2^{-j}y) \dd y \dd \tau,
\end{align}
and by applying Lemmas \ref{lem:Rq} and \ref{lem:Tw-es} we obtain that for every $\lambda\in \{0,1,\cdots,\ell+1\}$,
\begin{align}\label{es:Tw-el-Xi3}
  2^{-q\epsilon} \|\Delta_q (T_{\mathcal{W}\cdot\nabla})^\lambda \Xi_3\|_{L^p}
  & \lesssim 2^{-q\epsilon} \sum_{j\in \N,j\sim q} \|\Delta_q (T_{\mathcal{W}\cdot\nabla})^\lambda \big([m(D)\widetilde{\psi}(2^{-j}D),S_{j-1}\mathcal{W}\cdot]\nabla \Delta_j \phi\big)\|_{L^p} \nonumber \\
  & \lesssim 2^{-q\epsilon} \sum_{j\in \N,j\sim q} \sum_{\lambda_1+\lambda_2\leq \lambda} 2^{-j}
  \|(T_{\mathcal{W}\cdot\nabla})^{\lambda_1}\nabla S_{j-1}\mathcal{W}\|_{L^\infty}
  \|(T_{\mathcal{W}\cdot\nabla})^{\lambda_2} \nabla \Delta_j \phi\|_{L^p} \nonumber \\
  & \lesssim \sum_{j\in\N,j\sim q} \sum_{\lambda_1+\lambda_2\leq \lambda} \Big( \sum_{j'\leq j-1}\|(T_{\mathcal{W}\cdot\nabla})^{\lambda_1} \nabla \Delta_{j'} \mathcal{W}\|_{L^\infty}\Big)
  2^{-j -j\epsilon} \|(T_{\mathcal{W}\cdot\nabla})^{\lambda_2} \Delta_j \nabla \phi \|_{L^p} \nonumber \\
  & \lesssim c_q \sum_{\lambda_1=0}^{\ell+1} \|\nabla \mathcal{W}\|_{\BB^{0,\lambda_1}_\mathcal{W}}
  \Big(\sum_{\mu_2=0}^{\ell+1-\lambda_1} \|(T_{\mathcal{W}\cdot\nabla})^{\mu_2}\nabla \phi\|_{B^{-1-\epsilon}_{p,r}}\Big) \nonumber \\
  & \lesssim c_q \sum_{\lambda_1=0}^{\ell+1} \|\nabla \mathcal{W}\|_{\BB^{0,\lambda_1}_\mathcal{W}} \|\nabla \phi\|_{\BB^{-1-\epsilon,\ell+1-\lambda_1}_{p,r,\mathcal{W}}} \nonumber \\
  & \lesssim c_q \sum_{\lambda_1=0}^{\ell+1} \|\mathcal{W}\|_{\BB^{1,\lambda_1}_\mathcal{W}} \|\phi\|_{\BB^{-\epsilon,\ell+1-\lambda_1}_{p,r,\mathcal{W}}}
  \lesssim c_q \| \mathcal{W}\|_{\bb^{1,\ell+1}_\mathcal{W}} \|\phi\|_{\bb^{-\epsilon,\ell+1}_{p,r,\mathcal{W}}},
\end{align}
with $\{c_q\}_{q\geq -1}$ satisfying $\|c_q\|_{\ell^r}=1$. It immediately leads to
\begin{align}\label{es:X3}
  \|\Xi_3\|_{\BB^{-\epsilon,\ell+1}_{p,r,\mathcal{W}}} \lesssim  \| \mathcal{W}\|_{\bb^{1,\ell+1}_\mathcal{W}} \|\phi\|_{\bb^{-\epsilon,\ell+1}_{p,r,\mathcal{W}}}.
\end{align}
For the remaining term $\Xi_4$, Lemma \ref{lem:Tw-es} together with the induction assumption ensure that
\begin{align}\label{es:X4}
  \|\Xi_4\|_{\BB^{-\epsilon,\ell+1}_{p,r,\mathcal{W}}} & \lesssim \|m(D)(T_{\mathcal{W}\cdot\nabla})\phi\|_{\bb^{-\epsilon,\ell+1}_{p,r,\mathcal{W}}} \nonumber \\
  & \lesssim \|(T_{\mathcal{W}\cdot\nabla})\phi\|_{\bb^{-\epsilon,\ell+1}_{p,r,\mathcal{W}}} +
  (1+ \|\mathcal{W}\|_{\bb^{1,\ell}_\mathcal{W}} ) \big(\|(T_{\mathcal{W}\cdot\nabla})\phi\|_{\bb^{-\epsilon,\ell}_{p,r,\mathcal{W}}}
  + \|\Delta_{-1}m(D)\phi\|_{L^p} \big) \nonumber \\
  & \lesssim \|\partial_\mathcal{W} \phi\|_{\bb^{-\epsilon,\ell+1}_{p,r,\mathcal{W}}}
  + \|T_{\nabla \phi}\cdot \mathcal{W}\|_{\bb^{-\epsilon,\ell+1}_{p,r,\mathcal{W}}}
  + \|R(\mathcal{W}\cdot,\nabla \phi)\|_{\bb^{-\epsilon,\ell+1}_{p,r,\mathcal{W}}}
  + \|\Delta_{-1}m(D)\phi\|_{L^p}  \nonumber \\
  & \lesssim \|\phi\|_{\bb^{-\epsilon,\ell+2}_{p,r,\mathcal{W}}}
  + \|T_{\nabla \phi}\cdot \mathcal{W}\|_{\BB^{-\epsilon,\ell+1}_{p,r,\mathcal{W}}}
  + \|R(\mathcal{W}\cdot,\nabla \phi)\|_{\BB^{-\epsilon,\ell+1}_{p,r,\mathcal{W}}}
  + \|\Delta_{-1}m(D) \phi\|_{L^p} \nonumber \\
  & \lesssim \|\phi\|_{\bb^{-\epsilon,\ell+2}_{p,r,\mathcal{W}}}
  + \|\phi\|_{\BB^{-\epsilon,\ell+1}_{p,r,\mathcal{W}}} \|\mathcal{W}\|_{\BB^{1,\ell+1}_\mathcal{W}}
  + \|\Delta_{-1}m(D)\phi\|_{L^p} \nonumber \\
  & \lesssim \|\phi\|_{\bb^{-\epsilon,\ell+2}_{p,r,\mathcal{W}}} + \|\mathcal{W}\|_{\bb^{1,\ell+1}_\mathcal{W}} \|\phi\|_{\bb^{-\epsilon,\ell+1}_{p,r,\mathcal{W}}}
  + \|\Delta_{-1}m(D)\phi\|_{L^p}.
\end{align}
Inserting estimates \eqref{decom:Xi}-\eqref{es:X4} into \eqref{es:mD-phi1} yields \eqref{eq:prod-es6} in the $(\ell+1)$-case,
and thus the induction method guarantees the desired inequality \eqref{eq:prod-es6}.

(3) We prove \eqref{eq:prod-es7} also by induction on the index $k$.
Note that for $k=0$, \eqref{eq:prod-es7} follows from \eqref{eq:str-es3-0}
with $C>0$ a universal constant (the norm $\|\mathcal{W}\|_{\bb^{1+\sigma,k-1}_{\infty,\mathcal{W}}}$ plays no role).

Now suppose that \eqref{eq:prod-es7} is satisfied for every $\ell\in \{0,1,\cdots, k-1\}$
(when $k=0$, for $\ell=0$) with $\ell$ in place of $k$-index, we intend to show that it also holds for $\ell+1$.
Using the decomposition \eqref{eq:mD-decom} and \eqref{decom:III} below, we have
\begin{align}\label{decom:mD-comm}
  [m(D),u\cdot\nabla]\phi & = \sum_{j\in \N} [m(D), S_{j-1}u\cdot\nabla]\Delta_j \phi + \sum_{j\in \N} [m(D), \Delta_j u\cdot\nabla] S_{j-1} \phi \nonumber \\
  & \quad + \sum_{j\geq 3}  m(D) \divg(\Delta_j u\,\widetilde{\Delta}_j \phi) - \sum_{j\geq 3} \divg(\Delta_j u\, m(D)\widetilde{\Delta}_j \phi)  + \sum_{-1\leq j\leq 2} [m(D),\Delta_j u\cdot\nabla] \widetilde{\Delta}_j \phi \nonumber \\
  & = : \mathcal{I} + \mathcal{II} + \mathcal{III}_1 + \mathcal{III}_2 + \mathcal{III}_3 .
\end{align}
It follows from \eqref{normBB-equiv} that
\begin{align}
  \|[m(D), u\cdot\nabla]\phi\|_{\bb^{-\epsilon,\ell+1}_{p,r,\mathcal{W}}}
  \lesssim \|(\mathcal{I}, \mathcal{II} , \mathcal{III}_1 , \mathcal{III}_2, \mathcal{III}_3)\|_{\BB^{-\epsilon,\ell+1}_{p,r,\mathcal{W}}}.
\end{align}
The estimation of $\mathcal{I}$ is quite similar to that of $\Xi_3$ in \eqref{es:Tw-el-Xi3}-\eqref{es:X3} (where $u$ plays a role as $\mathcal{W}$),
and one gets
\begin{align}
  \|\mathcal{I}\|_{\BB^{-\epsilon,\ell+1}_{p,r,\mathcal{W}}} \lesssim \sum_{\lambda_1=0}^{\ell+1}\|\nabla u\|_{\BB^{0,\lambda_1}_\mathcal{W}} \|\phi\|_{\BB^{-\epsilon,\ell+1-\lambda_1}_{p,r,\mathcal{W}}}
  \lesssim \|\nabla u\|_{\bb^{0,\ell+1}_\mathcal{W}} \|\phi\|_{\bb^{-\epsilon,\ell+1}_{p,r,\mathcal{W}}}.
\end{align}
Noting that $[m(D),\Delta_j u \cdot\nabla]S_{j-1}\phi$ has the expression formula \eqref{exp:mD-2} below,
and by arguing as \eqref{es:Tw-el-Xi3}, we find that for every $\lambda\in \{0,1,\cdots,\ell+1\}$,
\begin{align}
  & \quad 2^{-q\epsilon} \|\Delta_q (T_{\mathcal{W}\cdot\nabla})^\lambda \mathcal{II}\|_{L^p} \nonumber \\
  & \lesssim 2^{-q\epsilon}  \sum_{j\in \N,j\sim q}
  \| \Delta_q (T_{\mathcal{W}\cdot\nabla})^\lambda ([m(D) \widetilde{\psi}(2^{-j}D), \Delta_j u \cdot\nabla] S_{j-1}\phi) \|_{L^p} \nonumber \\
  & \lesssim 2^{-q \epsilon}  \sum_{j\in \N, j\sim q} \sum_{\lambda_1 + \lambda_2 \leq \lambda} 2^{-j}
  \|(T_{\mathcal{W}\cdot\nabla})^{\lambda_1} \nabla \Delta_j u\|_{L^\infty}
  \|(T_{\mathcal{W}\cdot\nabla})^{\lambda_2}\nabla S_{j-1} \phi\|_{L^p} \nonumber \\
  & \lesssim \sum_{j\in \N, j\sim q} \sum_{\lambda_1 + \lambda_2 \leq \lambda} 2^{-j(1+\epsilon)}
  \|(T_{\mathcal{W}\cdot\nabla})^{\lambda_1}\Delta_j\nabla u\|_{L^\infty}
  \Big(\sum_{j'\leq j-1} \|(T_{\mathcal{W}\cdot\nabla})^{\lambda_2} \nabla \Delta_{j'} \phi\|_{L^p} \Big) \nonumber \\
  & \lesssim c_q \sum_{\lambda_1 + \lambda_2 \leq \lambda} \Big(\sum_{\mu_1 =0}^{\lambda_1} \|(T_{\mathcal{W}\cdot\nabla})^{\mu_1} \nabla u\|_{B^0_{\infty,1}} \Big)
  \Big\|\sum_{j'\leq j-1} 2^{(j'-j)(1+\epsilon)} 2^{-j'(1+\epsilon)} \|(T_{\mathcal{W}\cdot\nabla})^{\lambda_2} \nabla \Delta_{j'}\phi\|_{L^p}\Big\|_{\ell^r} \nonumber \\
  & \lesssim c_q \|\nabla u\|_{\BB^{0,\ell+1}_\mathcal{W}} \sum_{\lambda_2=0}^\lambda \|\nabla \phi\|_{\BB^{-1-\epsilon,\lambda_2}_{p,r,\mathcal{W}}} \nonumber \\
  & \lesssim c_q \|\nabla u\|_{\BB^{0,\ell+1}_\mathcal{W}} \|\phi\|_{\BB^{-\epsilon,\ell+1}_{p,r,\mathcal{W}}}
  \lesssim c_q \|\nabla u\|_{\bb^{0,\ell+1}_\mathcal{W}} \|\phi\|_{\bb^{-\epsilon,\ell+1}_{p,r,\mathcal{W}}},  \nonumber
\end{align}
with $\{c_q\}_{q\geq -1}$ satisfying $\|c_q\|_{\ell^r}=1$. Then it directly leads to
\begin{align}
  \|\mathcal{II}\|_{\BB^{-\epsilon,\ell+1}_{p,r,\mathcal{W}}}
  \lesssim \|\nabla u\|_{\bb^{0,\ell+1}_\mathcal{W}} \|\phi\|_{\bb^{-\epsilon,\ell+1}_{p,r,\mathcal{W}}}.
\end{align}
For $\mathcal{III}_1$, by applying Lemmas \ref{lem:Rq} and \ref{lem:Tw-es}, we deduce that for every $\lambda\in \{0,1,\cdots,\ell+1\}$,
\begin{align}
  & \quad 2^{-q\epsilon}  \|\Delta_q (T_{\mathcal{W}\cdot\nabla})^\lambda \mathcal{III}_1\|_{L^p} \nonumber \\
  & \lesssim 2^{-q\epsilon}  \sum_{j\geq \max\{3,q -N_\lambda\}}
  \|\Delta_q (T_{\mathcal{W}\cdot\nabla})^\lambda \nabla m(D) (\Delta_j u\, \widetilde{\Delta}_j \phi)\|_{L^p} \nonumber \\
  & \lesssim 2^{q(1- \epsilon)} \sum_{\lambda_1=0}^\lambda  \Big( \sum_{j\geq \max\{3,q-N_\lambda\}}
  \|(T_{\mathcal{W}\cdot\nabla})^{\lambda_1} (\Delta_j u\, \widetilde{\Delta}_j \phi)\|_{L^p} \Big) \nonumber \\
  & \lesssim  2^{q(1-\epsilon)} \sum_{\lambda_1=0}^\lambda \sum_{j\geq \max\{3,q-N_\lambda\}} \sum_{\lambda_2+\lambda_3 \leq \lambda_1}
  \|(T_{\mathcal{W}\cdot\nabla})^{\lambda_2} \Delta_j u\|_{L^\infty} \|(T_{\mathcal{W}\cdot\nabla})^{\lambda_3} \widetilde{\Delta}_j \phi\|_{L^p} \nonumber \\
  & \lesssim \sum_{\lambda_2+\lambda_3 \leq \ell+1} \sum_{j\geq \max\{3,q- N_\lambda\}} 2^{(q-j)(1-\epsilon)}
  2^j \|(T_{\mathcal{W}\cdot\nabla})^{\lambda_2}\Delta_j u\|_{L^\infty}
  2^{-j\epsilon} \|(T_{\mathcal{W}\cdot\nabla})^{\lambda_3} \widetilde{\Delta}_j \phi\|_{L^p} \nonumber \\
  & \lesssim \sum_{\lambda_2 =0}^{\ell+1} \sum_{j\geq \max\{3,q-N_\lambda\}} 2^{(q-j)(1-\epsilon)}
  \Big( \sum_{j_1\sim j} \sum_{\mu_2=0}^{\lambda_2}  \|(T_{\mathcal{W}\cdot\nabla})^{\mu_2} \Delta_{j_1}\nabla u\|_{L^\infty} \Big) \|\phi\|_{\BB^{-\epsilon,\ell+1}_{p,\infty,\mathcal{W}}} \nonumber \\
  & \lesssim c_q \|\nabla u\|_{\BB^{0,\ell+1}_\mathcal{W}} \|\phi\|_{\BB^{-\epsilon,\ell+1}_{p,r,\mathcal{W}}}
  \lesssim c_q \|\nabla u\|_{\bb^{0,\ell+1}_\mathcal{W}} \|\phi\|_{\bb^{-\epsilon,\ell+1}_{p,r,\mathcal{W}}},\nonumber
\end{align}
which guarantees that
\begin{align}
  \|\mathcal{III}_1 \|_{\BB^{-\epsilon,\ell+1}_{p,r,\mathcal{W}}} \lesssim \|\nabla u\|_{\bb^{0,\ell+1}_\mathcal{W}} \|\phi\|_{\bb^{-\epsilon,\ell+1}_{p,r,\mathcal{W}}}.
\end{align}
For $\mathcal{III}_2$, noting that $\widetilde\Delta_j S_1 =0$ for every $j\geq 3$ (with $S_1$ defined by \eqref{exp:Del-Sj}), similarly as above we infer that for each $\lambda\in \{0,1,\cdots,\ell+1\}$,
\begin{align}
  &\quad 2^{-q\epsilon} \|\Delta_q (T_{\mathcal{W}\cdot\nabla})^\lambda \mathcal{III}_2\|_{L^p} \nonumber \\
  & \lesssim 2^{-q\epsilon}  \sum_{j\geq \max\{3,q- N_\lambda\}}
  \|\Delta_q (T_{\mathcal{W}\cdot\nabla})^\lambda \divg(\Delta_j u\, \widetilde{\Delta}_j m(D) \phi)\|_{L^p} \nonumber \\
  & \lesssim 2^{q (1-\epsilon)} \sum_{\lambda_1=0}^\lambda \sum_{j\geq \max\{3,q-N_\lambda\}}
  \|(T_{\mathcal{W}\cdot\nabla})^{\lambda_1}(\Delta_j u\, \widetilde{\Delta}_j m(D) \phi)\|_{L^p} \nonumber \\
  & \lesssim 2^{q(1-\epsilon)} \sum_{\lambda_1=0}^\lambda \sum_{j\geq \max\{3,q-N_\lambda\}} \sum_{\lambda_2+\lambda_3\leq \lambda}
  \|(T_{\mathcal{W}\cdot\nabla})^{\lambda_2} \Delta_j u\|_{L^\infty} \|(T_{\mathcal{W}\cdot\nabla})^{\lambda_3} \widetilde{\Delta}_j m(D) \phi\|_{L^p} \nonumber \\
  & \lesssim \sum_{\lambda_2 +\lambda_3 =0}^{\ell+1} \sum_{j\geq \max\{3,q-N_\lambda\}} 2^{(q-j)(1-\epsilon)}
  2^j \|(T_{\mathcal{W}\cdot\nabla})^{\lambda_2}\Delta_j u\|_{L^\infty}
  2^{-j\epsilon} \|(T_{\mathcal{W}\cdot\nabla})^{\lambda_3} \widetilde{\Delta}_j m(D) \phi\|_{L^p} \nonumber \\
  & \lesssim \sum_{\lambda_2=0}^{\ell+1} \sum_{j\geq \max\{3,q-N_\lambda\}} 2^{(q-j)(1-\epsilon)}
  \Big(\sum_{j_1\sim j} \sum_{\mu_2=0}^{\lambda_2} \|(T_{\mathcal{W}\cdot\nabla})^{\mu_2} \Delta_{j_1} \nabla u\|_{L^\infty} \Big)
  \|m(D)(\Id -S_1)\phi\|_{\BB^{-\epsilon,\ell+1}_{p,r,\mathcal{W}}} \nonumber \\
  & \lesssim c_q \|\nabla u\|_{\BB^{0,\ell+1}_\mathcal{W}} \|m(D) (\Id -S_1) \phi\|_{\BB^{-\epsilon,\ell+1}_{p,r,\mathcal{W}}}
  \lesssim c_q \| \nabla u\|_{\bb^{0,\ell+1}_\mathcal{W}} \|m(D) (\Id -S_1) \phi\|_{\bb^{-\epsilon,\ell+1}_{p,r,\mathcal{W}}}, \nonumber
\end{align}
and by using \eqref{eq:prod-es6} and the fact that $\Delta_{-1}(\Id -S_1)=0$,
\begin{align}\label{es:BB-LP}
  \|m(D) (\Id -S_1) \phi\|_{\bb^{-\epsilon,\ell+1}_{p,r,\mathcal{W}}}
  & \lesssim \|(\Id -S_1)\phi\|_{\bb^{-\epsilon,\ell+1}_{p,r,\mathcal{W}}} \nonumber \\
  & \lesssim \|\phi\|_{\bb^{-\epsilon,\ell+1}_{p,r,\mathcal{W}}} + \|S_1\phi\|_{\BB^{-\epsilon,\ell+1}_{p,r,\mathcal{W}}} \nonumber \\
  &  \lesssim \|\phi\|_{\bb^{-\epsilon,\ell+1}_{p,r,\mathcal{W}}} + \sum_{l=-1,0}\sum_{ q =-1}^{N_\ell} \sum_{\lambda=0}^{\ell+1} \|\Delta_q (T_{\mathcal{W}\cdot\nabla})^\lambda \Delta_l \phi\|_{L^p} \nonumber \\
  & \lesssim \|\phi\|_{\bb^{-\epsilon,\ell+1}_{p,r,\mathcal{W}}} + \|\mathcal{W}\|_{L^\infty}\sum_{l=-1,0}\sum_{q=-1}^{N_\ell} \sum_{q_1\sim q} \sum_{\lambda=1}^{\ell+1} \|\Delta_{q_1} (T_{\mathcal{W}\cdot\nabla})^{\lambda-1} \Delta_l\phi \|_{ L^p} \nonumber \\
  & \lesssim \|\phi\|_{\bb^{-\epsilon,\ell+1}_{p,r,\mathcal{W}}} + \sum_{l=-1,0}\|\Delta_l\phi\|_{L^p} \lesssim  \|\phi\|_{\bb^{-\epsilon,\ell+1}_{p,r,\mathcal{W}}},
\end{align}
thus the above two estimates yield that
\begin{align}
  \|\mathcal{III}_2\|_{\BB^{-\epsilon,\ell+1}_{p,r,\mathcal{W}}} \lesssim \|\nabla u\|_{\bb^{0,\ell+1}_\mathcal{W}} \|\phi\|_{\bb^{-\epsilon,\ell+1}_{p,r,\mathcal{W}}}.
\end{align}
Arguing as \eqref{es:BB-LP} and using the following fact (see e.g. Proposition 3.1 of \cite{HKR11}) that
\begin{align}\label{eq:fact3}
  \textrm{$\nabla \Delta_{-1} m(D)$ is bounded on $L^p$ for every $p\in [1,\infty]$},
\end{align}
the term $\mathcal{III}_3$ can be directly estimated as follows
\begin{align}\label{es:III3-BB}
  & \quad \|\mathcal{III}_3\|_{\BB^{-\epsilon,\ell+1}_{p,r,\mathcal{W}}} \nonumber \\
  & \lesssim \sum_{q=-1}^{N_\ell} \sum_{\lambda=0}^{\ell+1} \sum_{j=-1}^2
  \Big(\|\Delta_q (T_{\mathcal{W}\cdot\nabla})^\lambda m(D) \divg (\Delta_j u \,\widetilde{\Delta}_j\phi)\|_{L^p}
  + \|\Delta_q (T_{\mathcal{W}\cdot\nabla})^\lambda (\Delta_j u\cdot \nabla m(D)\widetilde{\Delta}_j \phi)\|_{L^p} \Big) \nonumber \\
  & \lesssim \sum_{j=-1}^2 \Big( \|m(D)\divg (\Delta_j u\, \widetilde{\Delta}_j \phi)\|_{L^p} + \|\Delta_j u\cdot\nabla m(D) \widetilde\Delta_j \phi\|_{L^p}\Big) \nonumber \\
  & \lesssim \|u\|_{L^\infty} \Big(\sum_{-1\leq j\leq 2} \|\widetilde\Delta_j \phi\|_{L^p} \Big) \lesssim \|u\|_{L^\infty} \|\phi\|_{\bb^{-\epsilon,\ell+1}_{p,r,\mathcal{W}}}.
\end{align}
Hence, gathering the above estimates \eqref{decom:mD-comm}--\eqref{es:III3-BB} leads to the desired inequality \eqref{eq:prod-es7}.

\end{proof}

\subsection{Proof of Lemma \ref{lem:prod-es}}\label{subsec:str-es2}

In the $k=1$ case of Lemmas \ref{lem:Rq} and \ref{lem:Tw-es}, we can show the explicit dependence of constant $C$ on $\|W\|_{B^1_{\infty,1}}$ as follows,
and one can see Subsection \ref{subsec:lem5.3} for the proof.
\begin{lemma}\label{lem:Tw-es-k1}
  Let $\mathcal{W}=\{W_i\}_{1\leq i\leq N}$ ($N\in \mathbb{Z}^+$) be a set of regular divergence-free vector fields of $\R^d$.
Let $(p,r)\in [1,\infty]^2$. The following statements are satisfied.
\begin{enumerate}[(1)]
\item Let $\alpha_i$ ($i=1,\cdots,m$) be a function satisfying that $\mathrm{supp}\,\widehat{\alpha_i} \subset B(0,C_i 2^q)$, $q\in\N$.
Then we have
\begin{align}\label{eq:Tw-Rq-2}
  & \quad  \big\|(T_{\mathcal{W}\cdot\nabla}) R_q(\alpha_1,\cdots,\alpha_m)\big\|_{L^p} \nonumber \\
  & \leq C \sum_{i=1}^m \min\bigg\{ \|(T_{\mathcal{W}\cdot\nabla}) \alpha_i\|_{L^p} \prod_{j\neq i} \|\alpha_j\|_{L^\infty}, \|(T_{\mathcal{W}\cdot\nabla})\alpha_i\|_{L^\infty} \min_{1\leq j\leq m,j\neq i}\Big(\|\alpha_j\|_{L^p} \prod_{l\neq i,j}\|\alpha_l\|_{L^\infty} \Big) \bigg\} \nonumber \\
  & \quad + C \|\mathcal{W}\|_{B^1_{\infty,1}} \min_{1\leq i\leq m} \Big(\|\alpha_i\|_{L^p} \prod_{j\neq i} \|\alpha_j\|_{L^\infty}\Big).
\end{align}
\item
We have that for every $s \in \R$,
\begin{align}\label{eq:Tw-es-key.2}
  &\;  \big\| \{2^{q s} \|(T_{\mathcal{W}\cdot\nabla}) \Delta_q \phi\|_{L^p}\}_{q\geq -1} \big\|_{\ell^r}
  + \big\| \{2^{q (s-1)} \|(T_{\mathcal{W}\cdot\nabla}) \nabla \Delta_q \phi\|_{L^p} \}_{q\geq -1} \big\|_{\ell^r}  \nonumber \\
  \leq &
  C \| (T_{\mathcal{W}\cdot\nabla}) \phi\|_{B^s_{p,r}} + C \|\mathcal{W}\|_{W^{1,\infty}} \|\phi\|_{B^s_{p,r}} ,
\end{align}
and for every $s< 0$,
\begin{align}\label{eq:Tw-es-key.22}
  \big\|\{ 2^{qs} \|(T_{\mathcal{W}\cdot\nabla})S_{q-1}\phi\|_{L^p} \}_{q\in \N} \big\|_{\ell^r}
  \leq C \|(T_{\mathcal{W}\cdot\nabla}) \phi\|_{B^s_{p,r}} + C \|\mathcal{W}\|_{W^{1,\infty}} \|\phi\|_{B^s_{p,r}} ,
\end{align}
and for every $q\geq -1$,
\begin{align}\label{eq:Tw-es-key2.2}
  \|\Delta_q (T_{\mathcal{W}\cdot\nabla}) \nabla \phi\|_{L^p} \leq C 2^q \|\Delta_q (T_{\mathcal{W}\cdot\nabla}) \phi\|_{L^p}
  + C 2^q \|\mathcal{W}\|_{W^{1,\infty}} \Big(\sum_{q_1\in \N, |q_1-q|\leq 5} \|\Delta_{q_1}\phi\|_{L^p} \Big),
\end{align}
and for every $q\in \N$,
\begin{align}\label{eq:Tw-es-key3.2}
  2^q \| (T_{\mathcal{W}\cdot\nabla}) \Delta_q \phi \|_{L^p} \leq C \sum_{q_1\in \N,|q_1-q|\leq 5} \Big(\|(T_{\mathcal{W}\cdot\nabla}) \Delta_{q_1} \nabla \phi\|_{L^p}
  + \|\mathcal{W}\|_{W^{1,\infty}} \|\Delta_{q_1}\nabla \phi\|_{L^p} \Big).
\end{align}
\item 
We have that for every $s< 0$,
\begin{equation}\label{es:Tw-parap-1-2}
  \|(T_{\mathcal{W}\cdot\nabla})T_{\nabla \phi}\cdot u\|_{B^s_{p,r}}  + \|(T_{\mathcal{W}\cdot\nabla})T_{u\cdot} \nabla \phi\|_{B^s_{p,r}}
  \leq C \big(B_1(s) + B_2(s) + B_3(s) \big),
\end{equation}
with
\begin{align}
  B_1(s):=& \min\big\{\|u\|_{B^0_{\infty,1}} \|(T_{\mathcal{W}\cdot\nabla})\nabla \phi\|_{B^s_{p,r}}, \|u\|_{B^s_{p,r}} \|(T_{\mathcal{W}\cdot\nabla})\nabla \phi\|_{B^0_{\infty,1}} \big\}, \label{B1} \\
  B_2(s):=& \min\big\{\|(T_{\mathcal{W}\cdot\nabla})u\|_{B^0_{\infty,1}} \|\nabla \phi\|_{B^s_{p,r}} , \|(T_{\mathcal{W}\cdot\nabla})u\|_{B^s_{p,r}} \|\nabla \phi\|_{B^0_{\infty,1}} \big\}, \label{B2}\\
  B_3(s):=& \|\mathcal{W}\|_{B^1_{\infty,1}} \min\big\{\|u\|_{B^0_{\infty,1}} \|\nabla \phi\|_{B^s_{p,r}}, \|u\|_{B^s_{p,r}} \|\nabla \phi\|_{B^0_{\infty,1}} \big\}. \label{B3}
\end{align}
\item Assume that $u$ is a divergence-free vector field of $\R^d$, then we have that for every $s>-1$,
\begin{equation}\label{es:Tw-rem-1-2}
\begin{split}
  \|(T_{\mathcal{W}\cdot\nabla})R(u\cdot,\nabla \phi)\|_{B^s_{p,r}}
  &  \leq C \big(B_1(s) + B_2(s) + B_3(s)\big).
\end{split}
\end{equation}
\item For every $s\in (-1,1)$, we have
\begin{equation}\label{normBB-equiv-2}
   \|\partial_\mathcal{W} \phi - (T_{\mathcal{W}\cdot\nabla})\phi\|_{B^s_{p,r}}
   \leq C \|\mathcal{W}\|_{W^{1,\infty}} \|\phi\|_{B^s_{p,r}}.
\end{equation}
\end{enumerate}
\end{lemma}

Now relied on Lemma \ref{lem:Tw-es-k1}, we present the proof of Lemma \ref{lem:prod-es}.
\begin{proof}[Proof of Lemma \ref{lem:prod-es}]
  (1) Owing to \eqref{es:Tw-parap-1-2}, \eqref{es:Tw-rem-1-2} and Bony's decomposition, we see that 
\begin{align}\label{es:Tw-uphi-1}
  \|(T_{\mathcal{W}\cdot\nabla})(u\cdot\nabla \phi)\|_{B^{-\epsilon}_{p,r}} &
  \leq \|(T_{\mathcal{W}\cdot\nabla})(T_{u\cdot}\nabla \phi)\|_{B^{-\epsilon}_{p,r}}
  + \|(T_{\mathcal{W}\cdot\nabla})(T_{\nabla \phi}\cdot u)\|_{B^{-\epsilon}_{p,r}} + \|(T_{\mathcal{W}\cdot\nabla})R(u\cdot,\nabla\phi)\|_{B^{-\epsilon}_{p,r}} \nonumber \\
  & \leq C\big( B_1(-\epsilon)+  B_2(-\epsilon) +  B_3(-\epsilon)\big) \nonumber \\
  & \leq C \min\Big\{ \|u\|_{B^0_{\infty,1}} \big(\|\partial_\mathcal{W} \nabla \phi\|_{B^{-\epsilon}_{p,r}} + \|\mathcal{W}\|_{B^1_{\infty,1}} \|\nabla \phi\|_{B^{-\epsilon}_{p,r}}\big), \nonumber\\
  & \qquad\qquad\quad \|u\|_{B^{-\epsilon}_{p,r}} \big( \|\partial_\mathcal{W} \nabla \phi\|_{B^0_{\infty,1}}
  + \|\mathcal{W}\|_{B^1_{\infty,1}} \|\nabla \phi\|_{B^0_{\infty,1}} \big) \Big\} \nonumber \\
  &\quad + C \min\Big\{ \big(\|\partial_\mathcal{W} u\|_{B^0_{\infty,1}} + \|\mathcal{W}\|_{B^1_{\infty,1}} \|u\|_{B^0_{\infty,1}}\big)\|\nabla \phi\|_{B^{-\epsilon}_{p,r}}, \nonumber \\
  &\qquad \qquad\qquad  \big(\| \partial_\mathcal{W} u\|_{B^{-\epsilon}_{p,r}} + \|\mathcal{W}\|_{B^1_{\infty,1}} \|u\|_{B^{-\epsilon}_{p,r}} \big)
  \|\nabla \phi\|_{B^0_{\infty,1}} \Big\} \nonumber \\
  & \leq C \min\{A_1,  A_2,  A_3\} ,
\end{align}
where $B_1,B_2,B_3$ are given by \eqref{B1}--\eqref{B3} and $A_1,A_2,A_3$ are defined by \eqref{A1}--\eqref{A3}.

By using \eqref{eq:prod-es}, \eqref{es:Tw-uphi-1} and \eqref{normBB-equiv-2}, we directly deduce that
\begin{align*}
  \|\partial_\mathcal{W} (u\cdot\nabla \phi)\|_{B^{-\epsilon}_{p,r}} & \leq
  \|(T_{\mathcal{W}\cdot\nabla})(u\cdot\nabla \phi)\|_{B^{-\epsilon}_{p,r}} + \|\partial_\mathcal{W}(u\cdot\nabla \phi)-(T_{\mathcal{W}\cdot\nabla})(u\cdot\nabla\phi)\|_{B^{-\epsilon}_{p,r}} \\
  & \leq C \min\{A_1, A_2, A_3\} + C B_3(-\epsilon)
  \leq C \min\{A_1, A_2, A_3\}.
\end{align*}

(2) Thanks to \eqref{normBB-equiv-2} and the unform boundedness property of $m(D) \Delta_j$ ($j\in \N$) on $L^p$,
we see that
\begin{align}\label{es:str-es1-1}
  \|\partial_\mathcal{W} (m(D) \phi)\|_{B^{-\epsilon}_{p,r}} & \leq  \|(T_{\mathcal{W}\cdot\nabla}) m(D)\phi\|_{B^{-\epsilon}_{p,r}}
  + C \|\mathcal{W}\|_{B^1_{\infty,1}} \|m(D)\phi\|_{B^{-\epsilon}_{p,r}} \nonumber \\
  & \leq \|(T_{\mathcal{W}\cdot\nabla}) m(D)\phi\|_{B^{-\epsilon}_{p,r}}
  +  C \|\mathcal{W}\|_{B^1_{\infty,1}}  \big( \|\Delta_{-1}m(D)\phi\|_{L^p} + \|\phi\|_{B^{-\epsilon}_{p,r}} \big).
\end{align}

Note that there exists a bump function $\widetilde{\psi} \in C^\infty_c(\R^d)$ supported on an annulus so that \eqref{exp:Tw-mD0} holds,
we have that for every $q\geq -1$,
\begin{align}\label{J1qJ2q}
  2^{-q\epsilon} \|\Delta_q (T_{\mathcal{W}\cdot \nabla})m(D)\phi\|_{L^p}
  & \lesssim 2^{-q\epsilon} \sum_{j\in \N, |j-q|\leq 4} \|\Delta_q \big([m(D)\widetilde{\psi}(2^{-j}D),
  S_{j-1}\mathcal{W}\cdot\nabla] \Delta_j \phi\big)\|_{L^p} \nonumber \\
  & \quad + 2^{-q\epsilon} \big\|\Delta_q m(D) (T_{\mathcal{W}\cdot\nabla}) \phi\big\|_{L^p} =: J_{1,q} + J_{2,q} .
\end{align}
Thanks to \eqref{exp:mD-comm0}, we immediately get
\begin{align}\label{es:J1q}
  J_{1,q} \lesssim 2^{-q\epsilon} \sum_{j\in\N,|j-q|\leq 4} 2^{-j} \|\nabla S_{j-1}\mathcal{W}\|_{L^\infty} \|\nabla \Delta_j \phi\|_{L^p}  \lesssim c_q \|\mathcal{W}\|_{W^{1,\infty}} \|\phi\|_{B^{-\epsilon}_{p,r}},
\end{align}
with $\{c_q\}_{q\geq -1}$ satisfying $\|c_q\|_{\ell^r} =1$. For $J_{2,q}$,
by virtue of \eqref{eq:fact3} and \eqref{normBB-equiv-2} again,
we find
\begin{align}\label{es:J2q}
  & \quad \|J_{2,q}\|_{\ell^r(q\geq -1)} \lesssim \|m(D) (T_{\mathcal{W}\cdot\nabla}) \phi\|_{B^{-\epsilon}_{p,r}} \nonumber \\
  & \lesssim \sum_{0\leq j\leq 3}\|\Delta_{-1}m(D) \div \big(S_{j-1} \mathcal{W} \, \Delta_j \phi\big)\|_{L^p}
  +  \|(T_{\mathcal{W}\cdot\nabla}) \phi\|_{B^{-\epsilon}_{p,r}}
  \lesssim \|\mathcal{W}\|_{B^1_{\infty,1}} \|\phi\|_{B^{-\epsilon}_{p,r}} + \|\partial_\mathcal{W} \phi\|_{B^{-\epsilon}_{p,r}}.
\end{align}
Combining \eqref{es:J1q} and \eqref{es:J2q} leads to
\begin{align}\label{es:Tw-mD1}
  \|(T_{\mathcal{W}\cdot\nabla})m(D)\phi\|_{B^{-\epsilon}_{p,r}} \lesssim \|\mathcal{W}\|_{B^1_{\infty,1}}
  \|\phi\|_{B^{-\epsilon}_{p,r}}   + \|\partial_\mathcal{W} \phi\|_{B^{-\epsilon}_{p,r}},
\end{align}
which together with \eqref{es:str-es1-1} guarantees the desired estimate \eqref{eq:str-es1}.


(3) Bony's decomposition gives that
\begin{align}\label{eq:mD-decom}
  [m(D),u\cdot\nabla] \phi & = \sum_{j\in \N} [m(D), S_{j-1}u\cdot\nabla]\Delta_j \phi + \sum_{j\in \N} [m(D), \Delta_j u\cdot\nabla] S_{j-1} \phi + \sum_{j\geq -1} [m(D), \Delta_j u\cdot\nabla]\widetilde{\Delta}_j \phi \nonumber \\
  & = \mathcal{I} + \mathcal{II} + \mathcal{III}.
\end{align}
Similarly as \eqref{exp:Tw-mD0}, \eqref{exp:mD-comm0}, there exists $\widetilde{h}\in \mathcal{S}(\R^d)$ so that
\begin{align}\label{exp:mD-1}
  [m(D), S_{j-1}u\cdot\nabla] \Delta_j \phi 
  = 2^{-j}\int_0^1 \int_{\R^d} \widetilde{h}(y) y\cdot \nabla S_{j-1}u(x+ \delta 2^{-j}y)\cdot \nabla \Delta_j \phi(x+2^{-j}y) \dd y \dd \delta,
\end{align}
and
\begin{align}\label{exp:mD-2}
  [m(D), \Delta_j u\cdot\nabla] S_{j-1} \phi = 2^{-j}\int_0^1 \int_{\R^d} \widetilde{h}(y) y\cdot \nabla \Delta_j u(x+ \delta 2^{-j}y)\cdot \nabla S_{j-1} \phi(x+2^{-j}y) \dd y \dd \delta.
\end{align}
Thus by using \eqref{eq:fact}, we directly get
\begin{align}\label{es:I-1}
  2^{-q\epsilon} \|\Delta_q \mathcal{I}\|_{L^p}
  \lesssim 2^{-q\epsilon}  \sum_{j\in \N, |q-j|\leq 4} \|\Delta_q [m(D), S_{j-1}u\cdot\nabla] \Delta_j \phi \|_{L^p}
  \lesssim c_q \|\nabla u\|_{L^\infty} \|\phi\|_{B^{-\epsilon}_{p,r}},
\end{align}
and
\begin{align}
  2^{-q\epsilon}  \|\Delta_q \mathcal{II}\|_{L^p} &
  \lesssim 2^{-q\epsilon} \sum_{j\in \N, |q-j|\leq 4} \|\Delta_q [m(D), \Delta_j u\cdot\nabla] S_{j-1} \phi\|_{L^p} \nonumber \\
  & \lesssim \|\nabla u\|_{L^\infty} 2^{-q\epsilon} \sum_{j\in \N,|q-j|\leq 4} \|S_{j-1} \phi\|_{L^p}
  \lesssim c_q \|\nabla u\|_{L^\infty} \|\phi\|_{B^{-\epsilon}_{p,r}},
\end{align}
where $\{c_q\}_{q\geq -1}$ is such that $\|c_q\|_{\ell^r} =1$.
We further decompose the term $\mathcal{III}$ as
\begin{align}\label{decom:III}
  \mathcal{III} & = \sum_{j\geq 2}  m(D) \divg(\Delta_j u\,\widetilde{\Delta}_j \phi) - \sum_{j\geq 2} \divg(\Delta_j u\, m(D)\widetilde{\Delta}_j \phi)  + \sum_{-1\leq j\leq 1} [m(D),\Delta_j u\cdot\nabla] \widetilde{\Delta}_j \phi \nonumber \\
  & = : \mathcal{III}_1 + \mathcal{III}_2 + \mathcal{III}_3.
\end{align}
For $\mathcal{III}_1$ and $\mathcal{III}_2$, thanks to \eqref{eq:fact3} and the discrete Young's inequality, we infer that
\begin{align}
  & \quad 2^{-q \epsilon}  \|\Delta_q \mathcal{III}_1\|_{L^p}
  \lesssim 2^{-q\epsilon} \sum_{j\geq q-3, j\geq 2} \|\Delta_q m(D)\divg(\Delta_j u\,\widetilde{\Delta}_j \phi) \|_{L^p} \nonumber \\
  & \lesssim 2^{q(1-\epsilon)} \sum_{j\geq q-3, j\geq 2} 2^{-j(1-\epsilon)} 2^j\|\Delta_j u\|_{L^\infty} 2^{-j\epsilon}\|\widetilde{\Delta}_j \phi\|_{L^p}
  \lesssim c_q \|\nabla u\|_{L^\infty} \|\phi\|_{B^{-\epsilon}_{p,r}},
\end{align}
and
\begin{align}
  &\quad 2^{-q\epsilon} \|\Delta_q \mathcal{III}_2\|_{L^p}
  \lesssim 2^{q (1- \epsilon)} \sum_{j\geq q-3,j\geq 2} \| \Delta_j u\|_{L^\infty} \|m(D) \widetilde{\Delta}_j \phi\|_{L^p} \nonumber \\
  & \lesssim \sum_{j\geq q-3,j\geq 2} 2^{(q-j)(1-\epsilon)} \|\nabla \Delta_j u\|_{L^\infty} 2^{-j\epsilon} \|\widetilde{\Delta}_j\phi\|_{L^p}
  \lesssim c_q \|\nabla u\|_{L^\infty} \|\phi\|_{B^{-\epsilon}_{p,r}}.
\end{align}
For term $\mathcal{III}_3$, we do not use the commutator structure, and by virtue of \eqref{eq:fact3} we find
\begin{align}\label{es:III3}
  & \quad \|\{2^{-q\epsilon} \|\Delta_q \mathcal{III}_3\|_{L^p}\}_{q\geq -1}\|_{\ell^r} \nonumber \\
  & \lesssim \sum_{-1\leq q\leq 5} \sum_{-1\leq j \leq 1} \Big(\| \Delta_q m(D) \divg (\Delta_j u\, \widetilde{\Delta}_j \phi)\|_{L^p}
  + \|\Delta_q(\Delta_j u \cdot \nabla m(D)\widetilde{\Delta}_j \phi)\|_{L^p} \Big) \nonumber \\
  & \lesssim \sum_{-1\leq j\leq 1} \|\Delta_j u\|_{L^\infty} \|\widetilde{\Delta}_j \phi\|_{L^p}
  \lesssim \|u\|_{L^\infty} \|\phi\|_{B^{-\epsilon}_{p,r}}.
\end{align}
Gathering the above estimates \eqref{es:I-1}--\eqref{es:III3} leads to the desired estimate \eqref{eq:str-es3-0}.

\end{proof}

\section{Proof of Lemmas \ref{lem:Rq} - \ref{lem:Tw-es-k1}}\label{sec:str-es-Lem}

\subsection{Proof of Lemma \ref{lem:Rq}}\label{subsec:lem5.1}
  The proof is via the induction method. It is clear that \eqref{eq:Tw-Rq}--\eqref{eq:BB-nab-es} are satisfied for $\ell=0$.
Assume that \eqref{eq:Tw-Rq}--\eqref{eq:BB-nab-es} hold for all $\ell' \in \{0,1,\cdots,\ell\}$ with $ \ell\in \{0,1,\cdots, k-1\}$.
Next we will prove that \eqref{eq:Tw-Rq}--\eqref{eq:BB-nab-es} hold for $\ell+1$ (replacing the $\ell$-index in \eqref{eq:Tw-Rq}--\eqref{eq:BB-nab-es}).
Observe that
\begin{align}\label{eq:Ra123}
  & (T_{\mathcal{W}\cdot\nabla}) R_q(\alpha_1,\cdots,\alpha_m) =
  \sum_{ -1\leq q_1\leq q + N_0} S_{q_1-1}\mathcal{W} \cdot \nabla \Delta_{q_1} R_q(\alpha_1,\cdots,\alpha_m) \nonumber \\
  & = \sum_{-1\leq q_1\leq q + N_0} \big( S_{q_1-1}\mathcal{W} - S_{q-1} \mathcal{W} \big)\cdot \nabla \Delta_{q_1} R_q(\alpha_1,\cdots,\alpha_m)  \nonumber \\
  & \quad + \sum_{ -1\leq q_1\leq q+ N_0}\big(S_{q-1}\mathcal{W} -S_{q_1-1}\mathcal{W}\big)\cdot
  \sum_{i=1}^m R_q(\alpha_1,\cdots, \nabla \Delta_{q_1} \alpha_i,\cdots, \alpha_m) \nonumber \\
  & \quad + \sum_{ -1\leq q_1\leq q+ N_0} S_{q_1-1} \mathcal{W} \cdot \sum_{i=1}^m R_q(\alpha_1,\cdots, \nabla \Delta_{q_1} \alpha_i ,\cdots,\alpha_m)  \nonumber \\
  & = : R_q^1 + R_q^2 + R_q^3 .
\end{align}
Note that
\begin{align*}
  S_{q_1-1}\mathcal{W} (x) = S_{q_1-1} \mathcal{W}(x + 2^{-q} f_i(\tau) y) - \int_0^1 2^{-q} f_i(\tau) y \cdot \nabla S_{q_1-1} \mathcal{W}(x + \delta 2^{-q} f_i(\tau) y) \dd \delta ,
\end{align*}
we can further decompose $R_q^3$ as
\begin{align}\label{eq:Rq45}
  R_q^3 & = \sum_{i=1}^m R_q(\alpha_1,\cdots, (T_{\mathcal{W}\cdot\nabla})\alpha_i,\cdots, \alpha_m)
  + \sum_{ -1\leq q_1\leq q+ N_0} \sum_{i=1}^m
  R_q^{(i)}(\nabla S_{q_1-1}\mathcal{W},\alpha_1, \cdots, \nabla \Delta_{q_1} \alpha_i,\cdots, \alpha_m) \nonumber \\
  & =: R_q^4 + R_q^5,
\end{align}
with
\begin{align*}
  & R_q^{(i)}(\nabla S_{q_1-1}\mathcal{W},\alpha_1, \cdots, \nabla \Delta_{q_1} \alpha_i,\cdots, \alpha_m) \\
  & : = \int_0^1 \int_{[0,1]^m} \int_{\R^d} \nabla S_{q_1-1} \mathcal{W}(x + \delta 2^{-q} f_i(\tau) y) \cdot \nabla \Delta_{q_1} \alpha_i(x + 2^{-q}f_i(\tau) y) \\
  &\quad \times\prod_{j\neq i} \alpha_j(x + 2^{-q}f_j(\tau)y)\cdot 2^{-q} y f_i(\tau) h(\tau,y) \,\dd y \dd \tau \dd \delta .
\end{align*}

We first consider $(T_{W\cdot\nabla})^\ell R_q^1$. Denoting by $\nabla \Delta_{q_1} = i D \varphi(2^{-q_1}D) = : i 2^{q_1} \varphi_1(2^{-q_1}D)$ for $q_1\in \N$, we infer that
\begin{align*}
  & \|(T_{\mathcal{W}\cdot\nabla})^\ell R_q^1 \|_{L^p} \leq \sum_{-1\leq q_1 \leq q+ N_0}
  \Big\|(T_{\mathcal{W}\cdot\nabla})^\ell \Big(\big( S_{q_1-1}\mathcal{W} - S_{q-1} \mathcal{W} \big)\cdot \nabla \Delta_{q_1} R_q(\alpha_1,\cdots,\alpha_m)  \Big) \Big\|_{L^p} \\
  & \lesssim \sum_{-1\leq q_1\leq q+N_0} \sum_{\ell_1 + \ell_2 \leq \ell} \|(T_{\mathcal{W}\cdot\nabla})^{\ell_1} (S_{q_1-1} \mathcal{W} - S_{q-1} \mathcal{W})\|_{L^\infty}
  \|(T_{\mathcal{W}\cdot\nabla})^{\ell_2} \nabla \Delta_{q_1} R_q(\alpha_1, \cdots,\alpha_m)\|_{L^p} \\
  & \lesssim \sum_{q_1=-1, \ell_1 + \ell_2\leq \ell} \|(T_{\mathcal{W}\cdot\nabla})^{\ell_1} S_{q-1} \mathcal{W}\|_{L^\infty}
  \|(T_{\mathcal{W}\cdot\nabla})^{\ell_2} \nabla \Delta_{-1} R_q(\alpha_1, \cdots,\alpha_m)\|_{L^p} \\
  &\quad + \sum_{0\leq q_1 \leq q-1} \sum_{\ell_1 + \ell_2 \leq \ell}  \sum_{q_2= q_1}^{q-1} \|(T_{\mathcal{W}\cdot\nabla})^{\ell_1}
  \Delta_{q_2} \mathcal{W}\|_{L^\infty} 2^{q_1} \|(T_{\mathcal{W}\cdot\nabla})^{\ell_2} \varphi_1(2^{-q_1} D) R_q(\alpha_1,\cdots,\alpha_m) \|_{L^p} \\
  & \quad   + \sum_{q+1\leq q_1 \leq q+ N_0}  \sum_{\ell_1 + \ell_2 \leq \ell} \sum_{q_2= q}^{q_1-1} \|(T_{\mathcal{W}\cdot\nabla})^{\ell_1}
  \Delta_{q_2} \mathcal{W}\|_{L^\infty} 2^{q_1} \|(T_{\mathcal{W}\cdot\nabla})^{\ell_2} \varphi_1(2^{-q_1} D) R_q(\alpha_1,\cdots,\alpha_m) \|_{L^p} \\
  & = : R_q^{11} + R_q^{12} + R_q^{13} .
\end{align*}

For $R_q^{11}$, by using the induction assumptions \eqref{eq:Tw-Rq}-\eqref{eq:Tw-es-key}, it follows that
\begin{align}\label{eq:Tw-SqW-1}
  \|(T_{W\cdot\nabla})^{\ell_1} S_{q-1}\mathcal{W}\|_{L^\infty} & \leq \sum_{q_2 =-1}^{q-1} \|(T_{\mathcal{W}\cdot\nabla})^{\ell_1} \Delta_{q_2} \mathcal{W}\|_{L^\infty} \nonumber\\
  & \lesssim \sum_{q_2 =-1}^{q-1} 2^{-q_2(1+\sigma)} \big( 2^{q_2(1+\sigma)} \|(T_{\mathcal{W}\cdot \nabla})^{\ell_1} \Delta_{q_2} \mathcal{W}\|_{L^\infty} \big)
  \lesssim \|\mathcal{W}\|_{\BB^{1+\sigma,\ell_1}_{\infty,\mathcal{W}}} < \infty,
\end{align}
and
\begin{align*}
  & \quad \|(T_{\mathcal{W}\cdot\nabla})^{\ell_2} \nabla \Delta_{-1} R_q(\alpha_1,\cdots,\alpha_m)\|_{L^p}
  \lesssim \sum_{\lambda_2=0}^{\ell_2} \|(T_{\mathcal{W}\cdot\nabla})^{\lambda_2} R_q(\alpha_1,\cdots,\alpha_m)\|_{L^p} \\
  & \lesssim \min_{1\leq i\leq m}\Big(\sum_{\lambda_2=0 }^{\ell_2}\sum_{ |\mu|\leq \lambda_2}
  \|(T_{\mathcal{W}\cdot\nabla})^{\mu_i} \alpha_i\|_{L^p} \prod_{1\leq j\neq i\leq m}\|(T_{\mathcal{W}\cdot\nabla})^{\mu_j}\alpha_j\|_{L^\infty} \Big)\\
  & \lesssim \min_{1\leq i\leq m}\Big(\sum_{|\mu|\leq \ell_2}  \|(T_{\mathcal{W}\cdot\nabla})^{\mu_i} \alpha_i\|_{L^p} \prod_{1\leq j\neq i\leq m}\|(T_{\mathcal{W}\cdot\nabla})^{\mu_j}\alpha_j\|_{L^\infty} \Big),
\end{align*}
thus we get
\begin{align*}
  |R_q^{11}| & \lesssim \min_{1\leq i\leq m}\Big( \sum_{\ell_1+ \ell_2 \leq \ell} \sum_{ |\mu|\leq \ell_2}
  \|(T_{\mathcal{W}\cdot\nabla})^{\mu_i} \alpha_i\|_{L^p} \prod_{1\leq j\neq i\leq m}\|(T_{\mathcal{W}\cdot\nabla})^{\mu_j}\alpha_j\|_{L^\infty} \Big) \\
  & \lesssim \min_{1\leq i\leq m}\Big( \sum_{ |\mu|\leq \ell}
  \|(T_{\mathcal{W}\cdot\nabla})^{\mu_i} \alpha_i\|_{L^p} \prod_{1\leq j\neq i\leq m}\|(T_{\mathcal{W}\cdot\nabla})^{\mu_j}\alpha_j\|_{L^\infty} \Big).
\end{align*}
For $R_q^{12}$, observe that for every $q_1\in\N$,
\begin{align}\label{es:Tw-el-W1}
  2^{q_1} \sum_{q_2= q_1}^{q-1} \|(T_{\mathcal{W}\cdot\nabla})^{\ell_1}  \Delta_{q_2} \mathcal{W}\|_{L^\infty}
  & \lesssim 2^{q_1} \sum_{q_2=q_1}^{q-1} 2^{-q_2(1+\sigma)}
  \Big(\sup_{q_2\in \N} 2^{q_2(1+\sigma)} \|(T_{\mathcal{W}\cdot\nabla})^{\ell_1} \Delta_{q_2} \mathcal{W}\|_{L^\infty} \Big) \nonumber \\
  & \lesssim 2^{-q_1\sigma} \|\mathcal{W}\|_{\BB^{1+\sigma,\ell_1}_{\infty,\mathcal{W}}} < \infty,
\end{align}
and
\begin{align*}
  & \quad \|(T_{\mathcal{W}\cdot\nabla})^{\ell_2} \varphi_1(2^{-q_1} D) R_q(\alpha_1,\cdots,\alpha_m) \|_{L^p}
  \lesssim \sum_{\lambda_2=0}^{\ell_2}
  \|(T_{\mathcal{W}\cdot\nabla})^{\lambda_2}R_q(\alpha_1,\cdots,\alpha_m) \|_{L^p} \\
  & \lesssim \min_{1\leq i\leq m}\Big( \sum_{\lambda_2=0}^{\ell_2} \sum_{|\mu|\leq \lambda_2}
  \|(T_{\mathcal{W}\cdot\nabla})^{\mu_i} \alpha_i\|_{L^p} \prod_{1\leq j\neq i\leq m}\|(T_{\mathcal{W}\cdot\nabla})^{\mu_j}\alpha_j\|_{L^\infty} \Big),
\end{align*}
thus we find
\begin{align*}
  \|R_q^{12}\|_{L^p} & \lesssim \min_{1\leq i\leq m}\bigg( \sum_{0\leq q_1 \leq q-1} \sum_{\ell_1 + \ell_2 \leq \ell} \sum_{|\mu|\leq \ell_2}
  2^{-q_1\sigma}\|\mathcal{W}\|_{\BB^{1+\sigma,\ell_1}_{\infty,\mathcal{W}}} \|(T_{\mathcal{W}\cdot\nabla})^{\mu_i} \alpha_i\|_{L^p} \prod_{1\leq j\neq i\leq m}\|(T_{\mathcal{W}\cdot\nabla})^{\mu_j}\alpha_j\|_{L^\infty} \bigg) \\
  & \lesssim  \min_{1\leq i\leq m}\Big( \sum_{|\mu|\leq \ell} \|(T_{\mathcal{W}\cdot\nabla})^{\mu_i} \alpha_i\|_{L^p} \prod_{1\leq j\neq i\leq m}\|(T_{\mathcal{W}\cdot\nabla})^{\mu_j}\alpha_j\|_{L^\infty} \Big).
\end{align*}
The estimation of $R_q^{13}$ is almost identical to that of $R_q^{12}$,
one also gets
\begin{align*}
  \|R_q^{13}\|_{L^p} \lesssim  \min_{1\leq i\leq m}\Big( \sum_{|\mu|\leq \ell} \|(T_{\mathcal{W}\cdot\nabla})^{\mu_i} \alpha_i\|_{L^p} \prod_{1\leq j\neq i\leq m}\|(T_{\mathcal{W}\cdot\nabla})^{\mu_j}\alpha_j\|_{L^\infty} \Big).
\end{align*}

The term $R_q^2$ can be estimated in the analogous manner, and we infer that
\begin{align*}
  &\quad \|(T_{\mathcal{W}\cdot\nabla})^\ell R_q^2\|_{L^p} \\
  & \lesssim \sum_{-1\leq q_1 \leq q+ N_0} \sum_{i=1}^m \Big\|(T_{\mathcal{W}\cdot\nabla})^\ell \Big(\big( S_{q_1-1}\mathcal{W} - S_{q-1} \mathcal{W} \big)
  \cdot R_q(\alpha_1,\cdots,\nabla \Delta_{q_1}\alpha_i,\cdots \alpha_m)  \Big) \Big\|_{L^p} \\
  & \lesssim \sum_{i=1}^m \sum_{-1\leq q_1\leq q+N_0, \ell_1 + \ell_2 \leq \ell}
  \|(T_{\mathcal{W}\cdot\nabla})^{\ell_1} (S_{q_1-1} \mathcal{W} - S_{q-1} \mathcal{W})\|_{L^\infty}
  \|(T_{\mathcal{W}\cdot\nabla})^{\ell_2} R_q(\alpha_1, \cdots, \nabla \Delta_{q_1} \alpha_i,\cdots \alpha_m)\|_{L^p}  \\
  & \lesssim \sum_{i=1}^m \sum_{q_1=-1, \ell_1 + \ell_2\leq \ell} \|(T_{\mathcal{W}\cdot\nabla})^{\ell_1} S_{q-1} \mathcal{W}\|_{L^\infty}
  \|(T_{\mathcal{W}\cdot\nabla})^{\ell_2}R_q(\alpha_1, \cdots, \nabla \Delta_{-1} \alpha_i,\cdots,\alpha_m)\|_{L^p} \\
  &\quad + \sum_{i=1}^m \sum_{0\leq q_1 \leq q+N_0} \sum_{ \ell_1 + \ell_2 \leq \ell}  \|(T_{\mathcal{W}\cdot\nabla})^{\ell_1}
  (S_{q_1-1}\mathcal{W} - S_{q-1} \mathcal{W})\|_{L^\infty} 2^{q_1} \cdot \\
  & \qquad \cdot\|(T_{\mathcal{W}\cdot\nabla})^{\ell_2} R_q(\alpha_1,\cdots, \varphi_1(2^{-q_1} D) \alpha_i ,\cdots, \alpha_m) \|_{L^p} \\
  & = : R_q^{21} + R_q^{22} .
\end{align*}
For $R_q^{21}$, noting that
\begin{align}\label{es:Tw-el-Rq1}
  & \quad \sum_{i=1}^m \|(T_{\mathcal{W}\cdot\nabla})^{\ell_2} R_q(\alpha_1, \cdots, \nabla \Delta_{-1} \alpha_i,\cdots,\alpha_m)\|_{L^p} \nonumber \\
  & \lesssim  \sum_{i=1}^m \min \bigg\{\sum_{|\mu| \leq \ell_2} \|(T_{\mathcal{W}\cdot\nabla})^{\mu_i}\nabla \Delta_{-1}\alpha_i\|_{L^p} \prod_{j\neq i} \|(T_{\mathcal{W}\cdot\nabla})^{\mu_j} \alpha_j\|_{L^\infty}, \nonumber \\
  & \qquad\;  \min_{1\leq j\neq i \leq m} \Big(\sum_{|\mu| \leq \ell_2} \|(T_{\mathcal{W}\cdot\nabla})^{\mu_j}\alpha_j\|_{L^p}
  \|(T_{\mathcal{W}\cdot\nabla})^{\mu_i} \nabla \Delta_{-1} \alpha_i\|_{L^\infty} \prod_{l\neq i,j} \|(T_{\mathcal{W}\cdot\nabla})^{\mu_l} \alpha_l\|_{L^\infty} \Big)\bigg\} \nonumber \\
  & \lesssim \sum_{i=1}^m \min \bigg\{\sum_{|\mu|\leq \ell_2} \sum_{\lambda_i=0}^{\mu_i} \|(T_{\mathcal{W}\cdot\nabla})^{\lambda_i} \alpha_i\|_{L^p}
  \prod_{j\neq i} \|(T_{\mathcal{W}\cdot\nabla})^{\mu_j} \alpha_j\|_{L^\infty}, \nonumber \\
  & \qquad\quad  \min_{1\leq j\neq i \leq m} \Big( \sum_{|\mu|\leq \ell_2} \sum_{\lambda_i=0}^{\mu_i} \|(T_{\mathcal{W}\cdot\nabla})^{\mu_j}\alpha_j\|_{L^p}\|(T_{\mathcal{W}\cdot\nabla})^{\lambda_i}\alpha_i\|_{L^\infty}
  \prod_{l\neq i,j} \|(T_{\mathcal{W}\cdot\nabla})^{\mu_l} \alpha_l\|_{L^\infty} \Big)\bigg\} \nonumber \\
  & \lesssim \sum_{i=1}^m \min \bigg\{ \sum_{|\tilde\mu|\leq \ell_2}  \|(T_{\mathcal{W}\cdot\nabla})^{\tilde\mu_i} \alpha_i\|_{L^p}
  \prod_{j\neq i} \|(T_{\mathcal{W}\cdot\nabla})^{\tilde\mu_j} \alpha_j\|_{L^\infty}, \nonumber \\
  & \qquad \qquad \qquad\min_{1\leq j\neq i \leq m} \Big( \sum_{|\tilde\mu|\leq \ell_2}  \|(T_{\mathcal{W}\cdot\nabla})^{\tilde\mu_j}\alpha_j\|_{L^p}
  \prod_{l\neq j} \|(T_{\mathcal{W}\cdot\nabla})^{\tilde\mu_l} \alpha_l\|_{L^\infty} \Big)\bigg\} \nonumber \\
  & \lesssim \min_{1\leq j \leq m} \Big( \sum_{|\tilde\mu|\leq \ell_2} \|(T_{\mathcal{W}\cdot\nabla})^{\tilde\mu_j}\alpha_j\|_{L^p}
  \prod_{l\neq j} \|(T_{\mathcal{W}\cdot\nabla})^{\tilde\mu_l} \alpha_l\|_{L^\infty} \Big) ,
\end{align}
where in the above third inequality we have denoted $\tilde\mu=(\tilde\mu_1,\cdots,\tilde\mu_m)\in \N^m$ as $\tilde\mu_j = \mu_j$ for every $j\neq i$ and $\tilde\mu_i = \lambda_i$.
Then by using \eqref{eq:Tw-SqW-1} and the induction assumptions, we obtain
\begin{align*}
  R_q^{21}& \lesssim  \min_{1\leq j\leq m}\Big( \sum_{|\mu|\leq \ell}  \|(T_{\mathcal{W}\cdot\nabla})^{\mu_j} \alpha_j\|_{L^p} \prod_{1\leq l\neq j\leq m}\|(T_{\mathcal{W}\cdot\nabla})^{\mu_l}\alpha_l\|_{L^\infty} \Big) .
\end{align*}

For $R_q^{22}$, similarly as obtaining \eqref{es:Tw-el-Rq1}, we see that for every $0\leq q_1 \leq q+ N_0$,
\begin{align}\label{es:Tw-el-Rq2}
  & \quad \sum_{i=1}^m \|(T_{\mathcal{W}\cdot\nabla})^{\ell_2} R_q(\alpha_1,\cdots, \varphi_1(2^{-q_1} D) \alpha_i ,\cdots, \alpha_m) \|_{L^p} \nonumber \\
  & \lesssim  \sum_{i=1}^m  \min \bigg\{ \sum_{|\mu| \leq \ell_2} \|(T_{\mathcal{W}\cdot\nabla})^{\mu_i}\varphi_1(2^{-q_1}D)\alpha_i\|_{L^p} \prod_{j\neq i} \|(T_{\mathcal{W}\cdot\nabla})^{\mu_j} \alpha_j\|_{L^\infty}, \nonumber \\
  & \quad  \min_{1\leq j\neq i \leq m} \Big( \sum_{|\mu| \leq \ell_2} \|(T_{\mathcal{W}\cdot\nabla})^{\mu_i} \varphi_1(2^{-q_1}D)\alpha_i\|_{L^\infty}\|(T_{\mathcal{W}\cdot\nabla})^{\mu_j}\alpha_j\|_{L^p}
  \prod_{l\neq i,j} \|(T_{\mathcal{W}\cdot\nabla})^{\mu_l} \alpha_l\|_{L^\infty} \Big)\bigg\} \nonumber \\
  & \lesssim \min_{1\leq j \leq m} \Big( \sum_{|\tilde\mu|\leq \ell_2} \|(T_{\mathcal{W}\cdot\nabla})^{\tilde\mu_j}\alpha_j\|_{L^p}
  \prod_{l\neq j} \|(T_{\mathcal{W}\cdot\nabla})^{\tilde\mu_l} \alpha_l\|_{L^\infty} \Big),
\end{align}
thus by using \eqref{es:Tw-el-W1}, we deduce that
\begin{align*}
  R_q^{22}& \lesssim  \min_{1\leq j\leq m}\Big( \sum_{|\mu|\leq \ell} \|(T_{\mathcal{W}\cdot\nabla})^{\mu_j} \alpha_j\|_{L^p} \prod_{1\leq l\neq j\leq m}\|(T_{\mathcal{W}\cdot\nabla})^{\mu_l}\alpha_l\|_{L^\infty} \Big) .
\end{align*}

Next we consider $(T_{\mathcal{W}\cdot\nabla})^\ell R_q^4$, and by applying \eqref{eq:Tw-Rq}, one gets
\begin{align*}
  & \quad \|(T_{\mathcal{W}\cdot\nabla})^\ell R_q^4\|_{L^p} \leq \sum_{i=1}^m \Big\|(T_{\mathcal{W}\cdot\nabla})^\ell R_q(\alpha_1,\cdots, (T_{\mathcal{W}\cdot\nabla})\alpha_i,\cdots,\alpha_m) \Big\|_{L^p} \\
  & \lesssim \sum_{i=1}^m  \min \bigg\{ \sum_{|\mu|\leq \ell}
  \|(T_{\mathcal{W}\cdot\nabla})^{\mu_i+1} \alpha_i\|_{L^p} \prod_{j\neq i} \|(T_{\mathcal{W}\cdot\nabla})^{\mu_j} \alpha_j\|_{L^\infty}, \\
  & \quad \quad \;\min_{1\leq  j\neq i\leq m } \Big( \sum_{|\mu|\leq \ell} \|(T_{\mathcal{W}\cdot\nabla})^{\mu_j} \alpha_j\|_{L^p} \|(T_{\mathcal{W}\cdot\nabla})^{\mu_i+1} \alpha_i\|_{L^\infty} \prod_{1\leq l\leq m,l\neq i,j}\|(T_{\mathcal{W}\cdot\nabla})^{\mu_l} \alpha_l\|_{L^\infty}\Big) \bigg\} \\
  & \lesssim \sum_{i=1}^m \min \bigg\{ \sum_{|\mu|\leq \ell+1}  \|(T_{\mathcal{W}\cdot\nabla})^{\mu_i} \alpha_i\|_{L^p} \prod_{j\neq i} \|(T_{\mathcal{W}\cdot\nabla})^{\mu_j} \alpha_j\|_{L^\infty}, \\
  & \quad \quad \qquad \min_{1\leq  j\neq i\leq m } \Big(\sum_{|\mu|\leq \ell+1} \|(T_{\mathcal{W}\cdot\nabla})^{\mu_j} \alpha_j\|_{L^p}  \prod_{1\leq l\leq m,l\neq j}\|(T_{\mathcal{W}\cdot\nabla})^{\mu_l} \alpha_l\|_{L^\infty} \Big) \bigg\} \\
  & \lesssim \min_{1\leq j\leq m}\Big( \sum_{|\mu|\leq \ell+1}  \|(T_{\mathcal{W}\cdot\nabla})^{\mu_j} \alpha_j\|_{L^p} \prod_{1\leq l\neq j\leq m}\|(T_{\mathcal{W}\cdot\nabla})^{\mu_l}\alpha_l\|_{L^\infty} \Big).
\end{align*}
For the last term $(T_{\mathcal{W}\cdot\nabla})^\ell R_q^5$,
notice that 
for every $0\leq \mu_{m+1}\leq \ell$, 
\begin{align}\label{eq:Tw-nabSqW}
  \|(T_{\mathcal{W}\cdot\nabla})^{\mu_{m+1}} \nabla S_{q_1-1}\mathcal{W}\|_{L^\infty} & \leq
  \sum_{q_2=-1}^{q_1-1} 2^{-q_2 \sigma} \big( 2^{q_2 \sigma} \|(T_{\mathcal{W}\cdot\nabla})^{\mu_{m+1}} \Delta_{q_2} \nabla \mathcal{W}\|_{L^\infty} \big) \nonumber \\
  & \lesssim \|\nabla \mathcal{W}\|_{\BB^{\sigma,\mu_{m+1}}_{\infty,\mathcal{W}}} \lesssim \|\mathcal{W}\|_{\BB^{1+\sigma,\mu_{m+1}}_{\infty,\mathcal{W}}} <\infty,
\end{align}
then by applying the induction assumptions and arguing as \eqref{es:Tw-el-Rq2}, we have
\begin{align*}
  & \quad \|(T_{\mathcal{W}\cdot\nabla})^\ell R_q^5\|_{L^p} \lesssim \sum_{i=1}^m \sum_{q_1=0}^{ q+ N_0}
  \Big\|(T_{\mathcal{W}\cdot\nabla})^\ell R_q^{(i)}(\nabla S_{q_1-1}\mathcal{W},\alpha_1, \cdots, \nabla \Delta_{q_1} \alpha_i,\cdots, \alpha_m) \Big\|_{L^p} \\
  & \lesssim \sum_{i=1}^m \sum_{q_1=0}^{ q+ N_0} \min \bigg\{ \sum_{\mu_1+\cdots + \mu_{m+1}\leq \ell} 2^{-q}
  \|(T_{\mathcal{W}\cdot\nabla})^{\mu_{m+1}} \nabla S_{q_1-1}\mathcal{W}\|_{L^\infty} \cdot \\
  &\qquad \qquad \cdot 2^{q_1}\|(T_{\mathcal{W}\cdot\nabla})^{\mu_i} \varphi_1(2^{-q_1}D) \alpha_i\|_{L^p}
  \prod_{j\neq i} \|(T_{\mathcal{W}\cdot\nabla})^{\mu_j} \alpha_j\|_{L^\infty}, \\
  & \qquad \qquad \quad \min_{1\leq j\neq i\leq m} \Big( \sum_{\mu_1+\cdots + \mu_{m+1}\leq \ell} 2^{-q}
  \|(T_{\mathcal{W}\cdot\nabla})^{\mu_{m+1}} \nabla S_{q_1-1}\mathcal{W}\|_{L^\infty} \cdot \\
  & \qquad \qquad \qquad \qquad \cdot 2^{q_1}\|(T_{\mathcal{W}\cdot\nabla})^{\mu_i} \varphi_1(2^{-q_1}D) \alpha_i\|_{L^\infty}
  \|(T_{\mathcal{W}\cdot\nabla})^{\mu_j} \alpha_j\|_{L^p} \prod_{l\neq i,j}
  \|(T_{\mathcal{W}\cdot\nabla})^{\mu_l} \alpha_l\|_{L^\infty}\Big) \bigg\} \\
  & \lesssim \sum_{i=1}^m \sum_{q_1=0}^{q+N_0} 2^{q_1 -q}
  \min\bigg\{ \sum_{|\mu|\leq \ell} \sum_{\lambda_i=0}^{\mu_i}  \|(T_{\mathcal{W}\cdot\nabla})^{\lambda_i}\alpha_i\|_{L^p} \prod_{ j\neq i}\|(T_{\mathcal{W}\cdot\nabla})^{\mu_j}\alpha_j\|_{L^\infty}, \\
  & \qquad \quad \min_{1\leq j \neq i\leq m} \Big(\sum_{|\mu|\leq \ell} \sum_{\lambda_i=0}^{\mu_i} \|(T_{\mathcal{W}\cdot\nabla})^{\lambda_i} \alpha_i\|_{L^\infty}
  \|(T_{\mathcal{W}\cdot\nabla})^{\mu_j}\alpha_j\|_{L^p} \prod_{l\neq j,i}\|(T_{\mathcal{W}\cdot\nabla})^{\mu_l}\alpha_l\|_{L^\infty}\Big) \bigg\} \\
  & \lesssim \min_{1\leq j\leq m}\Big( \sum_{|\mu|\leq \ell}  \|(T_{\mathcal{W}\cdot\nabla})^{\mu_j} \alpha_j\|_{L^p} \prod_{1\leq l\neq j\leq m}\|(T_{\mathcal{W}\cdot\nabla})^{\mu_l}\alpha_l\|_{L^\infty} \Big) .
\end{align*}
Collecting the above estimates concludes estimate \eqref{eq:Tw-Rq} at the $(\ell+1)$-case.

Now we prove \eqref{eq:Tw-phi-es1}-\eqref{eq:BB-nab-es} at the ($\ell+1$)-case. We see that
\begin{align*}
  \|(T_{\mathcal{W}\cdot\nabla})^{\ell+1} \psi(2^{-q}D) \phi\|_{L^p} \leq \|(T_{\mathcal{W}\cdot\nabla})^\ell\psi(2^{-q} D) (T_{\mathcal{W}\cdot\nabla} \phi)\|_{L^p}
  + \|(T_{\mathcal{W}\cdot\nabla})^\ell[\psi(2^{-q}D), T_{\mathcal{W}\cdot\nabla}] \phi\|_{L^p} .
\end{align*}
The induction assumption guarantees that
\begin{align}\label{eq:Tw-phi-es1-1}
  \|(T_{\mathcal{W}\cdot\nabla})^\ell\psi(2^{-q} D)( T_{\mathcal{W}\cdot\nabla} \phi)\|_{L^p} \lesssim
  \sum_{\lambda=0}^\ell \|(T_{\mathcal{W}\cdot\nabla})^{\lambda+1} \phi\|_{L^p}
  \lesssim \sum_{\lambda=1}^{\ell+1} \|(T_{\mathcal{W}\cdot\nabla})^\lambda \phi\|_{L^p}.
\end{align}
Noting that
\begin{align}\label{eq:comm-exp2}
  & \quad\, [\psi(2^{-q}D),T_{\mathcal{W}\cdot\nabla}] \phi = \sum_{q_1\in \N,q_1\lesssim q} [\psi(2^{-q}D), S_{q_1-1}\mathcal{W} \cdot\nabla] \Delta_{q_1}\phi \nonumber \\
  & = \sum_{q_1\in \N, q_1\lesssim q} \int_{\R^d} \mathcal{F}^{-1}\psi(y) (S_{q_1-1}\mathcal{W}(x + 2^{-q}y) -S_{q_1-1}\mathcal{W}(x)) \cdot \nabla \Delta_{q_1} \phi(x + 2^{-q}y) \dd y \nonumber \\
  & = \sum_{q_1\in \N, q_1\lesssim q } 2^{-q} \int_0^1 \int_{\R^d} \big(\mathcal{F}^{-1}\psi(y)\big) y \cdot\nabla S_{q_1-1}\mathcal{W}(x + \delta 2^{-q}y) \cdot\nabla \Delta_{q_1} \phi(x + 2^{-q}y) \dd y \dd \delta,
\end{align}
where $q_1\lesssim q$ means that $q_1\leq q+N_0$ with some $N_0\in\N$,
we then use \eqref{eq:Tw-Rq} and \eqref{eq:Tw-nabSqW} to get
\begin{align}\label{eq:Tw-phi-es1-2}
  & \quad \|(T_{\mathcal{W}\cdot\nabla})^\ell [\psi(2^{-q}D),T_{\mathcal{W}\cdot\nabla}] \phi\|_{L^p} \nonumber \\
  & \lesssim 2^{-q} \sum_{q_1\in \N,q_1 \lesssim q }\sum_{\mu_1 + \mu_2\leq \ell}
  \|(T_{\mathcal{W}\cdot\nabla})^{\mu_1} \nabla S_{q_1-1}\mathcal{W}\|_{L^\infty} \|(T_{W\cdot\nabla})^{\mu_2} \nabla \Delta_{q_1}\phi\|_{L^p} \nonumber \\
  & \lesssim \sum_{q_1\in \N, q_1\lesssim q} \sum_{\mu_1 + \mu_2\leq \ell} 2^{q_1 -q}
  \|\mathcal{W}\|_{\BB^{1+\sigma,\mu_1}_{\infty,\mathcal{W}}}
  \|(T_{\mathcal{W}\cdot\nabla})^{\mu_2} \varphi_1(2^{-q_1}D) \phi\|_{L^p} \nonumber \\
  & \lesssim \sum_{0\leq\mu_2\leq \ell}  \sum_{\lambda_2 =0}^{\mu_2} \|(T_{\mathcal{W}\cdot\nabla})^{\lambda_2}\phi\|_{L^p}
  \lesssim \sum_{\lambda=0}^\ell  \|(T_{\mathcal{W}\cdot\nabla})^\lambda \phi\|_{L^p}.
\end{align}
Combining \eqref{eq:Tw-phi-es1-1} with \eqref{eq:Tw-phi-es1-2} leads to the desired estimate \eqref{eq:Tw-phi-es1} in the $(\ell+1)$-case.

Concerning \eqref{eq:Tw-es-key}, note that for $q=-1$,
\begin{align}
  & \quad \|(T_{\mathcal{W}\cdot\nabla})^{\ell+1} \Delta_{-1}\phi\|_{L^p} + \|(T_{\mathcal{W}\cdot\nabla})^{\ell+1} \nabla \Delta_{-1} \phi\|_{L^p} \nonumber \\
  & \lesssim \sum_{q_1 \leq N_\ell} \|S_{q_1-1} \mathcal{W}\cdot \nabla \Delta_{q_1} (T_{\mathcal{W}\cdot\nabla})^\ell \Delta_{-1}\phi\|_{L^p}
  + \sum_{q_1\leq N_\ell} \|S_{q_1-1}\mathcal{W}\cdot\nabla \Delta_{q_1} (T_{\mathcal{W}\cdot\nabla})^\ell \nabla \Delta_{-1}\phi\|_{L^p} \nonumber \\
  & \lesssim  \|\mathcal{W}\|_{L^\infty} \big( \|(T_{\mathcal{W}\cdot\nabla})^\ell \Delta_{-1}\phi\|_{L^p} + \|(T_{\mathcal{W}\cdot\nabla})^\ell \nabla \Delta_{-1}\phi\|_{L^p}\big) \lesssim \|\phi\|_{\BB^{s,\ell}_{p,r,\mathcal{W}}},
\end{align}
and for every $q\in \N$,
\begin{align}\label{decom:Tw-es1}
  &\; 2^{q s} \|(T_{\mathcal{W}\cdot\nabla})^{\ell+1} \Delta_q \phi\|_{L^p}
  + 2^{q (s-1)} \|(T_{\mathcal{W}\cdot\nabla})^{\ell+1} \nabla \Delta_q \phi\|_{L^p} \nonumber \\
  \leq &\; 2^{q s}  \|(T_{\mathcal{W}\cdot\nabla})^\ell \Delta_q (T_{\mathcal{W}\cdot\nabla} \phi)\|_{L^p}
  +  2^{q (s-1)} \|(T_{\mathcal{W}\cdot\nabla})^\ell \nabla\Delta_q (T_{\mathcal{W}\cdot\nabla} \phi)\|_{L^p} \nonumber \\
  &\; + 2^{q s} \|(T_{\mathcal{W}\cdot\nabla})^\ell [\Delta_q, T_{\mathcal{W}\cdot\nabla}]\phi\|_{L^p}
  + 2^{q (s-1)}  \|(T_{\mathcal{W}\cdot\nabla})^\ell [\nabla\Delta_q, T_{\mathcal{W}\cdot\nabla}]\phi\|_{L^p}.
\end{align}
The induction assumption implies that
\begin{align*}
  2^{q s} \|(T_{\mathcal{W}\cdot\nabla})^\ell \Delta_q (T_{\mathcal{W}\cdot\nabla} \phi)\|_{L^p}
  + 2^{q (s-1)}  \|(T_{\mathcal{W}\cdot\nabla})^\ell \nabla\Delta_q (T_{\mathcal{W}\cdot\nabla} \phi)\|_{L^p}
  \lesssim c_q  \|(T_{\mathcal{W}\cdot\nabla})\phi\|_{\BB^{s,\ell}_{p,r,\mathcal{W}}}
  \lesssim c_q \|\phi\|_{\BB^{s,\ell+1}_{p,r,\mathcal{W}}},
\end{align*}
where $\{c_q\}_{q\in \N}$ is such that $\|c_q\|_{\ell^r} = 1$.
For every $q\in \N$, according to the following formulas (analogous to \eqref{eq:comm-exp2})
\begin{align}\label{eq:comm-exp2b}
  & \quad [\Delta_q, T_{\mathcal{W}\cdot\nabla}]\phi = \sum_{q_1\in \N,q_1\sim q} [\varphi(2^{-q}D), S_{q_1-1}\mathcal{W}] \cdot\nabla \Delta_{q_1}\phi  \nonumber \\
  & = \sum_{q_1\in \N, q_1\sim q } 2^{-q} \int_0^1 \int_{\R^d} \big(\mathcal{F}^{-1}\varphi(y)\big) y \cdot\nabla S_{q_1-1}\mathcal{W}(x + \delta 2^{-q}y) \cdot\nabla \Delta_{q_1} \phi(x + 2^{-q}y) \dd y \dd \delta,
\end{align}
and
\begin{align}\label{eq:comm-exp3}
  & \quad [\nabla\Delta_q, T_{\mathcal{W}\cdot\nabla}]\phi = 2^q [\varphi_1(2^{-q}D), T_{\mathcal{W}\cdot\nabla}]\phi \nonumber \\
  & = \sum_{q_1\in \N, q_1\sim q } \int_0^1 \int_{\R^d} \big(\mathcal{F}^{-1}\varphi_1(y)\big) y \cdot\nabla S_{q_1-1}\mathcal{W}(x + \delta 2^{-q}y) \cdot\nabla \Delta_{q_1} \phi(x + 2^{-q}y) \dd y \dd \delta,
\end{align}
we can apply \eqref{eq:Tw-Rq} and \eqref{eq:Tw-nabSqW} to get
\begin{align*}
  & 2^{q s} \|(T_{\mathcal{W}\cdot\nabla})^\ell [\Delta_q, T_{\mathcal{W}\cdot\nabla}]\phi\|_{L^p}
  + 2^{q (s-1)} \|(T_{\mathcal{W}\cdot\nabla})^\ell [\nabla\Delta_q, T_{\mathcal{W}\cdot\nabla}]\phi\|_{L^p} \\
  \lesssim & 2^{q (s-1)}
  \sum_{q_1\in \N,q_1 \sim q }\sum_{\mu_1 + \mu_2\leq \ell}
  \|(T_{\mathcal{W}\cdot\nabla})^{\mu_1} \nabla S_{q_1-1}\mathcal{W}\|_{L^\infty}
  \|(T_{\mathcal{W}\cdot\nabla})^{\mu_2} \nabla \Delta_{q_1}\phi\|_{L^p} \\
  \lesssim &
  \sum_{q_1\in \N, q_1\sim q} \sum_{\mu_1 + \mu_2\leq \ell}
  \|\mathcal{W}\|_{\BB^{1+\sigma,\mu_1}_{\infty,\mathcal{W}}}
   2^{q_1(s-1)} \|(T_{\mathcal{W}\cdot\nabla})^{\mu_2}\nabla \Delta_{q_1} \phi\|_{L^p} \\
   \lesssim &  c_q \sum_{\mu_2=0}^\ell \|\phi\|_{\BB^{s,\mu_2}_{p,r,\mathcal{W}}}
   \lesssim  c_q \|\phi\|_{\BB^{s,\ell}_{p,r,\mathcal{W}}}.
\end{align*}
Gathering the above estimates implies \eqref{eq:Tw-es-key} in the $(\ell+1)$-case. 

We then consider \eqref{eq:BB-nab-es}.
In view of the induction assumption and \eqref{eq:Tw-es-key}, \eqref{eq:Tw-nabSqW},
we have 
\begin{align}
  & \quad 2^{qs } \| \Delta_q (T_{\mathcal{W}\cdot\nabla})^{\ell+1} \nabla \phi\|_{L^p} \nonumber \\
  & \lesssim 2^{qs }  \|\Delta_q (T_{\mathcal{W}\cdot\nabla})^\ell \nabla(T_{\mathcal{W}\cdot\nabla} \phi)\|_{L^p}
  + 2^{qs} \sum_{j\in \N,j\sim q} \|\Delta_q (T_{\mathcal{W}\cdot\nabla})^\ell (\nabla S_{j-1} \mathcal{W} \cdot\nabla \Delta_j \phi)\|_{L^p} \nonumber \\
  & \lesssim c_q \|\nabla (T_{\mathcal{W}\cdot\nabla} \phi)\|_{\BB^{s,\ell}_{p,r,\mathcal{W}}} +  2^{qs} \sum_{j\in \N,j\sim q} \sum_{\mu_1+\mu_2\leq \ell}
  \|(T_{\mathcal{W}\cdot\nabla})^{\mu_1}(\nabla S_{j-1}\mathcal{W})\|_{L^\infty}
  \|(T_{\mathcal{W}\cdot\nabla})^{\mu_2} (\nabla \Delta_j\phi)\|_{L^p} \nonumber \\
  & \lesssim c_q \|(T_{\mathcal{W}\cdot\nabla})\phi\|_{\BB^{s+1,\ell}_{p,r,\mathcal{W}}}
  + \sum_{j\in\N,j\sim q}\sum_{\mu_1 +\mu_2\leq \ell}
  2^{js} \|\mathcal{W}\|_{\BB^{1+\sigma,\mu_1}_{\infty,\mathcal{W}}}
  \|(T_{\mathcal{W}\cdot\nabla})^{\mu_2}(\Delta_j\nabla\phi)\|_{L^p} \nonumber \\
  & \lesssim c_q \|\phi\|_{\BB^{s+1,\ell+1}_{p,r,\mathcal{W}}} + c_q \|\mathcal{W}\|_{\BB^{1+\sigma,\ell}_{\infty,\mathcal{W}}}
  \|\nabla \phi\|_{\BB^{s,\ell}_{p,r,\mathcal{W}}}
  \lesssim c_q \|\phi\|_{\BB^{s+1,\ell+1}_{p,r,\mathcal{W}}} + c_q \|\phi\|_{\BB^{s+1,\ell}_{p,r,\mathcal{W}}}
  \lesssim c_q \|\phi\|_{\BB^{s+1,\ell+1}_{p,r,\mathcal{W}}}, \nonumber
\end{align}
with $\{c_q\}_{q\geq -1}$ satisfying $\|c_q\|_{\ell^r}=1$, and it clearly leads to $\|(T_{\mathcal{W}\cdot\nabla})^{\ell+1}\nabla \phi\|_{B^s_{p,r}} \lesssim \|\phi\|_{\BB^{s+1,\ell+1}_{p,r,\mathcal{W}}}$.
Hence, we find
\begin{align}
  \|\nabla \phi\|_{\BB^{s,\ell+1}_{p,r,\mathcal{W}}}
  = \|(T_{\mathcal{W}\cdot\nabla})^{\ell+1}\nabla \phi\|_{B^s_{p,r}} + \|\nabla \phi\|_{\BB^{s,\ell}_{p,r,\mathcal{W}}}
  \lesssim \|\phi\|_{\BB^{s+1,\ell+1}_{p,r}} + \|\phi\|_{\BB^{s+1,\ell}_{p,r,\mathcal{W}}} \lesssim \|\phi\|_{\BB^{s+1,\ell+1}_{p,r,\mathcal{W}}},
\end{align}
which corresponds to \eqref{eq:BB-nab-es} in the $(\ell+1)$-case.

Finally, the induction method completes the proof of \eqref{eq:Tw-Rq}-\eqref{eq:BB-nab-es}.


\subsection{Proof of Lemmas \ref{lem:Tw-es}}\label{subsec:lem5.2}
(1)  We prove estimates \eqref{eq:Tw-es-key2}-\eqref{eq:Tw-es-key3} by the induction method.
It is obvious that \eqref{eq:Tw-es-key2}-\eqref{eq:Tw-es-key3} are correct for $\ell=0$.
Assume they hold for every $\ell'\in \{0,\cdots,\ell\}$ with some $\ell\in \{0,1,\cdots, k-1\}$,
we show that \eqref{eq:Tw-es-key2}-\eqref{eq:Tw-es-key3} are also satisfied for the $(\ell+1)$-case.

Concerning \eqref{eq:Tw-es-key2} in the $(\ell+1)$-case, noticing that (similar to \eqref{exp:Tw-mD0} and \eqref{exp:mD-comm0})
\begin{align}\label{exp:mDnab}
  (T_{\mathcal{W}\cdot\nabla})m(D) \nabla f & = \sum_{q_1\in \N} S_{q_1-1}\mathcal{W}\cdot \nabla m(D)\widetilde{\psi}(2^{-{q_1}} D) \nabla \Delta_{q_1} f \nonumber \\
  & = - \sum_{q_1\in \N}[m(D)\widetilde{\psi}(2^{-{q_1}}D)\nabla, S_{q_1-1}\mathcal{W}\cdot\nabla] \Delta_{q_1} f + m(D)\nabla (T_{\mathcal{W}\cdot\nabla}) f ,
\end{align}
and
\begin{align}\label{exp:mDnab-com}
  & \quad [m(D)\widetilde{\psi}(2^{-{q_1}}D)\nabla, S_{q_1-1}\mathcal{W}\cdot\nabla] \Delta_{q_1} f \nonumber \\
  & = 2^{q_1} \int_{\R^d} \nabla \widetilde{h}(y) \big(S_{q_1-1}\mathcal{W}(x+ 2^{-{q_1}}y) - S_{q_1-1}\mathcal{W}(x) \big)
  \cdot \nabla \Delta_{q_1} f(x + 2^{-{q_1}}y ) \dd y \nonumber \\
  & = \int_0^1 \int_{\R^d} \nabla \widetilde{h}(y) y \cdot \nabla S_{q_1-1}\mathcal{W}(x + \delta 2^{-{q_1}}y) \cdot\nabla \Delta_{q_1} f(x+ 2^{-{q_1}}y) \dd y \dd \delta,
\end{align}
and by using the induction assumption, Lemma \ref{lem:Rq} and \eqref{eq:Tw-nabSqW},
we see that for every $q\geq -1$,
\begin{align*}
  & \quad\; \|\Delta_q(T_{\mathcal{W}\cdot\nabla})^{\ell +1} \nabla m(D) \phi \|_{L^p} \nonumber \\
  &  \leq \|\Delta_q (T_{\mathcal{W}\cdot\nabla})^\ell \nabla m(D) (T_{\mathcal{W}\cdot\nabla}\phi)\|_{L^p}
  + \|\Delta_q (T_{\mathcal{W}\cdot\nabla})^\ell ( [m(D)\nabla, T_{\mathcal{W}\cdot\nabla}]\phi)\|_{L^p} \\
  & \lesssim 2^q \sum_{\lambda=0}^\ell \|(T_{\mathcal{W}\cdot\nabla})^{\lambda+1} \phi\|_{L^p}
  + \sum_{q_1\in \N, q_1\sim q} \sum_{\mu_1+\mu_2\leq \ell}  \|(T_{\mathcal{W}\cdot\nabla})^{\mu_1} S_{q_1-1}\nabla \mathcal{W}\|_{L^\infty}
  \|(T_{\mathcal{W}\cdot\nabla})^{\mu_2} \nabla \Delta_{q_1} \phi\|_{L^p} \\
  & \lesssim \sum_{\lambda=1}^{\ell+1} 2^q  \|(T_{\mathcal{W}\cdot\nabla})^\lambda \phi\|_{L^p}  +
  \sum_{\lambda=0}^\ell 2^q  \|(T_{\mathcal{W}\cdot\nabla})^\lambda \phi\|_{L^p}
  \lesssim \sum_{\lambda=0}^{\ell+1} 2^q   \|(T_{\mathcal{W}\cdot\nabla})^\lambda \phi\|_{L^p},
\end{align*}
as desired. Hence the induction method guarantees \eqref{eq:Tw-es-key2}.

For \eqref{eq:Tw-es-key3} in the $(\ell+1)$-case, the induction assumption ensures that for every $q\in \N$,
\begin{align*}
  &\quad 2^q \|(T_{\mathcal{W}\cdot\nabla})^{\ell+1} \Delta_q \phi\|_{L^p} \leq 2^q \|(T_{\mathcal{W}\cdot\nabla})^\ell \Delta_q (T_{\mathcal{W}\cdot\nabla}\phi)\|_{L^p}
  + 2^q \|(T_{\mathcal{W}\cdot\nabla})^\ell [\Delta_q, T_{\mathcal{W}\cdot\nabla}]\phi\|_{L^p} \\
  & \lesssim \sum_{q_1\in \N, |q_1-q|\leq N_\ell} \sum_{\lambda=0}^\ell  \|(T_{\mathcal{W}\cdot\nabla})^\lambda \Delta_{q_1} (T_{\nabla \mathcal{W}\cdot\nabla}\phi)\|_{L^p}
  + \sum_{q_1\in \N, |q_1-q|\leq N_\ell} \sum_{\lambda=0}^\ell  \|(T_{\mathcal{W}\cdot\nabla})^{\lambda+1} \Delta_{q_1} \nabla \phi \|_{L^p} \\
  & \quad \; + \sum_{q_1\in \N, |q_1-q|\leq N_\ell} \sum_{\lambda=0}^\ell  \|(T_{\mathcal{W}\cdot\nabla})^\lambda ([\Delta_{q_1}, T_{\mathcal{W}\cdot\nabla}]\nabla\phi)\|_{L^p}
  +  2^q \|(T_{\mathcal{W}\cdot\nabla})^\ell [\Delta_q, T_{\mathcal{W}\cdot\nabla}]\phi\|_{L^p} \\
  & = : \mathcal{I}_q^1 + \mathcal{I}_q^2 + \mathcal{I}_q^3 + \mathcal{I}_q^4.
\end{align*}
By virtue of \eqref{eq:Tw-Rq}-\eqref{eq:Tw-phi-es1}, we treat $\mathcal{I}_q^1$ as follows
\begin{align*}
  \mathcal{I}_q^1 & \lesssim \sum_{q_1\in \N,|q_1-q|\leq N_\ell} \sum_{q_2\in \N,|q_2-q_1|\leq 4}\sum_{\lambda=0}^\ell
  \|(T_{\mathcal{W}\cdot\nabla})^\lambda \Delta_{q_1} (S_{q_2-1}\nabla \mathcal{W}\cdot\nabla \Delta_{q_2} \phi)\|_{L^p} \\
  & \lesssim \sum_{q_2\in \N, |q_2-q|\leq N_\ell+4} \sum_{\lambda =0}^\ell
  \|(T_{\mathcal{W}\cdot\nabla})^\lambda (S_{q_2-1}\nabla \mathcal{W} \cdot\nabla \Delta_{q_2}\phi)\|_{L^p} \\
  & \lesssim \sum_{q_2\in \N, |q_2-q|\leq N_\ell + 4} \sum_{\lambda=0}^\ell \sum_{\lambda_1+\lambda_2\leq \lambda}
  \|(T_{\mathcal{W}\cdot\nabla})^{\lambda_1} S_{q_2-1}\nabla \mathcal{W}\|_{L^\infty}
  \|(T_{\mathcal{W}\cdot\nabla})^{\lambda_2} \nabla \Delta_{q_2} \phi\|_{L^p} \\
  & \lesssim \sum_{q_2 \in \N, |q_2-q|\leq N_\ell+4} \sum_{\lambda=0}^\ell
  \|(T_{\mathcal{W}\cdot\nabla})^\lambda \Delta_{q_2} \nabla \phi\|_{L^p} .
\end{align*}
The estimation of $\mathcal{I}_q^2$ and $\mathcal{I}_q^4$ is relatively easy:
\begin{align*}
  \mathcal{I}_q^2 \lesssim \sum_{q_1\in \N, |q_1-q|\leq N_\ell} \sum_{\lambda=1}^{\ell+1} \|(T_{\mathcal{W}\cdot\nabla})^\lambda \Delta_{q_1} \nabla \phi \|_{L^p},
\end{align*}
\begin{align*}
   \mathcal{I}_q^4 & \lesssim \sum_{q_1\in\N, |q_1-q|\leq 4} \sum_{\mu_1+\mu_2\leq \ell} \|(T_{\mathcal{W}\cdot\nabla})^{\mu_1}\nabla S_{q_1-1} \mathcal{W}\|_{L^\infty}
   \|(T_{\mathcal{W}\cdot\nabla})^{\mu_2}\nabla\Delta_{q_1}\phi\|_{L^p} \\
   & \lesssim \sum_{q_1\in\N, |q_1-q|\leq 4} \sum_{\mu_2 =0}^\ell \|(T_{\mathcal{W}\cdot\nabla})^{\mu_2} \Delta_{q_1} \nabla\phi\|_{L^p}.
\end{align*}
For $\mathcal{I}_q^3$, owing to \eqref{eq:Tw-phi-es1}, \eqref{eq:Tw-nabSqW}, \eqref{eq:comm-exp2b} and using the fact that $\Delta_{q_2} = \Delta_{q_2} \widetilde{\Delta}_{q_2}$ for $q_2\in \N$, we get
\begin{align*}
  \mathcal{I}_q^3 & \lesssim \sum_{q_2\in \N, |q_2-q|\leq N_\ell+4} \sum_{\lambda=0}^\ell \sum_{\lambda_1+\lambda_2\leq \lambda} 2^{-q_2}
  \|(T_{\mathcal{W}\cdot\nabla})^{\lambda_1} S_{q_2-1}\nabla \mathcal{W}\|_{L^\infty} \|(T_{\mathcal{W}\cdot\nabla})^{\lambda_2}\nabla \Delta_{q_2} \nabla\phi\|_{L^p} \\
  & \lesssim \sum_{q_2\in \N, |q_2-q|\leq N_\ell+4} \sum_{\lambda=0}^\ell  2^{-q_2}
  \|(T_{\mathcal{W}\cdot\nabla})^\lambda \nabla \Delta_{q_2} \nabla\phi\|_{L^p} \\
  & \lesssim \sum_{q_2\in \N, |q_2-q|\leq N_\ell+4} \sum_{\lambda=0}^\ell
  \|(T_{\mathcal{W}\cdot\nabla})^\lambda \varphi_1(2^{-q_2} D) \widetilde\Delta_{q_2} \nabla\phi\|_{L^p} \\
  & \lesssim \sum_{q_2\in \N, |q_2-q|\leq N_\ell+4} \sum_{\lambda=0}^\ell
  \|(T_{\mathcal{W}\cdot\nabla})^\lambda \widetilde\Delta_{q_2} \nabla\phi\|_{L^p}
  \lesssim \sum_{q_3\in \N, |q_3-q|\leq N_\ell+5} \sum_{\lambda=0}^\ell
  \|(T_{\mathcal{W}\cdot\nabla})^\lambda \Delta_{q_3} \nabla\phi\|_{L^p}.
\end{align*}
Collecting the above estimates yields \eqref{eq:Tw-es-key3} in the $(\ell+1)$-case and thus finishes the proof of \eqref{eq:Tw-es-key3}.

(2) We first consider \eqref{es:Tw-parap-1} with $s<0$. In light of the Fourier support property, for every $\lambda\in \N$, there exists a positive integer $N_\lambda\in \N$
such that
\begin{align*}
  \Delta_q (T_{\mathcal{W}\cdot\nabla})^\lambda T_v w = \sum_{q_1\in \N,|q_1-q|\leq N_\lambda} \Delta_q (T_{\mathcal{W}\cdot\nabla})^\lambda (S_{q_1-1} v \Delta_{q_1} w) .
\end{align*}
By virtue of \eqref{eq:Tw-Rq}, and using the following estimates that for $q_1\in \N$,
\begin{align}\label{eq:Tw-fact}
  \|(T_{\mathcal{W}\cdot\nabla})^{\mu_1} S_{q_1-1}v\|_{L^\infty} \leq \sum_{q'\leq q_1-1} \|(T_{\mathcal{W}\cdot\nabla})^{\mu_1}\Delta_{q'}v\|_{L^\infty}
  \lesssim \|v\|_{\BB^{0,\mu_1}_\mathcal{W}},
\end{align}
\begin{align*}
  \|(T_{\mathcal{W}\cdot\nabla})^{\mu_1} S_{q_1-1} v\|_{L^p}
  & \leq \sum_{q'\leq q_1-1} \|(T_{\mathcal{W}\cdot\nabla})^{\mu_1}\Delta_{q'}v\|_{L^p} \\
  & \lesssim  \sum_{q'\leq q_1-1} 2^{-q's} \big( 2^{q's}\|(T_{\mathcal{W}\cdot\nabla})^{\mu_1}\Delta_{q'}v\|_{L^p} \big)
  \lesssim 2^{-q_1s}\|v\|_{\BB^{s,\mu_1}_{p,\infty,\mathcal{W}}}, 
\end{align*}
 we get that for every $\lambda \in \{0,1,\cdots, \ell\}$,
\begin{align*}
  & \quad\; 2^{qs} \|\Delta_q (T_{\mathcal{W}\cdot\nabla})^\lambda T_v w\|_{L^p} \\
  & \lesssim
  \sum_{|q_1-q|\leq N_\lambda} \min \bigg\{ \sum_{\mu_1+ \mu_2\leq\lambda} 2^{q_1 s} \|(T_{\mathcal{W}\cdot\nabla})^{\mu_1} S_{q_1-1} v\|_{L^\infty} \|(T_{\mathcal{W}\cdot\nabla})^{\mu_2} \Delta_{q_1} w\|_{L^p},  \\
  & \qquad \qquad\qquad \qquad \sum_{\mu_1+ \mu_2\leq\lambda} 2^{q_1 s}
  \|(T_{\mathcal{W}\cdot\nabla})^{\mu_1} S_{q_1-1} v\|_{L^p} \|(T_{\mathcal{W}\cdot\nabla})^{\mu_2} \Delta_{q_1} w\|_{L^\infty} \bigg \} \\
  & \lesssim c_q \min\bigg\{ \sum_{\mu_1=0}^\ell \sum_{\mu_2=0}^{\ell - \mu_1}  \|v\|_{\BB^{0,\mu_1}_\mathcal{W}}  \Big(\sum_{\lambda_2=0}^{\mu_2} \|(T_{\mathcal{W}\cdot\nabla})^{\lambda_2} w\|_{B^s_{p,r}}\Big),
  \sum_{\mu_1=0}^\ell \sum_{\mu_2=0}^{\ell - \mu_1}  \|v\|_{\BB^{s,\mu_1}_{p,r,\mathcal{W}}}
  \Big(\sum_{\lambda_2=0}^{\mu_2} \|(T_{\mathcal{W}\cdot\nabla})^{\lambda_2} w\|_{B^0_{\infty,1}} \Big)\bigg\} \\
  & \lesssim c_q  \min\Big\{ \sum_{\mu_1=0}^\ell \|v\|_{\BB^{0,\mu_1}_\mathcal{W}} \|w\|_{\BB^{s,\ell -\mu_1}_{p,r,\mathcal{W}}} ,
  \sum_{\mu_1=0}^\ell \|v\|_{\BB^{s,\mu_1}_{p,r,\mathcal{W}}} \|w\|_{\BB^{0,\ell-\mu_1}_\mathcal{W}} \Big\},
\end{align*}
where $\{c_q\}_{q\geq -1}$ is such that $\|c_q\|_{\ell^r} = 1$. By taking the $\ell^r$-norm and summing $\lambda$ from $0$ to $\ell$, we finish the proof of \eqref{es:Tw-parap-1} with $s<0$.

The estimation of \eqref{es:Tw-parap-1b} is quite similar, and from \eqref{eq:Tw-Rq} and \eqref{eq:Tw-fact}, we easily see that
\begin{align*}
  \|T_v w\|_{\BB^{0,\ell}_{p,r,\mathcal{W}}}
  \lesssim \sum_{\mu_1=0}^\ell \sum_{\mu_2=0}^{\ell -\mu_1}
  \|v\|_{\BB^{0,\mu_1}_\mathcal{W}}
  \Big(\sum_{\lambda_2=0}^{\mu_2} \|(T_{\mathcal{W}\cdot\nabla})^{\lambda_2} w\|_{B^0_{p,r}}\Big)
  \lesssim \sum_{\mu_1=0}^\ell \|v\|_{\BB^{0,\mu_1}_\mathcal{W}} \|w\|_{\BB^{0,\ell -\mu_1}_{p,r,\mathcal{W}}} .
\end{align*}

We then consider \eqref{es:Tw-parap-2} for every $s <1$. Note that
\begin{align*}
  \Delta_q (T_{\mathcal{W}\cdot\nabla})^\lambda T_{\nabla w}v = \sum_{q_1\in \N,|q_1-q|\leq N_\lambda} \Delta_q (T_{\mathcal{W}\cdot\nabla})^\lambda (S_{q_1-1} \nabla w\, \Delta_{q_1} v),
\end{align*}
with $N_\lambda\in \N$ an integer, and using the following estimate (thanks to \eqref{eq:Tw-es-key2} and \eqref{eq:Tw-es-key})
\begin{align*}
  \|(T_{\mathcal{W}\cdot\nabla})^{\mu_1} S_{q_1-1} \nabla w\|_{L^\infty} & \leq \sum_{q'\leq q_1-1} \|(T_{\mathcal{W}\cdot\nabla})^{\mu_1} \Delta_{q'} \nabla w\|_{L^\infty} \\
  & \lesssim \sum_{q'\leq q_1-1} \sum_{\lambda_1=0}^{\mu_1} 2^{q'}  \|(T_{\mathcal{W}\cdot\nabla})^{\lambda_1} \Delta_{q'}w\|_{L^\infty} \\
  & \lesssim 2^{q_1(1-s)} \sum_{\lambda_1=0}^{\mu_1} \|(T_{\mathcal{W}\cdot\nabla})^{\lambda_1} w\|_{B^s_{\infty,\infty}} \lesssim 2^{q_1(1-s)} \|w\|_{\BB^{s,\mu_1}_{\infty,\mathcal{W}}},
\end{align*}
we deduce that for every $\lambda \in \{0,\cdots, \ell\}$,
\begin{align*}
  &\quad 2^{qs} \|\Delta_q (T_{\mathcal{W}\cdot\nabla})^\lambda T_{\nabla w} v\|_{L^p} \\
  & \lesssim 2^{qs} \sum_{|q_1-q|\leq N_\lambda} \sum_{\mu_1+ \mu_2\leq\lambda}
  \|(T_{\mathcal{W}\cdot\nabla})^{\mu_1} S_{q_1-1} \nabla w\|_{L^\infty} \|(T_{\mathcal{W}\cdot\nabla})^{\mu_2} \Delta_{q_1} v\|_{L^p} \\
  & \lesssim \sum_{\mu_1+ \mu_2\leq \ell} \sum_{q_1\in \N, |q_1-q|\leq N_\ell} 2^{q_1 } \|(T_{\mathcal{W}\cdot\nabla})^{\mu_2} \Delta_{q_1} v\|_{L^p}
  \|w\|_{\BB^{s,\mu_1}_{\infty,\mathcal{W}}} \\
  & \lesssim c_q \sum_{\mu_1=0}^\ell \sum_{\mu_2=0}^{\ell-\mu_1}\Big(\sum_{\lambda_2=0}^{\mu_2}\|(T_{\mathcal{W}\cdot\nabla})^{\lambda_2} v\|_{B^1_{p,r}}\Big)
   \|w\|_{\BB^{s,\mu_1}_{\infty,\mathcal{W}}}
  \lesssim c_q \sum_{\mu_1=0}^\ell \|v\|_{\BB^{1,\ell -\mu_1}_{p,r,\mathcal{W}}} \|w\|_{\BB^{s,\mu_1}_{\infty,\mathcal{W}}},
\end{align*}
where $\{c_q\}_{q\geq -1}$ is such that $\|c_q\|_{\ell^r} = 1$. Then it directly implies the desired estimate \eqref{es:Tw-parap-2}.

For the estimation of \eqref{es:Tw-parap-2c}, noticing that
\begin{align}\label{eq:fact4}
  \|(T_{\mathcal{W}\cdot\nabla})^{\mu_1} S_{q_1-1}\nabla w\|_{L^\infty}
  \leq \sum_{q'\leq q_1-1} \|(T_{\mathcal{W}\cdot\nabla})^{\mu_1} \Delta_{q'}\nabla w\|_{L^\infty}
  \lesssim \|\nabla w\|_{\BB^{0,\mu_1}_\mathcal{W}},
\end{align}
we obtain that for every $s\in \R$ and $\lambda\in \{0,1,\cdots,\ell\}$,
\begin{align}
  &\quad 2^{qs} \|\Delta_q (T_{\mathcal{W}\cdot\nabla})^\lambda T_{\nabla w} v\|_{L^p} \nonumber \\
  & \lesssim \sum_{\mu_1+ \mu_2\leq \mu} \sum_{q_1\in \N, q_1\sim q} 2^{q_1 s} \|(T_{\mathcal{W}\cdot\nabla})^{\mu_2 }\Delta_{q_1} v\|_{L^p}
  \|(T_{\mathcal{W}\cdot\nabla})^{\mu_1} S_{q_1-1}\nabla w\|_{L^\infty} \nonumber \\
  & \lesssim c_q \sum_{\mu_1=0}^\ell \sum_{\mu_2=0}^{\ell-\mu_1} \Big(\sum_{\lambda_2=0}^{\mu_2} \|(T_{\mathcal{W}\cdot\nabla})^{\lambda_2} v\|_{B^s_{p,r}} \Big)
  \|\nabla w\|_{\BB^{0,\mu_1}_\mathcal{W}}
  \lesssim c_q \sum_{\mu_1 =0}^\ell \|\nabla w\|_{\BB^{0,\mu_1}_\mathcal{W}} \|v\|_{\BB^{s,\ell-\mu_1}_{p,r,\mathcal{W}}}, \nonumber
\end{align}
then taking the $\ell^r$-norm on $\{q\geq -1\}$ and summing over $\lambda$ lead to the estimate \eqref{es:Tw-parap-2c}.

(3) For every $\lambda \in \N$, there exists an integer $N_\lambda'>0$ such that
\begin{align}\label{decom:Tw-el-Rem}
  & \quad\; \Delta_q (T_{\mathcal{W}\cdot\nabla})^\lambda R(v\cdot,\nabla w) = \sum_{q_1\geq \max\{ q- N_\lambda',-1\}}
  \Delta_q (T_{\mathcal{W}\cdot\nabla})^\lambda (\Delta_{q_1} v\cdot \nabla \widetilde\Delta_{q_1} w ) \nonumber \\
  & = \sum_{q_1\geq \max\{q-N_\lambda' ,3\}} \Delta_q (T_{\mathcal{W}\cdot\nabla})^\lambda \nabla \cdot(\Delta_{q_1}v\, \widetilde\Delta_{q_1} w)
  + \sum_{q_1\leq 3} 1_{\{q\leq N_\lambda' +3\}} \Delta_q (T_{\mathcal{W}\cdot\nabla})^\lambda (\Delta_{q_1}v\,\cdot\nabla \widetilde\Delta_{q_1} w) \nonumber \\
  & = : \Upsilon_{q,\lambda}^1 + \Upsilon_{q,\lambda}^2 .
\end{align}
By applying \eqref{eq:Tw-es-key2} and Lemma \ref{lem:Rq}, we find that for every $\lambda \in \{0,1,\cdots, \ell\}$,
\begin{align*}
  2^{qs} \|\Upsilon_{q,\lambda}^1\|_{L^p}
  & \leq  2^{qs}  \sum_{q_1\geq \max\{q-N_\lambda', 3\}}
  \|\Delta_q (T_{\mathcal{W}\cdot\nabla})^\lambda \nabla \cdot(\Delta_{q_1}v\, \widetilde\Delta_{q_1} w)\|_{L^p} \\
  & \lesssim \sum_{q_1\geq \max\{q-N_\lambda',3\}}\sum_{\mu=0}^\lambda 2^{q(1+s)}
  \|(T_{\mathcal{W}\cdot\nabla})^\mu (\Delta_{q_1}v\, \widetilde\Delta_{q_1}w) \|_{L^p} \\
  & \lesssim \sum_{q_1\geq \max\{q-N_\lambda',3\}} 2^{(q-q_1)(1+s)}\min \bigg\{ \sum_{\mu_1+\mu_2\leq \ell}
  2^{q_1(1+s)}  \|(T_{\mathcal{W}\cdot\nabla})^{\mu_1}\Delta_{q_1}v\|_{L^p}
  \|(T_{\mathcal{W}\cdot\nabla})^{\mu_2} \widetilde\Delta_{q_1}w\|_{L^\infty}, \\
  & \qquad \qquad \qquad \qquad \quad \sum_{\mu_1+\mu_2\leq \ell}
  2^{q_1(1+s)}  \|(T_{\mathcal{W}\cdot\nabla})^{\mu_1}\Delta_{q_1}v\|_{L^\infty}
  \|(T_{\mathcal{W}\cdot\nabla})^{\mu_2} \widetilde\Delta_{q_1}w\|_{L^p}\bigg\} \\
  & \lesssim c_q \min \bigg\{ \sum_{\mu_1+\mu_2\leq \ell} \Big\|
  2^{q_1(1+s)} \|(T_{\mathcal{W}\cdot\nabla})^{\mu_1}\Delta_{q_1}v\|_{L^p}
  \|(T_{\mathcal{W}\cdot\nabla})^{\mu_2} \widetilde\Delta_{q_1}w\|_{L^\infty}\Big\|_{\ell^r(\{q_1\geq 3\})}, \\
  & \qquad\qquad \quad  \sum_{\mu_1+\mu_2\leq \ell} \Big\|  2^{q_1(1+s)} \|(T_{\mathcal{W}\cdot\nabla})^{\mu_1}\Delta_{q_1}v\|_{L^\infty}
  \|(T_{\mathcal{W}\cdot\nabla})^{\mu_2} \widetilde\Delta_{q_1}w\|_{L^p}\Big\|_{\ell^r(\{q_1\geq 3\})} \bigg\},
\end{align*}
where $\{c_q\}_{q\geq -1}$ is such that $\|c_q\|_{\ell^r}= 1$. Thanks to estimates \eqref{eq:Tw-es-key} and \eqref{eq:Tw-es-key3}, we deduce that
\begin{align*}
  & \quad\; \big\| 2^{q_1} \|(T_{\mathcal{W}\cdot\nabla})^{\mu_2} \widetilde\Delta_{q_1}w\|_{L^\infty} \big\|_{\ell^\infty(\{q\geq 3\})} \nonumber \\
  & \lesssim  \big\| 2^{q_1} \|(T_{\mathcal{W}\cdot\nabla})^{\mu_2} \Delta_{q_1}w\|_{L^\infty}
  \big\|_{\ell^\infty(\{q_1\geq 2\})} \\
  & \lesssim  \Big\| \sum_{q_2\in \N,q_2\sim q_1} \sum_{\lambda_2=0}^{\mu_2}
  \|(T_{\mathcal{W}\cdot\nabla})^{\lambda_2} \Delta_{q_2} \nabla w\|_{L^\infty}\Big\|_{\ell^\infty(\{q_1\geq 2 \})} \\
  & \lesssim \sum_{\lambda_2=0}^{\mu_2} \big\| \|(T_{\mathcal{W}\cdot\nabla})^{\lambda_2} \Delta_{q_2} \nabla w\|_{L^\infty}\big\|_{\ell^\infty(\{q_2\in \N\})}
  \lesssim \sum_{\lambda_2=0}^{\mu_2} \|(T_{\mathcal{W}\cdot\nabla})^{\lambda_2} \nabla w\|_{B^{0,\ell- \mu_1-\lambda_2}_{\infty,\infty}} ,
\end{align*}
and
\begin{align*}
  \big\|2^{q_1(1+s)}  \|(T_{\mathcal{W}\cdot\nabla})^{\mu_2} \widetilde\Delta_{q_1}w\|_{L^p}\big\|_{\ell^r(\{q_1\geq 3\})}
  & \lesssim  \sum_{\lambda_2=0}^{\mu_2} \big\| 2^{q_2 s} \|(T_{\mathcal{W}\cdot\nabla})^{\lambda_2} \Delta_{q_2} \nabla w\|_{L^p} \big\|_{\ell^r(\{q_2\in \N\})} \\
  & \lesssim \sum_{\lambda_2=0}^{\mu_2} \|(T_{\mathcal{W}\cdot\nabla})^{\lambda_2} \nabla w\|_{B^s_{p,r}},
\end{align*}
and
\begin{align*}
  \big\| 2^{q_1} \|(T_{\mathcal{W}\cdot\nabla})^{\mu_1}\Delta_{q_1}v\|_{L^\infty} \big\|_{\ell^\infty(\{q_1\geq 3\})} \lesssim
  \|v\|_{\BB^{1,\mu_1}_{\infty,\mathcal{W}}},
\end{align*}
which immediately leads to that
\begin{align}\label{es:Ups-1}
  & \quad 2^{qs} \|\Upsilon_{q,\lambda}^1\|_{L^p} \nonumber \\
  & \lesssim c_q \min\bigg\{ \sum_{\mu_1=0}^\ell \sum_{\mu_2=0}^{\ell-\mu_1} \|v\|_{\BB^{0,\mu_1}_{\infty,\mathcal{W}}}
  \Big(\sum_{\lambda_2=0}^{\mu_2} \|(T_{\mathcal{W}\cdot\nabla})^{\lambda_2} \nabla w\|_{B^s_{p,r}}\Big), \nonumber \\
  &\qquad\qquad\quad \sum_{\mu_1=0}^\ell \sum_{\mu_2=0}^{\ell-\mu_1}  \|v\|_{\BB^{s,\mu_1}_{p,r,\mathcal{W}}}
  \Big(\sum_{\lambda_2=0}^{\mu_2} \|(T_{\mathcal{W}\cdot\nabla})^{\lambda_2} \nabla w\|_{B^0_{\infty,\infty}} \Big), \nonumber \\
  &\qquad\qquad\quad \sum_{\mu_1=0}^\ell \sum_{\mu_2=0}^{\ell-\mu_1} \|v\|_{\BB^{1,\mu_1}_{\infty,\mathcal{W}}}
  \Big(\sum_{\lambda_2=0}^{\mu_2} \|(T_{\mathcal{W}\cdot\nabla})^{\lambda_2} w\|_{B^s_{p,r}} \Big)\bigg\} \nonumber \\
  & \lesssim c_q\min\bigg\{ \sum_{\mu_1=0}^\ell \|v\|_{\BB^{0,\mu_1}_{\infty,\mathcal{W}}} \|\nabla w\|_{\BB^{s,\ell -\mu_1}_{p,r,\mathcal{W}}} ,
  \sum_{\mu_1=0}^\ell \|v\|_{\BB^{s,\mu_1}_{p,r,\mathcal{W}}} \|\nabla w\|_{\BB^{0,\ell-\mu_1}_{\infty,\mathcal{W}}},
  \sum_{\mu_1=0}^\ell \|v\|_{\BB^{1,\mu_1}_{\infty,\mathcal{W}}} \|w\|_{\BB^{s,\ell-\mu_1}_{p,r,\mathcal{W}}} \bigg\}.
\end{align}

On the other hand, we argue as \eqref{es:BB-LP} to infer that for every $\lambda \in \{0,1,\cdots,\ell\}$,
\begin{align}\label{es:Ups-2}
  & \quad \big\| 2^{qs} \|\Upsilon_{q,\lambda}^2\|_{L^p}\big\|_{\ell^r(q\geq -1)}
  \lesssim \sum_{-1\leq q_1\leq 3} \|(T_{\mathcal{W}\cdot\nabla})^\lambda (\Delta_{q_1}v\,\cdot\nabla \widetilde\Delta_{q_1} w)\|_{L^p} \nonumber \\
  & \lesssim \sum_{-1\leq q_1\leq 3}  \min\Big\{ \|\Delta_{q_1}v\|_{L^p}
  \|\widetilde\Delta_{q_1}\nabla w\|_{L^\infty}, \|\Delta_{q_1}v\|_{L^\infty} \|\widetilde\Delta_{q_1}\nabla w\|_{L^p} \Big\} \nonumber \\
  & \lesssim \min\Big\{ \|v\|_{B^0_{\infty,\infty}} \|\nabla w\|_{B^s_{p,r}} , \|v\|_{B^s_{p,r}}
  \|\nabla w\|_{B^0_{\infty,\infty}}, \|v\|_{B^0_{\infty,\infty}} \|w \|_{B^s_{p,r}}  \Big\} .
\end{align}
Hence, collecting the above estimates \eqref{decom:Tw-el-Rem}--\eqref{es:Ups-2} yields the desired inequality \eqref{es:Tw-rem-1}.

(4) We prove \eqref{normBB-equiv} by the induction method. Suppose that \eqref{normBB-equiv} holds for $\ell\in \{0,1,\cdots,k-1\}$, we next prove that it holds for the $(\ell+1)$-case.
Bony's decomposition gives
\begin{align}\label{decom:Tw-el-phi}
  (T_{\mathcal{W}\cdot\nabla})^{\ell+1} \phi &  = (T_{\mathcal{W}\cdot\nabla})^\ell (T_{\mathcal{W}\cdot\nabla} - \partial_\mathcal{W})\phi + (T_{\mathcal{W}\cdot\nabla})^\ell \partial_\mathcal{W} \phi \nonumber\\
  & = - (T_{\mathcal{W}\cdot\nabla})^\ell (T_{\nabla \phi}\cdot \mathcal{W}) - (T_{\mathcal{W}\cdot\nabla})^\ell R(\mathcal{W}\cdot,\nabla \phi) + (T_{\mathcal{W}\cdot\nabla})^\ell \partial_\mathcal{W} \phi.
\end{align}
According to \eqref{es:Tw-parap-2}, \eqref{es:Tw-rem-1} and the induction assumption, we get
\begin{align}\label{es:Tw-para-key}
  \|(T_{\mathcal{W}\cdot\nabla})^\ell (T_{\nabla \phi}\cdot \mathcal{W})\|_{B^s_{p,r}}
  \lesssim \|T_{\nabla \phi}\cdot \mathcal{W}\|_{\BB^{s,\ell}_{p,r,\mathcal{W}}}  \lesssim
  \| \phi\|_{\BB^{s,\ell}_{p,r,\mathcal{W}}} \|\mathcal{W}\|_{\BB^{1,\ell}_\mathcal{W}}
  \lesssim \|\phi\|_{\BB^{s,\ell}_{p,r,\mathcal{W}}} \lesssim \|\phi\|_{\bb^{s,\ell}_{p,r,\mathcal{W}}},
\end{align}
and
\begin{align*}
  \|(T_{\mathcal{W}\cdot\nabla})^\ell R(\mathcal{W}\cdot,\nabla \phi)\|_{B^s_{p,r}}
  \lesssim \|R(\mathcal{W}\cdot,\nabla \phi)\|_{\BB^{s,\ell}_{p,r,\mathcal{W}}}
  \lesssim \|\mathcal{W}\|_{\BB^{1,\ell}_\mathcal{W}} \|\phi\|_{\BB^{s,\ell}_{p,r,\mathcal{W}}}
  \lesssim \|\phi\|_{\BB^{s,\ell}_{p,r,\mathcal{W}}} \lesssim \|\phi\|_{\bb^{s,\ell}_{p,r,\mathcal{W}}},
\end{align*}
where in the above we have used the fact $\|\mathcal{W}\|_{\BB^{1,\ell}_\mathcal{W}}
\leq \|\mathcal{W}\|_{\BB^{1,k-1}_\mathcal{W}}<\infty$. 
The induction assumption also guarantees that
\begin{align*}
  \|(T_{\mathcal{W}\cdot\nabla})^\ell \partial_\mathcal{W} \phi\|_{B^s_{p,r}}  \lesssim \|\partial_\mathcal{W} \phi\|_{\BB^{s,\ell}_{p,r,\mathcal{W}}}
  \lesssim \sum_{\lambda=0}^\ell \|\partial_\mathcal{W}^{\lambda+1} \phi\|_{B^s_{p,r}}
  \lesssim  \sum_{\lambda=1}^{\ell+1} \|\partial_\mathcal{W}^\lambda \phi\|_{B^s_{p,r}}
  \lesssim \|\phi\|_{\bb^{s,\ell+1}_{p,r,\mathcal{W}}} .
\end{align*}
Gathering the above estimates lead to
\begin{align}\label{es:phiBBbb-equ}
  \|\phi\|_{\BB^{s,\ell+1}_{p,r,\mathcal{W}}} = \|(T_{\mathcal{W}\cdot\nabla})^{\ell +1} \phi\|_{B^s_{p,r}} + \|\phi\|_{\BB^{s,\ell}_{p,r,\mathcal{W}}}
  \lesssim \|\phi\|_{\bb^{s,\ell+1}_{p,r,\mathcal{W}}}.
\end{align}

For the second part of \eqref{normBB-equiv}, note that
\begin{align*}
  \partial_\mathcal{W}^{\ell+1} \phi 
  = \partial_\mathcal{W}^\ell (T_{\nabla \phi}\cdot \mathcal{W}) + \partial_\mathcal{W}^\ell R(\mathcal{W}\cdot,\nabla \phi) + \partial_\mathcal{W}^\ell (T_{\mathcal{W}\cdot\nabla}) \phi.
\end{align*}
Thanks to \eqref{es:Tw-parap-2}, \eqref{es:Tw-rem-1} and the induction assumption, we find
\begin{align}
  \|\partial_\mathcal{W}^\ell (T_{\nabla \phi}\cdot \mathcal{W})\|_{B^s_{p,r}} &
  \leq \|T_{\nabla \phi}\cdot \mathcal{W}\|_{\bb^{s,\ell}_{p,r,\mathcal{W}}}
  \lesssim  \|T_{\nabla \phi}\cdot \mathcal{W}\|_{\BB^{s,\ell}_{p,r,\mathcal{W}}} \nonumber \\
  & \lesssim \|\phi\|_{\BB^{s,\ell}_{p,r,\mathcal{W}}} \|\mathcal{W}\|_{\BB^{1,\ell}_\mathcal{W}}
  \lesssim \|\phi\|_{\BB^{s,\ell}_{p,r,\mathcal{W}}}, \nonumber
\end{align}
and
\begin{align*}
  \|\partial_\mathcal{W}^\ell R(\mathcal{W}\cdot,\nabla \phi)\|_{B^s_{p,r}} & \lesssim
  \|R(\mathcal{W}\cdot,\nabla \phi)\|_{\bb^{s,\ell}_{p,r,\mathcal{W}}} \lesssim \|R(\mathcal{W}\cdot,
  \nabla \phi)\|_{\BB^{s,\ell}_{p,r,\mathcal{W}}} \\
  & \lesssim \|\mathcal{W}\|_{\BB^{1,\ell}_\mathcal{W}} \|\phi\|_{\BB^{s,\ell}_{p,r,\mathcal{W}}}
  \lesssim \|\phi\|_{\BB^{s,\ell}_{p,r,\mathcal{W}}},
\end{align*}
and
\begin{align*}
  \|\partial_\mathcal{W}^\ell (T_{\mathcal{W}\cdot\nabla})\phi\|_{B^s_{p,r}}
  \lesssim \|(T_{\mathcal{W}\cdot\nabla}) \phi\|_{\bb^{s,\ell}_{p,r,\mathcal{W}}}
  \lesssim \|(T_{\mathcal{W}\cdot\nabla}) \phi\|_{\BB^{s,\ell}_{p,r,\mathcal{W}}} \lesssim \|\phi\|_{\BB^{s,\ell+1}_{p,r,\mathcal{W}}}.
\end{align*}
Collecting the above estimates yields
\begin{align}\label{es:phiBBbb-equ2}
  \|\phi\|_{\bb^{s,\ell+1}_{p,r,\mathcal{W}}} = \|\partial_\mathcal{W}^{\ell+1} \phi\|_{B^s_{p,r}}
  + \|\phi\|_{\bb^{s,\ell}_{p,r,\mathcal{W}}} \lesssim \|\phi\|_{\BB^{s,\ell+1}_{p,r,\mathcal{W}}}.
\end{align}
Hence \eqref{es:phiBBbb-equ}, \eqref{es:phiBBbb-equ2} and the induction method ensure the desired estimate \eqref{normBB-equiv}.

By using \eqref{es:Tw-parap-1b} in place of \eqref{es:Tw-parap-2}, the estimation of \eqref{normBB-equivb} is almost identical to the proof of \eqref{normBB-equiv} with $(s,p,r)$ replaced by $(1,\infty,1)$, and thus we omit the details.

Now we consider \eqref{eq:W-BBbb-equ}. By using \eqref{es:Tw-parap-2c}, the estimate \eqref{es:Tw-para-key} with $\phi=W$
can be improved to hold for every $s>-1$:
\begin{align}
  \|(T_{\mathcal{W}\cdot\nabla})^\ell (T_{\nabla \mathcal{W}} \cdot \mathcal{W})\|_{B^s_{p,r}}
  & \lesssim \|T_{\nabla \mathcal{W}}\cdot \mathcal{W}\|_{\BB^{s,\ell}_{p,r,\mathcal{W}}}
  \lesssim \|\nabla \mathcal{W}\|_{\BB^{0,\ell}_\mathcal{W}} \|\mathcal{W}\|_{\BB^{s,\ell}_{p,r,W}}
  \lesssim \|\mathcal{W}\|_{\BB^{s,\ell}_{p,r,\mathcal{W}}} , \nonumber
\end{align}
thus along the same lines as proving \eqref{es:phiBBbb-equ}, we can easily verify \eqref{eq:W-BBbb-equ}.

We then prove \eqref{eq:BBbb-equiv1} by the induction method. Clearly it holds for $\ell=0$.
Suppose that \eqref{eq:BBbb-equiv1} holds for $\ell\in \{0,1,\cdots,k-1\}$, we intend to prove it for the $(\ell+1)$-case.
In view of \eqref{decom:Tw-el-phi}, and by applying \eqref{es:Tw-parap-2c}, \eqref{es:Tw-rem-1}, \eqref{normBB-equivb}-\eqref{eq:W-BBbb-equ} and the induction assumption,
we infer that for every $s\geq 1$,
\begin{align}
  \|(T_{\mathcal{W}\cdot\nabla})^\ell (T_{\nabla \phi}\cdot \mathcal{W})\|_{B^s_{p,r}} \leq \|T_{\nabla \phi}\cdot \mathcal{W}\|_{\BB^{s,\ell}_{p,r,\mathcal{W}}}
  \lesssim  \|\phi\|_{\BB^{1,\ell}_\mathcal{W}} \|\mathcal{W}\|_{\BB^{s,\ell}_{p,r,\mathcal{W}}} \lesssim \|\phi\|_{\bb^{1,\ell}_\mathcal{W}} \|\mathcal{W}\|_{\bb^{s,\ell}_{p,r,\mathcal{W}}} , \nonumber
\end{align}
and
\begin{align}
  \|(T_{\mathcal{W}\cdot\nabla})^\ell R(\mathcal{W}\cdot,\nabla \phi)\|_{B^s_{p,r}} &
  \lesssim \|R(\mathcal{W}\cdot,\nabla \phi)\|_{\BB^{s,\ell}_{p,r,\mathcal{W}}}
  \lesssim \|\mathcal{W}\|_{\BB^{1,\ell}_\mathcal{W}} \|\phi\|_{\BB^{s,\ell}_{p,r,\mathcal{W}}} \nonumber \\
  & \lesssim \|\phi\|_{\BB^{s,\ell}_{p,r,\mathcal{W}}} \lesssim \|\phi\|_{\bb^{s,\ell}_{p,r,\mathcal{W}}}
  +  \|\phi\|_{\bb^{1,\ell}_\mathcal{W}} \|\mathcal{W}\|_{\bb^{s,\ell}_{p,r,\mathcal{W}}}, \nonumber
\end{align}
and
\begin{align}
  \|(T_{\mathcal{W}\cdot\nabla})^\ell \partial_\mathcal{W} \phi\|_{B^s_{p,r}}
  \leq \|\partial_\mathcal{W} \phi\|_{\BB^{s,\ell}_{p,r,\mathcal{W}}}
  & \lesssim \|\partial_\mathcal{W} \phi\|_{\bb^{s,\ell}_{p,r,\mathcal{W}}}
  + \|\partial_\mathcal{W} \phi\|_{\bb^{1,\ell}_\mathcal{W}} \|\mathcal{W}\|_{\bb^{s,\ell}_{p,r,\mathcal{W}}} \nonumber \\
  & \lesssim \|\phi\|_{\bb^{s,\ell+1}_{p,r,\mathcal{W}}} + \|\phi\|_{\bb^{1,\ell+1}_\mathcal{W}}
  \|\mathcal{W}\|_{\bb^{s,\ell}_{p,r,W}}. \nonumber
\end{align}
Gathering the above estimates gives
\begin{align}
  \|\phi\|_{\BB^{s,\ell+1}_{p,r,\mathcal{W}}} = \|(T_{\mathcal{W}\cdot\nabla})^{\ell +1} \phi\|_{B^s_{p,r}}
  + \|\phi\|_{\BB^{s,\ell}_{p,r,\mathcal{W}}}
  \lesssim \|\phi\|_{\bb^{s,\ell+1}_{p,r,\mathcal{W}}} + \|\phi\|_{\bb^{1,\ell+1}_\mathcal{W}}
  \|\mathcal{W}\|_{\bb^{s,\ell+1}_{p,r,\mathcal{W}}}, \nonumber
\end{align}
which corresponds to \eqref{eq:BBbb-equiv1} at the $(\ell+1)$-case, and thus the desired estimate \eqref{eq:BBbb-equiv1} is proved.

\subsection{Proof of Lemma \ref{lem:Tw-es-k1}}\label{subsec:lem5.3}

  (1) Recalling that the quantities $R_q^i$ ($i=1,2,4,5$) are given by \eqref{eq:Ra123}-\eqref{eq:Rq45},
and noting that
\begin{align}\label{eq:fact2}
  &\quad \sum_{-1\leq q_1\leq q + N_0} 2^{q_1}\| S_{q_1-1}\mathcal{W} - S_{q-1} \mathcal{W} \|_{L^\infty} \nonumber  \\
  &\leq 2^{-1}\|S_{q-1}\mathcal{W}\|_{L^\infty}
  + \sum_{q_1=0}^{q-1} 2^{q_1} \sum_{q_2=q_1}^{q-1} \|\Delta_{q_2}\mathcal{W}\|_{L^\infty} + \sum_{q_1=q+1}^{q+N_0} 2^{q_1} \sum_{q_2=q}^{q_1-1} \|\Delta_{q_2}\mathcal{W}\|_{L^\infty} \nonumber \\
  & \leq C_0 \| \mathcal{W}\|_{W^{1,\infty}}
  + \sum_{q_2=0}^{q-1} \Big(\sum_{q_1= 0}^{q_2} 2^{q_1}\Big) \|\Delta_{q_2}\mathcal{W}\|_{L^\infty}
  \leq C  \|\mathcal{W}\|_{B^1_{\infty,1}},
\end{align}
it is obvious to see that
\begin{align}\label{es:Rq125}
  \|R_q^1\|_{L^p} + \|R_q^2\|_{L^p} + \|R_q^5\|_{L^p} \leq C  \|\mathcal{W}\|_{B^1_{\infty,1}} \min_{1\leq i\leq m}
  \Big( \|\alpha_i\|_{L^p} \prod_{1\leq j\neq i \leq m} \|\alpha_j\|_{L^\infty} \Big),
\end{align}
\begin{align}\label{es:Rq4}
  |R_q^4| \leq C \sum_{i=1}^m \min\bigg\{ \|(T_{\mathcal{W}\cdot\nabla})\alpha_i\|_{L^p} \prod_{j\neq i} \|\alpha_j\|_{L^\infty}, \|(T_{\mathcal{W}\cdot\nabla})\alpha_i\|_{L^\infty} \min_{1\leq j\neq i\leq m} \Big(\|\alpha_j\|_{L^p} \prod_{l\neq i,j} \|\alpha_l\|_{L^\infty} \Big) \bigg\}.
\end{align}
Combining \eqref{es:Rq125} with \eqref{es:Rq4} leads to the desired estimate \eqref{eq:Tw-Rq-2}.


(2) Thanks to \eqref{eq:comm-exp2} and \eqref{eq:comm-exp3}, we get that for every $s\in \R$,
\begin{align*}
  & \quad \big\|\{ 2^{q s}\|[\Delta_q, T_{\mathcal{W}\cdot\nabla}] \phi\|_{L^p}\}_{q\geq -1} \big\|_{\ell^r}
  + \big\| \{ 2^{q (s-1)}\|[\nabla\Delta_q, T_{\mathcal{W}\cdot\nabla}] \phi\|_{L^p} \}_{q\geq -1} \big\|_{\ell^r} \\
  & \leq C \Big\|  2^{q s} \sum_{q_1\in \N,q_1\sim q} \|\nabla \mathcal{W}\|_{L^\infty} \|\Delta_{q_1}\phi\|_{L^p} \Big\|_{\ell^r (q\geq -1)}
  \leq C \|\mathcal{W}\|_{W^{1,\infty}} \|\phi\|_{B^s_{p,r}},
\end{align*}
and for every $s < 0$ we find 
\begin{align*}
  & \quad\big\|\{ 2^{qs} \| [S_{q-1}, T_{\mathcal{W}\cdot\nabla}] \phi \|_{L^p}\}_{q\in\N}\big\|_{\ell^r} \leq
  C \Big\| 2^{qs}  \Big( \sum_{q_1\leq q-1} \|[\Delta_{q_1},T_{\mathcal{W}\cdot\nabla}]\phi\|_{L^p} \Big) \Big\|_{\ell^r(q\in\N)} \\
  & \leq C \Big\| \sum_{q_1\leq q-1} 2^{(q-q_1)s} \Big( 2^{q_1 s} (q_1+2)^n \|[\Delta_{q_1}, T_{\mathcal{W}\cdot\nabla}]\phi\|_{L^p} \Big)\Big\|_{\ell^r(q\in \N)}\\
  & \leq C \big\|  2^{q_1 s} \|[\Delta_{q_1}, T_{\mathcal{W}\cdot\nabla}]\phi\|_{L^p}\big\|_{\ell^r(q_1\geq -1)}
  \leq C \|\mathcal{W}\|_{W^{1,\infty}} \|\phi\|_{B^s_{p,r}},
\end{align*}
thus the desired estimates \eqref{eq:Tw-es-key.2}--\eqref{eq:Tw-es-key.22} follow from a direct computation.

Similarly, observing that
\begin{align*}
  \|\Delta_q [\nabla,T_{\mathcal{W}\cdot\nabla}] \phi\|_{L^\infty} \lesssim \sum_{q_1\in \N, |q_1- q|\leq 5}
  \|\Delta_q(\nabla S_{q_1-1} \mathcal{W} \cdot\nabla \Delta_{q_1} \phi)\|_{L^\infty}
  \lesssim \|\mathcal{W}\|_{W^{1,\infty}} \sum_{q_1\in \N, |q_1- q|\leq 5} 2^{q_1} \|\Delta_{q_1}\phi\|_{L^\infty} ,
\end{align*}
and
\begin{align*}
  2^q \| [\Delta_q, T_{\mathcal{W}\cdot\nabla}] \phi \|_{L^p} \lesssim 2^q \sum_{q_1\in \N, |q_1- q|\leq 5} 2^{-q_1}
  \|\mathcal{W}\|_{W^{1,\infty}} \|\nabla \Delta_{q_1} \phi\|_{L^p}
  \lesssim \|\mathcal{W}\|_{W^{1,\infty}} \sum_{q_1\in \N, |q_1 -q|\leq 5} \|\nabla \Delta_{q_1} \phi\|_{L^p},
\end{align*}
we can obtain \eqref{eq:Tw-es-key2.2}--\eqref{eq:Tw-es-key3.2} as desired.

(3)
By using \eqref{eq:Tw-Rq-2}--\eqref{eq:Tw-es-key.22} we deduce that
\begin{align*}
  & \quad\; 2^{qs} \|\Delta_q (T_{\mathcal{W}\cdot\nabla}) (T_{u\cdot} \nabla \phi)\|_{L^p} \\
  & \leq C 2^{qs} \sum_{q_1\in \N,|q_1-q|\leq N_1} \|\Delta_q (T_{\mathcal{W}\cdot\nabla}) \big(S_{q_1-1}u \Delta_{q_1} \nabla \phi\big)\|_{L^p} \\
  & \leq C 2^{qs} \sum_{q_1\in \N,|q_1-q|\leq N_1} \min\Big\{\|(T_{\mathcal{W}\cdot\nabla}) S_{q_1-1}u\|_{L^p}
  \|\Delta_{q_1} \nabla \phi\|_{L^\infty},
  \|(T_{\mathcal{W}\cdot\nabla}) S_{q_1-1}u\|_{L^\infty} \|\Delta_{q_1} \nabla \phi\|_{L^p}\Big\}  \\
  & \quad + C 2^{qs} \sum_{q_1\in \N,|q_1-q|\leq N_1} \min\Big\{
  \|S_{q_1-1}u\|_{L^p} \|(T_{\mathcal{W}\cdot\nabla})\Delta_{q_1}\nabla \phi\|_{L^\infty} ,
  \|S_{q_1-1}u\|_{L^\infty} \|(T_{\mathcal{W}\cdot\nabla})\Delta_{q_1}\nabla \phi\|_{L^p} \Big\}  \\
  & \quad + C 2^{qs} \sum_{q_1\in \N,|q_1-q|\leq N_1} \|\mathcal{W}\|_{B^1_{\infty,1}}
  \min\Big\{ \|S_{q_1-1}u\|_{L^p} \|\Delta_{q_1}\nabla \phi\|_{L^\infty},
  \|S_{q_1-1}u\|_{L^\infty} \|\Delta_{q_1}\nabla \phi\|_{L^p} \Big\} \\
  & \leq C c_q \min \Big\{ \big\|2^{q_1 s} \|(T_{\mathcal{W}\cdot\nabla})S_{q_1-1}u\|_{L^p}\big\|_{\ell^r(q_1\in \N)}
  \|\nabla \phi\|_{L^\infty} ,
  \sup_{q_1\in \N} \|(T_{\mathcal{W}\cdot\nabla})S_{q_1-1}u\|_{L^\infty} \|\nabla \phi\|_{B^s_{p,r}}  \Big\} \\
  & \quad + C c_q \min \Big\{ \|u\|_{B^s_{p,r}} \sum_{q_1\in \N} \|(T_{\mathcal{W}\cdot\nabla})\Delta_{q_1}\nabla \phi\|_{L^\infty} ,
  \|u\|_{B^0_{\infty,1}}  \big\| 2^{q_1 s }  \|(T_{\mathcal{W}\cdot\nabla})\Delta_{q_1} \nabla \phi\|_{L^p} \big\|_{\ell^r(q_1\in \N)} \Big\} \\
  & \quad + C c_q \|\mathcal{W}\|_{B^1_{\infty,1}} \min \big\{\|u\|_{B^s_{p,r}} \|\nabla \phi\|_{B^0_{\infty,1}} ,
  \|u\|_{B^0_{\infty,1}} \|\nabla \phi\|_{B^s_{p,r}} \big\} \\
  & \leq C c_q \big(B_1(s) + B_2(s) + B_3(s)\big),
\end{align*}
where $\{c_q\}_{q\geq -1}$ is such that $\|c_q\|_{\ell^r}=  1$ and $B_1, B_2, B_3$ are given by \eqref{B1}--\eqref{B3}.
Similarly, we can get the same estimate about $2^{qs} \|\Delta_q (T_{\mathcal{W}\cdot\nabla})(T_{\nabla \phi}\cdot u)\|_{L^p}$.

(4) Notice that
\begin{align*}
  2^{qs} \|\Delta_q(T_{\mathcal{W}\cdot\nabla}) R(u\cdot,\nabla \phi)\|_{L^p}
  & \leq C 2^{qs} \sum_{q_1\geq \max\{3,q-5\}} \|\Delta_q(T_{\mathcal{W}\cdot\nabla})\nabla\cdot \big(\Delta_{q_1}u\, \widetilde\Delta_{q_1} \phi\big)\|_{L^p} \\
  & \quad + C 2^{qs} 1_{\{q\leq 7\}} \sum_{-1\leq q_1 \leq 2} \|\Delta_q(T_{\mathcal{W}\cdot\nabla})\big(\Delta_{q_1}u\cdot \widetilde\Delta_{q_1} \nabla \phi\big)\|_{L^p}
  = : I_{1,q} + I_{2,q} .
\end{align*}
For $I_{1,q}$, thanks to \eqref{eq:Tw-es-key2.2} and \eqref{eq:Tw-es-key3.2}
we infer that for every $s>-1$,
\begin{align*}
  I_{1,q} & \lesssim 2^{q(s+1)}  \sum_{q_1\geq \max\{3,q-5\}} \|\Delta_q (T_{\mathcal{W}\cdot\nabla})\big(\Delta_{q_1}u\, \widetilde\Delta_{q_1} \phi\big)\|_{L^p}
  +  \|\mathcal{W}\|_{W^{1,\infty}} 2^{q(s+1)}   \sum_{q_1\geq \max\{3,q-5\}} \|\Delta_{q_1}u\, \widetilde{\Delta}_{q_1} \phi\|_{L^p} \\
  & \lesssim 2^{q(s+1)}  \sum_{q_1\geq \max\{3,q-5\}} \min \Big\{\|(T_{\mathcal{W}\cdot\nabla})\Delta_{q_1}u\|_{L^p} \|\widetilde\Delta_{q_1}\phi\|_{L^\infty},
  \|(T_{\mathcal{W}\cdot\nabla})\Delta_{q_1}u\|_{L^\infty} \|\widetilde\Delta_{q_1}\phi\|_{L^p}\Big\} \\
  & \quad + 2^{q(s+1)}  \sum_{q_1\geq \max\{3,q-5\}} \min \Big\{
  \|\Delta_{q_1}u\|_{L^p} \|(T_{\mathcal{W}\cdot\nabla})\widetilde\Delta_{q_1}\phi\|_{L^\infty},
  \|\Delta_{q_1}u\|_{L^\infty} \|(T_{\mathcal{W}\cdot\nabla})\widetilde\Delta_{q_1}\phi\|_{L^p}\Big\} \\
  & \quad +  \|\mathcal{W}\|_{B^1_{\infty,1}} 2^{q(s+1)} \sum_{q_1\geq \max\{3,q-5\}} \min\Big\{ \|\Delta_{q_1}u\|_{L^p} \|\widetilde\Delta_{q_1}\phi\|_{L^\infty} ,
  \|\Delta_{q_1}u\|_{L^\infty} \|\widetilde\Delta_{q_1}\phi\|_{L^p}\Big\}\\
  & \lesssim  \sum_{q_1\geq \max\{3,q-5\}} 2^{(q-q_1)(s+1)}  \min\Big\{ 2^{q_1 s} \|(T_{\mathcal{W}\cdot\nabla})\Delta_{q_1}u\|_{L^p} \|\widetilde\Delta_{q_1}\nabla \phi\|_{L^\infty},
  \|(T_{\mathcal{W}\cdot\nabla})\Delta_{q_1}u\|_{L^\infty} 2^{q_1 s} \|\widetilde\Delta_{q_1}\nabla \phi\|_{L^p}  \Big\} \\
  & \quad +  \sum_{q_1\geq \max\{3,q-5\}} 2^{(q-q_1)(s+1)}  2^{q_1 s}\cdot \\
  & \quad\quad \cdot \min\bigg\{ \|\Delta_{q_1} u\|_{L^p} \Big(\sum_{|q_2-q_1|\leq 7}\|(T_{\mathcal{W}\cdot\nabla})\Delta_{q_2}\nabla \phi\|_{L^\infty}  \Big),
  \|\Delta_{q_1} u\|_{L^\infty}  \Big(\sum_{|q_2-q_1|\leq 7}\|(T_{\mathcal{W}\cdot\nabla})\Delta_{q_2}\nabla \phi\|_{L^p}  \Big) \bigg\}\\
  & \quad +  \|\mathcal{W}\|_{B^1_{\infty,1}} \sum_{q_1\geq \max\{3,q-5\}} 2^{(q-q_1)(s+1)} \cdot \\
  & \qquad \cdot \min\bigg\{ 2^{q_1 s} \|\Delta_{q_1} u\|_{L^p} \Big(\sum_{|q_2-q_1|\leq 7} \|\Delta_{q_2}\nabla \phi\|_{L^\infty}\Big),
  \|\Delta_{q_1} u\|_{L^\infty} 2^{q_1 s} \Big(\sum_{|q_2-q_1|\leq 7} \|\Delta_{q_2}\nabla \phi\|_{L^p}\Big)\bigg\} \\
  & \lesssim c_q \min \Big\{\big\| 2^{q_1s} \|(T_{\mathcal{W}\cdot\nabla})\Delta_{q_1}u\|_{L^p} \big\|_{\ell^r } \|\nabla \phi\|_{B^0_{\infty,1}},
  \sum_{q_1\geq -1} \|(T_{\mathcal{W}\cdot\nabla}) \Delta_{q_1}u\|_{L^\infty} \|\nabla \phi\|_{B^s_{p,r}}\Big\} \\
  & \quad +  c_q \min \Big\{\|u\|_{B^0_{\infty,1}}  \big\| 2^{q_2s}
  \|(T_{\mathcal{W}\cdot\nabla})\Delta_{q_2}\nabla \phi\|_{L^p} \big\|_{\ell^r},
  \|u\|_{B^s_{p,r}} \sum_{q_2\geq -1} \|(T_{\mathcal{W}\cdot\nabla}) \Delta_{q_2}\nabla \phi\|_{L^\infty}\Big\} \\
  & \quad +  c_q  \|\mathcal{W}\|_{B^1_{\infty,1}} \min \Big\{\|u\|_{B^s_{p,r}} \|\nabla \phi\|_{B^0_{\infty,1}},
  \|u\|_{B^0_{\infty,1}} \|\nabla \phi\|_{B^s_{p,r}}\Big\}  \\
  & \leq C c_q \big(B_1(s) + B_2(s) + B_3(s)\big),
\end{align*}
where $\{c_q\}_{q\geq -1}$ is such that $\|c_q\|_{\ell^r}=  1$.
For $I_{2,q}$, by applying Bernstein's inequality, one gets
\begin{align*}
  \| I_{2,q}\|_{\ell^r}
  & \leq C \|\mathcal{W}\|_{B^1_{\infty,1}} \sup_{-1\leq q_1 \leq 2} 2^{q_1s} \min \Big\{\|\Delta_{q_1} u \|_{L^p} \|\widetilde\Delta_{q_1} \nabla \phi\|_{L^\infty},
  \|\Delta_{q_1} u \|_{L^\infty} \|\widetilde\Delta_{q_1} \nabla \phi\|_{L^p} \Big\} \\
  & \leq C \|\mathcal{W}\|_{B^1_{\infty,1}} \min\Big\{\|u\|_{B^s_{p,r}} \|\nabla \phi\|_{B^0_{\infty,1}}, \|u\|_{B^0_{\infty,1}} \|\nabla \phi\|_{B^s_{p,r}} \Big\}.
\end{align*}
Hence, gathering the above estimates on $I_{1,q}$ and $I_{2,q}$ leads to the desired estimate \eqref{es:Tw-rem-1-2}.

(5)  By arguing as \eqref{eq:Tu-phi} and \eqref{eq:Ru-phi} below, we easily find that for every $s\in (-1,1)$,
\begin{equation}\label{par-W-fact}
\begin{split}
  \|\partial_\mathcal{W} \phi - T_{\mathcal{W}\cdot\nabla} \phi\|_{B^s_{p,r}}
  \leq \|T_{\nabla \phi}\cdot \mathcal{W}\|_{B^s_{p,r}} + \|R(\mathcal{W}\cdot,\nabla \phi)\|_{B^s_{p,r}}
  \leq C \|\mathcal{W}\|_{W^{1,\infty}} \|\phi\|_{B^s_{p,r}}.
\end{split}
\end{equation}

\section{Appendix: proof of Lemmas \ref{lem:prod-es0}}\label{sec:append}

\begin{proof}[Proof of Lemma \ref{lem:prod-es0}]
Bony's decomposition gives
$u\cdot\nabla \phi = T_{u\cdot}\nabla \phi + T_{\nabla \phi} \cdot u + R(u\cdot,\nabla\phi)$.
For the paraproduct terms, by using the spectrum support property and the inequality that for every $\epsilon>0$,
\begin{equation}\label{eq:fact}
\begin{split}
  \big\|2^{-j\epsilon} \|S_{j-1} f\|_{L^p} \big\|_{\ell^r}
  \leq C \Big\|\sum_{j'\leq j-1} 2^{-(j-j')\epsilon} 2^{-\epsilon j'}  \|\Delta_{j'} f\|_{L^p} \Big\|_{\ell^r(j\geq -1)}
  \leq C \|f\|_{B^{-\epsilon}_{p,r}} ,
\end{split}
\end{equation}
we have
\begin{align}\label{eq:Tu-phi}
  & \quad 2^{-q\epsilon} \|\Delta_q (T_u\cdot\nabla \phi)\|_{L^p}
  \leq 2^{-q\epsilon} \sum_{ j\in \N,|j-q|\leq 4} \|\Delta_q (S_{j-1} u \cdot \nabla \Delta_j \phi)\|_{L^p} \nonumber \\
  & \lesssim  \sum_{ j\in\N,|j-q|\leq 4} 2^{-j\epsilon}  \min \big\{\|S_{j-1} u\|_{L^p}  \|\nabla \Delta_j \phi\|_{L^\infty}, \|S_{j-1} u\|_{L^\infty}  \|\nabla \Delta_j \phi\|_{L^p} \big\} \nonumber \\
  & \lesssim  c_q \min\Big\{\|u\|_{B^{-\epsilon}_{p,r}} \|\nabla\phi\|_{L^\infty}, \|u\|_{L^\infty} \|\nabla\phi\|_{B^{-\epsilon}_{p,r}} \Big\} ,
\end{align}
and
\begin{align*}
  & \quad 2^{-q\epsilon} \|\Delta_q (T_{\nabla \phi}\cdot u)\|_{L^\infty}\leq 2^{-q\epsilon} \sum_{j\in \N,|j-q|\leq 4} \|\Delta_q (\Delta_j u \cdot \nabla S_{j-1} \phi)\|_{L^\infty} \\
  & \lesssim  \sum_{j\in\N,|j-q|\leq 4} 2^{-j\epsilon}
  \min\big\{ \|\Delta_j u\|_{L^p} \|\nabla S_{j-1} \phi\|_{L^\infty}, \|\Delta_j u\|_{L^\infty} \|\nabla S_{j-1} \phi\|_{L^p} \big\} \\
  & \lesssim  c_q \min\Big\{\|u\|_{B^{-\epsilon}_{p,r}} \|\nabla\phi\|_{L^\infty}, \|u\|_{L^\infty} \|\nabla\phi\|_{B^{-\epsilon}_{p,r}} \Big\} ,
\end{align*}
where $\{c_q\}_{q\geq -1}$ is such that $\| c_q\|_{\ell^r} = 1$. While for the reminder term, thanks to the divergence-free property of $u$, we get that for every $\epsilon<1$,
\begin{align}\label{eq:Ru-phi}
  & \quad\;  2^{-q\epsilon} \|\Delta_q  R(u\cdot,\nabla\phi)\|_{L^p} = 2^{-q\epsilon}  \|\Delta_q \divg R(u,\phi)\|_{L^p} \nonumber \\
  & \lesssim 2^{q(1-\epsilon)}  \sum_{ j\geq 2,j\geq q-3} \|\Delta_q(\Delta_j u \widetilde{\Delta}_j \phi)\|_{L^p}
  + 2^{-q\epsilon}  \sum_{ j\leq 2,j\geq q-3} \|\Delta_q(\Delta_j u\cdot\nabla \widetilde{\Delta}_j \phi)\|_{L^p}  \nonumber \\
  & \lesssim \sum_{j\geq q-3,j\geq 2} 2^{(q-j)(1-\epsilon)}  2^{-j\epsilon} \min\big\{\|\Delta_j u\|_{L^p} \|\nabla\widetilde{\Delta}_j \phi\|_{L^\infty},
  \|\Delta_j u\|_{L^\infty} \|\nabla\widetilde{\Delta}_j \phi\|_{L^p}\big\} \nonumber \\
  & \quad + \sum_{j\geq q-3,j\leq 2} 2^{-j \epsilon} \min\big\{\|\Delta_j u \|_{L^p} \|\nabla \widetilde{\Delta}_j \phi\|_{L^\infty},
  \|\Delta_j u \|_{L^\infty} \|\nabla \widetilde{\Delta}_j \phi\|_{L^p}\big\} \nonumber \\
  & \lesssim c_q \min\Big\{\|u\|_{B^{-\epsilon}_{p,r}} \|\nabla\phi\|_{L^\infty}, \|u\|_{L^\infty} \|\nabla\phi\|_{B^{-\epsilon}_{p,r}} \Big\}.
\end{align}
Hence gathering the above estimates yields \eqref{eq:prod-es}.

\end{proof}

\vskip0.2cm

\textbf{Acknowledgements.}
D. Chae was supported partially by the grant of NRF(No. 2021R1A2C1003234). Q. Miao was partially supported by National Natural Science Foundation of China (No. 12001041).
L. Xue was partially supported by National Key Research and Development Program of China (No. 2020YFA0712900) and National Natural Science Foundation of China (No. 11771043).


\begin{thebibliography}{60}

\bibitem{AH07} H. Abidi and T. Hmidi, On the global well-posedness for Boussinesq system.
 J. Diff. Equ., \textbf{233} (2007), no. 1, 199--220.



\bibitem{BCD11} H. Bahouri, J.-Y. Chemin and R. Danchin, \textit{Fourier Analysis and Nonlinear Partial Differential Equations}. Grundlehren der
mathematischen Wissenschaften \textbf{343}, Springer-Verlag, (2011).

\bibitem{BerC} A. L. Bertozzi and P. Constantin, Global regularity for vortex patches.
Comm. Math. Phys., \textbf{152} (1993), 19--28.

\bibitem{CaoW13} C. Cao and J. Wu, Global regularity for the two-dimensional anisotropic Boussinesq equations with vertical dissipation.
Arch. Ration. Mech. Anal., \textbf{208} (2013), no. 3, 985--1004.


\bibitem{CCL19} A. Castro, D. C\'ordoba and D. Lear, On the asymptotic stability of stratified solutions for the 2D Boussinesq equations with a velocity damping term.
Math. Models Methods Appl. Sci., \textbf{29} (2019), no. 7, 1227--1277.

\bibitem{CCFG13} A. Castro, D. C\'ordoba, C. Fefferman and F. Gancedo, Breakdown of smoothness for the Muskat problem.
Arch. Ration. Mech. Anal., \textbf{208} (2013), no. 3, 805--909.

\bibitem{CCFGG13} A. Castro, D. C\'ordoba, C. Fefferman, F. Gancedo and J. G\'omez-Serrano, Finite time singularities for the free boundary incompressible Euler equations.
Ann. Math. (2), \textbf{178} (2013), no. 3, 1061--1134.

\bibitem{CCFGG19} A. Castro, D. C\'ordoba, C. Fefferman, F. Gancedo and J. G\'omez-Serrano, Splash singularities for the free boundary Navier-Stokes equations.
Ann. PDE,  \textbf{5} (2019), no. 1, Paper No. 12, 117 pp.

\bibitem{CCFG16} A. Castro, D. C\'ordoba, C. Fefferman and F. Gancedo, Splash singularities for the one-phase Muskat problem in stable regimes.
Arch. Ration. Mech. Anal., \textbf{222} (2016), no. 1, 213--243.

\bibitem{Cha06} D. Chae, Global regularity for the 2D Boussinesq equations with partial viscosity terms.
Adv. Math., \textbf{203} (2006), no. 2, 497--513.

\bibitem{CW} D. Chae and J. Wu, The 2D Boussinesq equations with logarithmically supercritical velocities, Adv. Math., \textbf{230},  (2012), 1618-1645.


\bibitem{Chem88} J.-Y. Chemin, Calcul symbolique et propagation des singularit\'es pour les \'equations aux d\'eriv\'ees partielles non semilin\'eaires.
Duke Math. J., \textbf{56} (1988), 431--469.

\bibitem{Chem91} J.-Y. Chemin, Sur le mouvement des particules d'un fluide parfait incompressible bidimensionnel.
Invent. Math., \textbf{103} (1991), 599--629.

\bibitem{CKY15} K. Choi, A. Kiselev and Y. Yao, Finite time blow up for a 1D model of 2D Boussinesq system.
Comm. Math. Phys., \textbf{334} (2015), no. 3, 1667--1679.

\bibitem{CHKLSY} K. Choi, T.Y. Hou, A. Kiselev, G. Luo, V. Sverak and Y. Yao,
On the finite time blowup of a one-dimensional model for the three-dimensional axisymmetric Euler equations.
Comm. Pure Appl. Math., \textbf{70} (2017), 2218--2243.

\bibitem{ConD99} P. Constantin and C. R. Doering, Infinite Prandtl number convection.
J. Stat. Phys., \textbf{94} (1999), 159--172.

\bibitem{CFMR} D. C\'ordoba, M.A. Fontelos, A.M. Mancho and J.L. Rodrigo,
Evidence of singularities for a family of contour dynamics equations.
Proc. Nat. Acad. Sci. USA, \textbf{102} (2005), no. 17, 5949--5952.


\bibitem{CouS14} D. Coutand and S. Shkoller, On the finite-time splash and splat singularities for the 3-D free-surface Euler equations.
Comm. Math. Phys., \textbf{325} (2014), no. 1, 143--183.

\bibitem{CouS19} D. Coutand and S. Shkoller, On the splash singularity for the free-surface of a Navier-Stokes fluid.
Ann. Inst. H. Poincar\'e Anal. Non Lin\'eaire, \textbf{36} (2019), no. 2, 475--503.

\bibitem{Dan97} R. Danchin, Poches de tourbillon visqueuses.
J. Math. Pures Appl., (9) 76 (1997), no. 7, 609--647.



%

%

\bibitem{DanP09} R. Danchin and M. Paicu, Les th\'eor\'emes de Leray et de Fujita-Kato pour le systme de Boussinesq partiellement visqueux.
Bull. Soc. Math. France, \textbf{136} (2008), no. 2, 261--309.


\bibitem{DanZ17} R. Danchin and X. Zhang, Global persistence of geometrical structures for the Boussinesq equation with no diffusion.
Comm. Partial Differ. Equ., \textbf{42} (2017), no. 1, 68--99.

\bibitem{DanZ17b} R. Danchin and X. Zhang, On the persistence of H\"older regular patches of density for the inhomogeneous Navier-Stokes equations.
J. Ec. Polytech. Math., \textbf{4} (2017), 781--811.
%
%

\bibitem{ES94} W. E and C.-W. Shu, Small-scale structures in Boussinesq convection.
Phys. Fluids, \textbf{6} (1994), no. 1, 49--58.

\bibitem{ElgJ20} T.M. Elgindi and I.-J. Jeong,
Finite-time singularity formation for strong solutions to the Boussinesq system.
Ann. PDE, \textbf{6} (2020), no. 1, paper no. 5, 50 pp.

\bibitem{ElgW16} T.M. Elgindi and K. Widmayer, Sharp decay estimates for an anisotropic linear semigroup and applications to the SQG and inviscid Boussinesq systems.
SIAM J. Math. Anal., \textbf{47} (2016), no. 6, 4672--4684.

\bibitem{FIL16} C. Fefferman, A.D. Ionescu and V. Lie, On the absence of splash singularities in the case of two-fluid interfaces.
Duke Math. J., \textbf{165} (2016), no. 3, 417--462.


%
%
%
\bibitem{FN09} E. Feireisl, A. Novotny, The Oberbeck-Boussinesq approximation as a singular limit of the full Navier-Stokes-Fourier system.
J. Math. Fluid Mech. 11 (2009), no. 2, 274--302.

%


\bibitem{GGJ17} F. Gancedo and E. Garc\'ia-Ju\'arez,
Global regularity for 2D Boussinesq temperature patches with no diffusion.
Ann. PDE, \textbf{3} (2017), no. 2, Art. 14, 34pp.

\bibitem{GGJ18} F. Gancedo and E. Garc\'ia-Ju\'arez,
Global regularity of 2D density patches for inhomogeneous Navier-Stokes.
Arch. Ration. Mech. Anal., \textbf{229} (2018), no. 1, 339--360.

\bibitem{GGJ20} F. Gancedo and E. Garc\'ia-Ju\'arez,
Regularity results for viscous 3D Boussinesq temperature fronts.
Comm. Math. Phys., \textbf{376} (2020), 1705--1736.

\bibitem{GP21} F. Gancedo and P. Neel,
On the local existence and blow-up for generalized SQG patches.
Ann. PDE, \textbf{7} (2021), no. 1, Paper No. 4, 63 pp.

\bibitem{GanS14} F. Gancedo and R. Strain, Absence of splash singularities for surface quasi-geostrophic sharp fronts and the Muskat problem.
Proc. Nat. Acad. Sci. USA, \textbf{111} (2014), no. 2, 635--639.

\bibitem{Gil82} A. Gill, \emph{Atmosphere-Ocean Dynamics}. International Geophysics Series, vol. 30. Academic Press, New York (1982).
%
%
%

\bibitem{HK07} T. Hmidi and S. Keraani, On the global well-posedness of the two-dimensional Boussinesq system with a zero diffusivity.
Adv. Diff. Equ., \textbf{12} (2007), no. 4, 461--480.

\bibitem{HKR10} T. Hmidi, S. Keraani and F. Rousset, Global well-posedness for a Boussinesq-Navier-Stokes system with critical dissipation.
J. Differential Equations, \textbf{249} (2010), 2147--2174.

\bibitem{HKR11} T. Hmidi, S. Keraani and F. Rousset, Global well-posedness for Euler-Boussinesq system with critical dissipation.
Comm. Par. Diff. Equ., \textbf{36} (2011) no. 3, 420--445.


\bibitem{HouL05} T. Hou and C. Li, Global well-posedness of the viscous Boussinesq equations.
Discrete Contin. Dyn. Syst., \textbf{12} (2005), no. 1, 1--12.

\bibitem{HKZ15} W. Hu, I. Kukavica and M. Ziane, Persistence of regularity for the viscous Boussinesq equations with zero diffusivity.
Asympt. Anal., \textbf{91} (2015), 111--124.

\bibitem{KX20} C. Khor and X. Xu, Temperature patches for the subcritical Boussinesq-Navier-Stokes system with no diffusion.
ArXiv:2007.14578v1 [math.AP].

\bibitem{KRYZ16} A. Kiselev, L. Ryzhik, Y. Yao and A. Zlato$\rm\check{s}$,
Finite time singularity for the modified SQG patch equation.
Ann. Math., \textbf{184} (2016), no. 3, 909--948.

\bibitem{KT18} A. Kiselev and C. Tan, Finite time blow up in the hyperbolic Boussinesq system.
Adv. Math., \textbf{325} (2018), 34--55.
%
%

\bibitem{LLT13} A. Larios, T. Lunasin and E. Titi, Global well-posedness for the 2D Boussinesq system with anistropic viscosity and without heat diffusion.
J. Diff. Equ., \textbf{255} (2013), no. 9, 2636--2654.

\bibitem{LiT16} J. Li and E. Titi, Global well-posedness of the 2D Boussinesq equations with vertical dissipation.
Arch. Ration. Mech. Anal., \textbf{220} (2016), no. 3, 983--1001.

\bibitem{LZ16} X. Liao and P. Zhang, On the global regularity of the two-dimensional density patch for inhomogeneous incompressible viscous flow.
Arch. Rational Mech. Anal., \textbf{220} (2016), 937--981.

\bibitem{LZ19} X. Liao and P. Zhang, Global regularity of 2D density patches for viscous inhomogeneous incompressible flow with general density: low regularity case.
Comm. Pure Appl. Math., \textbf{72} (2019), no. 4, 835--884.


\bibitem{LH14b} G. Luo and T.Y. Hou,
Toward the finite-time blowup of the 3D axisymmetric Euler equations: a numerical investigation.
Multiscale Model. Simul., \textbf{12} (2014), no. 4, 1722--1776.

\bibitem{Maj03} A. J. Majda, \emph{Introduction to PDEs and Waves for the Atmosphere and Ocean}. Courant Lect. Notes Math., vol. 9, AMS/CIMS (2003).

\bibitem{MB02} A. J. Majda and A. L. Bertozzi, \emph{Vorticity and incompressible flow}. Cambridge Texts in Applied Mathematics, 27. Cambridge University Press, Cambridge, (2002).


%

\bibitem{Mof01} H. K. Moffatt, Some remarks on topological fluid mechanics, in: R. L. Ricca (Ed.), An Introduction to the Geometry
and Topology of Fluid Flows, Kluwer Academic, Dordrecht, (2001), pp. 3--10.
%
%

\bibitem{Ped87} J. Pedlosky, \emph{Geophysical Fluid Dynamics}. Springer, New York (1987).


\bibitem{RS96} T. Runst and W. Sickel, \emph{Sobolev spaces of fractional order, Nemytskij operator, and nonlinear partial differential equations}.
de Gruyter Series in Nonlinear Analysis and Applications 3, De Gruyter, (1996).

\bibitem{Ryc99} V. S. Rychkov, On restrictions and extensions of the Besov and Triebel-Lizorkin spaces with respect to Lipschitz domains.
J. Lond. Math. Soc., \textbf{60} (1999), no. 1, 237--257.

\bibitem{Stein} E. M. Stein, \emph{Singular integrals and differentiability properties of functions}.
Princeton Landmarks in Mathematics, Princeton University Press, Princeton (1970).

\bibitem{Sue15} F. Sueur, Viscous profiles of vortex patches.
J. Inst. Math. Jussieu, \textbf{14} no. 1, (2015), 1--68.
%
%



\bibitem{WX12} G. Wu and L. Xue, Global well-posedness for the 2D inviscid B\'enard system with fractional diffusivity and Yudovich's type data.
J. Differential Equations, \textbf{253} (2012), no. 1, 100--125.

\bibitem{Yud03} V. I. Yudovich, Eleven great problems of mathematical hydrodynamics.
Moscow Math. J., \textbf{3} (2003), no. 2, 711--737.


\end{thebibliography}
\end{document}